\documentclass{gtpart}     

\usepackage{amssymb}

\usepackage{graphicx}
\usepackage{ifdraft}
\usepackage{pgfmath}
\usepackage{tikz}
	\usetikzlibrary{calc}
\usepackage{amsthm}
\usepackage{amsmath}
\usepackage{amssymb}
\usepackage{graphicx}
\usepackage{tcolorbox}
\usepackage{subcaption}

\usepackage{xcolor}
\usepackage{thmtools}
\usepackage{thm-restate}
\usepackage[capitalize,noabbrev]{cleveref}

\title{A qualitative description of the horoboundary of the Teichmüller metric}
\author{Aitor Azemar}
\givenname{Aitor}
\surname{Azemar}
\address{School of Mathematics and Statistics, University of Glasgow, University place, Glasgow, G12 8QQ, UNITED KINGDOM}
\email{Aitor.Azemar@Glasgow.ac.uk}
\urladdr{Aitor.Azemar.xyz}

\keyword{Visual compactification, Gardiner--Masur compactification, Busemann points}
\subject{primary}{msc2010}{30F60,32G15}
\subject{secondary}{msc2010}{51F30}

\arxivreference{2108.11698}

\theoremstyle{plain}
\newtheorem{theorem}{Theorem}[section]

\newtheorem{lemma}[theorem]{Lemma}
\newtheorem{corollary}[theorem]{Corollary}
\newtheorem{proposition}[theorem]{Proposition}
\newtheorem{question}[theorem]{Question}

\theoremstyle{definition}
\newtheorem{definition}[theorem]{Definition}

\newcommand{\ext}{\operatorname{Ext}}

\newcommand{\tray}[2]{{R(#1;#2)}}
\newcommand{\N}{\mathbb{N}}
\newcommand{\MF}{\mathcal{MF}}
\newcommand{\PMF}{\mathcal{PMF}}
\renewcommand{\H}{\mathbb{H}}
\newcommand{\bussmap}{B}
\newcommand{\fibermap}{\Pi}
\newcommand{\bigfibermap}{\widetilde{\Pi}}
\newcommand{\hbd}[1]{\partial\overline{#1}^h}
\newcommand{\hbc}[1]{\overline{#1}^{h}}
\newcommand{\vbd}[1]{\partial\overline{#1}^v}
\newcommand{\vbc}[1]{\overline{#1}^{v}}
\newcommand{\Vbc}[1]{\overline{#1}^{V}}
\newcommand{\wals}{\mathcal{W}}
\newcommand{\B}{\mathcal{B}}
\newcommand{\Bclose}{\overline{\mathcal{B}}}
\newcommand{\E}{\mathcal{E}}
\newcommand{\T}{\mathcal{T}}

\newcommand{\Curv}{\mathcal{S}}

\DeclareMathOperator{\MCG}{MCG}

\begin{document}
	
	\begin{abstract}
		Two commonly studied compactifications of Teichmüller spaces of finite type surfaces with respect to the Teichmüller metric are the horofunction and visual compactifications. We show that these two compactifications are related, by proving that the horofunction compactification is finer than the visual compactification. This allows us to use the straightforwardness of the visual compactification to obtain topological properties of the horofunction compactification. Among other things, we show  that Busemann points of Teichmüller metric are not dense within the horoboundary, answering a question by Liu--Su. We also show that the horoboundary of Teichmüller space is path connected, determine for which surfaces the horofunction compactification is isomorphic to the visual one and show that some horocycles diverge in the visual compactification based at some point.
		As an ingredient in one of the proofs we show that extremal length is not $C^{2}$ along some paths that are smooth with respect to the piecewise linear structure on measured foliations.
	\end{abstract}

	\maketitle
	
\section{Introduction}

	The horofunction compactification of a metric space is defined in terms of the metric, so its properties are well aligned for studying the metric properties of the space. For example, all geodesic rays converge to points and isometries of the space can be extended to homeomorphisms of the compactification. This compactification was first introduced by Gromov \cite{Gromov} as a natural, general compactification, based on previous ideas of Busemann. The horofunction compactification has since found many applications; it was used to obtain asymptotic properties of random walks on weakly hyperbolic spaces by Maher--Tiozzo \cite{Tiozzo}, to determine the isometry group of some Hilbert geometries by Lemmens--Walsh \cite{Walsh4} and to obtain properties of quantum metric spaces by Rieffel \cite{Rieffel}. The compactification is obtained by embedding the metric space $X$ into the space $C(X)$ of continuous functions on $X$ via the map $h:X\hookrightarrow C(X)$ defined by
	\begin{equation}\label{eq:hdefinition}
		h(p)(\cdot)=d(p,\cdot)-d(p,b),
	\end{equation}
	where $b\in X$ is an arbitrarily chosen basepoint. As explained, for example, by Walsh \cite[Section 2]{Walsh3}, if the space $X$ is proper then $h$ is and embedding, the closure of $h(X)$ is compact and the \emph{horofunction compactification} of $X$ is defined as the pair $(h,\overline{h(X)})$. By considering two functions equivalent if they differ by a constant one can show that the compactification does not depend on the basepoint $b$. While this compactification has been rather useful, it is sometimes hard to visualize, and there are not that many examples where the horofunction boundary is explicitly known. Some cases where the horofunction compactification is understood include Hadamard manifolds and some of their quotients, by Dal'bo--Peigné--Sambusetti \cite{Dal'bo}, as well as the Heisenberg group with the Carnot--Carath\'{e}odory metric, by Klein--Nicas \cite{Klein}, and Hilbert geometries, by Walsh \cite{Walsh2}.
	
	On the other hand, for a proper, uniquely geodesic, straight metric space $X$ (see \cref{se:metricdefinitions} for definitions) the visual compactification based at some point $b\in X$ is defined by pasting the set of geodesic rays exiting $b$, denoted $D_b$, to the space $X$ in such a way that a sequence $(x_n)\subset X$ converges to some ray $\gamma\in D_b$ if the distance $d(b,x_n)$ goes to infinity as $n\to\infty$, and the geodesic ray between $b$ and $x_n$ converges uniformly on compacts to $\gamma$. See \cref{se:metricdefinitions} for details on the topology of $X\cup D_b$. This compactification may depend on the basepoint $b$, which restricts its usefulness. It can even happen that isometries of $X$ that move the basepoint can not be extended continuously to the compactification, as Kerckhoff showed for Teichmüller spaces \cite{Kerckhoff}. However, the visual compactification usually has a simple geometric interpretation. For example, for a Hadamard manifold, as well as for a Teichmüller space with the Teichmüller metric, this compactification is homeomorphic to a closed ball of the same dimension as the space, where the boundary of that ball is the space of geodesic rays based at $b$. In the context of Teichmüller spaces with the Teichmüller metric, the visual compactification is often called the Teichmüller compactification.
	
\subsection{Horoboundary of proper, uniquely geodesic, straight metric spaces}
	To make this work as general as possible, we begin our analysis by using the aforementioned metric properties of the Teichmüller metric. The relationship between the horofunction compactification and the visual compactification is established by observing that, for such a metric space, a sequence converging to a point in the horofunction compactification also converges in the visual compactification. This allows us to build a continuous map $\fibermap_b$ from the horofunction compactification $\overline{h(X)}$ to the visual compactification $X\cup D_b$, showing that the former is finer than the latter. In the context of Teichmüller spaces without boundary, the map $\fibermap_b$ coincides with the one defined by Liu--Shi in \cite[Definition 3.3]{LiuShi}. We may denote this map as simply $\fibermap$ when the basepoint is not relevant to the discussion.
	
	Given a geodesic $\gamma$, the path $\gamma(t)$ converges, as $t\to\infty$, to the \emph{Busemann point} associated to $\gamma$ in the horofunction compactification, which we denote $B_\gamma$. As the map $\fibermap$ is defined in terms of sequences it follows that $\fibermap(B_\gamma)=\gamma$. The existence of the map $\fibermap$ shows a strong relation between the horofunction and the visual compactification, which we state in the following result.
	\begin{theorem}\label{th:projectionfunction}
		Let $(X,d)$ be a proper, uniquely geodesic, straight metric space. For any basepoint $b\in X$, there is a continuous surjection $\fibermap$ from the horofunction compactification to the visual compactification based at b such that $\fibermap(\bussmap_\gamma)=\gamma$ for every ray $\gamma$ starting at $b$ and $\fibermap(h(p))=p$ for every $p\in X$.
		
		In particular, the horofunction compactification of $X$ is finer than the visual compactification of $X$ based at any point.
	\end{theorem}
	Most of the subsequent results in the paper follow as applications of this theorem.
	
	It is not the first time that a map such as $\fibermap$ appears in the literature. Similar maps have been found for $\delta$-hyperbolic spaces by Webster--Winchester \cite{Webster}. Walsh defined such a map for Hilbert geometries \cite{Walsh2}, which satisfy the hypothesis of the theorem whenever there are no coplanar noncollinear segments in the boundary of the convex set, as shown by de la Harpe \cite[Proposition 2]{Harpe}.
	
	The map $\fibermap$ does not induce a fibration, as its fibers $\fibermap^{-1}(\gamma)$ vary from points to higher dimensional sets (see Theorem \ref{th:dimensionfiberslowerbound}). Still, \cref{th:projectionfunction} characterizes the horoboundary as the disjoint union of all the fibers $\fibermap^{-1}(\gamma)$. Furthermore, our analysis of the topology of these fibers shows that they are path connected (see Proposition \ref{pr:pathconnected}), which gives the following characterization of the connectivity of the horoboundary.
	\begin{proposition}\label{pr:finslerconnected}
		The horoboundary of a proper, uniquely geodesic straight metric space is connected if and only if its visual boundary based at some point (and hence, any) is connected.
	\end{proposition}

	The \emph{Busemann map} $B$ from the visual compactification $X\cup D_b$ to the horofunction compactification is defined by setting $B(\gamma)=B_\gamma$ for each geodesic ray $\gamma\in D_b$ and $B(p)=h(p)$ for each $p\in X$. With this definition, the map satisfies $\fibermap \circ B =\operatorname{id}$. As the next result shows, the continuity of this map is related with the topology of the horofunction compactification.
	\begin{proposition}\label{pr:horobocompfiner}
		The visual compactification of a proper, uniquely geodesic, straight metric space based at some point is isomorphic to its horofunction compactification if and only if the Busemann map is continuous.
	\end{proposition}

	The Busemann map is essentially the identity inside $X$, so the only possible points of discontinuity are at the boundary. It is therefore of interest to find a criterion for the continuity of $B$ at the boundary, which turns out to give a criterion for when the fibers $\fibermap^{-1}(\gamma)$ are singletons.
	\begin{restatable}{proposition}{continuityatqintro}\label{pr:continuityatqintro}
		Let $X$ be a proper, uniquely geodesic, straight metric space, $b\in X$ a basepoint and $B$ the corresponding Busemann map. Furthermore, let $\gamma$ be a geodesic ray based at $b$. Then  the following three statements are equivalent:
		\begin{enumerate}
			\item The Busemann map $\bussmap$ restricted to the boundary is continuous at $\gamma$.
			\item The fiber $\fibermap^{-1}(\gamma)$ is a singleton.
			\item The Busemann map $\bussmap$ is continuous at $\gamma$.
		\end{enumerate}
	\end{restatable}
	
	In other words, we have reduced the continuity of $B$ to the continuity restricted to the boundary. This result can then be applied to different settings to obtain a more precise characterization. In the case of Teichmüller spaces \cref{pr:continuityatqintro} can be used to get an explicit criterion for the continuity of the Busemann map in terms of the quadratic differentials associated to the geodesic rays, giving us a characterization of the fibers that are singletons.
	
	\subsection{Horoboundary of the Teichmüller metric}
	Many compactifications have been defined for Teichmüller space, such as Thurston's compactification, the visual compactification (also known as the Teichm\"uller compactification) and the Gardiner--Masur compactification. These compactifications play an importal role in the study of mapping class groups and assymptotic aspects of Teichm\"uller space. See for example the articles by Thurston \cite{Thurston}, Kerckhoff \cite{Kerckhoff} or Ohshika \cite{Ohshika}. The main reason multiple compactifications have been introduced is that each one has been designed with a certain application in mind.
	
	Thurston's compactification takes the rather simple shape of a ball, upon which the mapping class groups acts as homeomorphisms. This facts make this compactification well suited for studying properties of the mapping class group. Indeed, Thurston's classification of the elements of the mapping class group relied on this compactification \cite{Thurston}. However, the Teichmüller metric is not directly related to the compactification, which results in some quirks when trying to use it to study the asymptotic geometry. For example, Lenzhen--Modami--Rafi \cite{Lenzhen} prove that there exist geodesic rays with high-dimensional limit sets.
	
	The visual compactification is defined directly using the metric, and takes the shape of a sphere where each point in the boundary has a clear geometric interpretation. This makes the compactifacation a good tool to interpret asymptotic geometric results. For example, Walsh \cite[Theorem 7]{Walsh} has proven that all geodesic rays converge to points in the visual boundary. However, as proved by Kerckhoff \cite{Kerckhoff}, the action of the mapping class group does not extend continuously to this compactification, which implies that the compactification depends on the choice of basepoint. This fact limits the usability of the visual compactification.
	
	The Gardiner--Masur compactification was initially defined to study the asymptotic properties of extremal lengths, following an analogous construction to that of the Thurston's compactification. It was later proved by Liu--Su \cite{LiuSu} that this compactification is isomorphic to the horofunction compactification with respect to the Teichm\"uller metric, giving it a geometric meaning. Furthermore, the mapping class group extends continuously to the compactification. These two properties make the Gardiner--Masur compactification a good canditate to study asymptotic properties of the Teichm\"uller metric. However, as noted by Miyachi \cite{Miyachi3} and Liu--Su \cite{LiuSu}, there is a lack of information on the shape of this compactification. In this paper, we start working towards an understanding of the shape of this boundary.
	
	Let $S$ be a compact surface with (possibly empty) boundary and finitely many marked points, where we allow marked points to be on the boundary. Denote by $\T(S)$ its Teichmüller space equipped with the Teichmüller metric. Furthermore, for any quadratic differential $q$ based at some basepoint $b\in \T(S)$, denote by $\tray{q}{\cdot}$ the geodesic ray in $\T(S)$ starting at $b$ in the direction $q$, and $V(q)$ the vertical foliation associated to $q$, see \cref{se:backgroundteichmuller} for a quick introduction or the book by Farb--Margalit \cite{primer} for a more in-depth explanation of these concepts. Recall that a measured foliation is \emph{indecomposable} if it is either a thickened curve, or a component with a transverse measure that cannot be expressed as the sum of two projectively distinct non zero transverse measures. Furthermore, each measured foliation can be decomposed uniquely into finitely many indecomposable components (see \cref{se:measuredfoliations} for detailed definitions). Walsh has shown the following characterization of the convergence of Busemann points in terms of the convergence of the associated quadratic differentials.
	\begin{theorem}[Walsh {\cite[Theorem 10]{Walsh}}]\label{th:infusibledefinitionintro}
		Let $(q_n)$ be a sequence of unit area quadratic differentials based at $b\in \T(S).$ Then, $B_{\tray{q_n}{\cdot}}$ converges to $B_{\tray{q}{\cdot}}$ if and only if both of the following hold:
		\begin{enumerate}
			\item $(q_n)$ converges to $q$ with respect to the $L^1$ norm on $T^*_b\T(S)$;
			\item for every subsequence $(G^n)_n$ of indecomposable measured foliations such that, for each $n\in \N$, $G^n$ is a component of $V(q_n)$, we have that every limit point of $G^n$ is indecomposable.
		\end{enumerate}
	\end{theorem}
	While Walsh's proof is done in the context of surfaces without boundary, it can be easily extended to our setting. In view of this theorem, we say that a sequence of quadratic differentials $(q_n)$ \emph{converges strongly} to $q$ if it satisfies the two conditions of \cref{th:infusibledefinitionintro}. Furthermore, we say that $q$ is \emph{infusible} if every sequence of quadratic differentials converging to $q$ converges strongly. By \cref{pr:continuityatqintro}, a quadratic differential $q$ is infusible if and only if the Busemann map is continuous at $\tray{q}{\cdot}$. In \cref{th:maxcondition}, we derive a topological characterization of the vertical foliations of infusible quadratic differentials. This allows us to determine precisely which surfaces only admit infusible quadratic differentials, yielding the following result.
	\begin{theorem}\label{th:homeomorphictovisualcomp}
		Let $S$ be a compact surface of genus $g$ with $b_m$ and $b_u$ boundary components with and without marked points respectively and $p$ interior marked points. Then the horofunction compactification of $\T(S)$ is isomorphic to the visual compactification if and only if $3g+2b_m+b_u+p\le 4$.
	\end{theorem}
	
	This result had been previously proven by Miyachi \cite{Miyachi3} for surfaces without boundary, that is, when $b_m=b_u=0$. For the cases where we do not have an isomorphism Miyachi found non-Busemann points in the boundary. These points are in the closure of Busemann points, which prompted Liu--Su to ask the following question:
	\begin{question}[Liu--Su {\cite[Question 1.4.2]{LiuSu}}]
		Is the set of Busemann points dense in the horofunction boundary?
	\end{question}
	We give a negative answer to this question, summed up in the following result.
	\begin{theorem}\label{th:busemannnotdense}
		Let $S$ be a closed surface of genus $g$ with $p$ marked points. Then the Busemann points are not dense in the horofunction boundary of $\T(S)$ whenever $3g+p\ge5$.
	\end{theorem}
	
	To achieve this result we use Liu--Su's \cite{LiuSu} and Walsh's \cite{Walsh} characterization of the horofunction compactification as the Gardiner--Masur compactification. We use an equivalent but slightly different definition than usual for the Gardiner--Masur compactification, as the definition we use is more well suited for our computations, and more easily extendable to surfaces with boundary (see \cref{se:GMboundaryWalshpaper} for the precise definition). For each point in the horofunction compactification there is an associated real-valued function on the set of measured foliations. We show that the functions associated to elements in the closure of Busemann points are polynomials of degree 2 with respect to some variables (see Proposition \ref{pr:busemanclosureshape} for the precise statement). We then show that the elements of the Gardiner--Masur boundary found by Fortier Bourque \cite{Fortier} do not satisfy that condition. The main ingredient for this last part of the reasoning is the following result, which shows that extremal length is not $C^{2}$ along certain smooth paths in \(\MF\).
	\begin{theorem}\label{th:extremallengthnotsmoothintro}
		Let $S$ be a closed surface of genus $g$ with $p$ marked points and empty boundary satisfying $3g+p\ge5$. Then there is a point $X\in\T(S)$ and a path $G_t$, $t\in[0,t_0]$, in the space of measured foliations on $X$, smooth with respect to the canonical piecewise linear structure of the space of measured foliations, such that $\ext(G_t)$ is not $C^2$.
	\end{theorem}
	The canonical piecewise linear structure of the space of measured foliations was developed by Bonahon \cite{Bonahon,Bonahon2,Bonahon3}. The first derivative of the extremal length along such a path was determined by Miyachi \cite{Miyachi2}, so our proof is based on finding cases where Miyachi's expression is not $C^{1}$. This follows from an explicit computation, whose complication is greatly reduced by using previous estimates established by Markovic \cite{Markovic}. 
	
	The relation between the Thurston compactification and the horofunction compactification was studied by Miyachi \cite{Miyachi}. He proves that, while neither Thurston's nor the horofunction compactification is finer than the other, there is a bicontinuous map from the union of $\T(S)$ and uniquely ergodic foliations in Thurston's boundary to a subset of the horofunction compactification. Masur showed \cite{Masur2} that this result can be interpreted to say that these two compactifications are the same almost everywhere according to the Lebesgue measure on Thurston's boundary. The image of uniquely ergodic foliations by the bicontinuous map is the set of Busemann points associated to uniquely ergodic foliations. As we show in \cref{th:nowheredense}, this set is nowhere dense within the horoboundary. Hence the map defined by Miyachi does not show that these two are the same almost everywhere according to any strictly positive measure on the horoboundary. In fact, any attempt to extend the identity map from the interior of the Thurston compactification to the interior of the horoboundary compactification to a set of full measure within the Thurston compactification results in the same problem.
	\begin{restatable}{corollary}{noequivalentmeasuresintro}\label{co:noequivalentmeasuresintro}
		Let $\nu$ be any finite strictly positive measure on the horoboundary and let $\mu$ be the Lebesgue measure on the Thurston boundary. Furthermore, let $\phi$ be a map from the Thurston compactification to the horofunction compactification satisfying $\phi|_{\T(S)}=h$, where $h$ is as in \eqref{eq:hdefinition}. Then there is no subset $U$ of the Thurston boundary with full $\mu$-measure such that $\phi$ is continuous at every point in $U$ and $\phi(U)$ has full $\nu$-measure.
	\end{restatable}
		
	Under some smoothness assumptions satisfied by Teichmüller metric, we are able to use the maps $\fibermap_b$ to give an alternative definition of the horofunction compactification based on geometric notions. This definition characterizes the horofunction compactification as the reachable subset of the product of all visual compactifications obtained by choosing different basepoints (see \cref{se:alternativehorofunctiondefinition} for details). Hence, the horofunction compactification can be interpreted as a collection of the asymptotic information provided by all visual compactifications. As a straightforward result of this alternative definition we get the following characterization of converging sequences in the horofunction compactification.
	\begin{corollary}\label{co:convergingtohoriffconvergingeveryvisintro}
		A sequence $(x_n)\subset \T(S)$ converges in the horofunction compactification if and only if the sequence converges in all the visual compactifications.
	\end{corollary}
	Considering the horocycles diverging in the horofunction compactification found by Fortier Bourque \cite{Fortier} we get that there is some visual compactification in which these horocycles do not converge.
	\begin{corollary}\label{co:divergencehorocycles}
		Let $S$ be a closed surface of genus $g$ with $p$ marked points, such that $3g+p\ge5$. There is a basepoint such that a horocycle diverges in the visual compactification based at that point.
	\end{corollary}
	This contrasts with the behavior of Teichmüller rays, which converge in all visual compactifications \cite[Theorem 7]{Walsh}.
	
	The structure of the horoboundary provided by \cref{th:projectionfunction}, as well as the path-connectivity of the fibers, allows us to prove the following connectivity result.
	\begin{theorem}\label{th:teichconnected}
		The horoboundary of any Teichmüller space of real dimension at least 2 is path connected.
	\end{theorem}
	Furthermore, we also prove that whenever the surface has empty boundary the map $\fibermap$ restricted to the horoboundary admits a section, while it only admits a section for surfaces of low complexity if the boundary is nonempty (see \cref{th:globalsection} for details).

	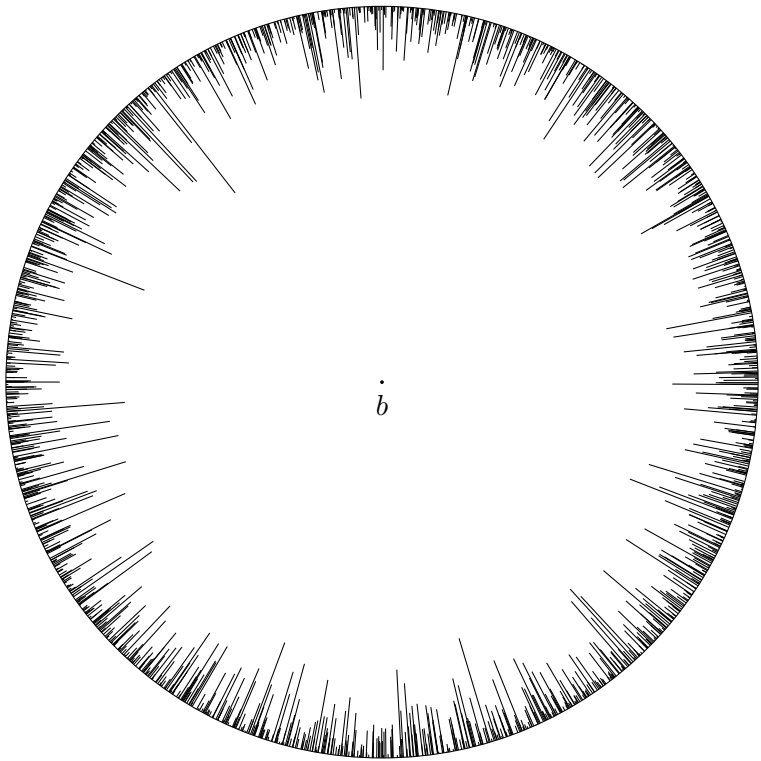
\begin{figure} \centering
		\begin{tikzpicture}
			\pgfmathsetseed{1}
			\pgfmathsetmacro\spin{5*360};
			\draw  (0,0) ellipse (5 and 5);
			\foreach \ang in {1,...,\spin}{
				\pgfmathsetmacro\disp{rnd};
				\pgfmathsetmacro\radius{ln(rnd)/4+5};
				\draw [smooth] (\ang+\disp:5) to (\ang+\disp: \radius);
			}
			\node [circle,fill,inner sep=0.5pt,label=below:$b$] (b) at (0,0) {};
		\end{tikzpicture}
		\caption{Sketch of the shape of the horoboundary of the Teichmüller metric for surfaces without boundary.}
		\label{fi:cartoonfinal}
	\end{figure}

	\cref{fi:cartoonfinal} shows a sketch of what we think the horoboundary looks like based on the results of this paper. The outer circle represents the section given by Theorem \ref{th:globalsection}. Each line perpendicular to the sphere represents one of the fibers induced by the map $\fibermap$, so it is associated with a unique Teichmüller ray starting at $b$. Note that while by Proposition \ref{pr:pathconnected} the fibers are path connected, by Theorem \ref{th:dimensionfiberslowerbound} they are bigger than segments in some cases. Furthermore, a priori they might not be contractible.
	
	The nearest point to the basepoint $b$ of each fiber represents the Busemann point associated to the geodesic joining $b$ to the fiber. This point could indeed be considered the nearest point to $b$ from the fiber, as one can access it in a straight way, through a geodesic exiting $b$. On the other hand, the points in the outer circle represents the points associated to the section alluded to earlier. These can be accessed through a sequence of Busemann points whose associated fiber is a point, which can be considered as the most tangentially possible way to reach points in the boundary. Following a result by Masur \cite{Masur2}, with respect to the measure on the fibers induced by the Lebesgue measure on the set of Teichmüller rays exiting $b$, almost all the fibers are actually points. As we shall see in \cref{th:nowheredense} these points are nowhere dense in the boundary.
	
	Note that there exist paths within the horoboundary connecting the fibers without passing through the section, and a priori there may be paths not represented in the sketch along which the fibers vary continuously. For surfaces for which the map $\fibermap$ does not admit a global section, a similar sketch could be drawn, although there would be no continuous global section in some cases. Hence, the outer circle would be broken at some places.
		
	Finally, Liu--Su's and Walsh's characterization of the horofunction compactification as the Gardiner--Masur compactification can be used to translate some of these findings to results regarding the asymptotic value of extremal length functions. For example, we get the following estimate.
	\begin{theorem}\label{th:limits}
		Let $(q_n)$ be a sequence of unit quadratic differentials converging strongly to a unit quadratic differential $q$. Denote $G_j$ the components of the vertical foliation associated to $q$, and $H(q)$ the horizontal foliation. Then, for any $F\in \MF$ and sequence $(t_n)$ of real values converging to positive infinity we have
		\[\lim_{n\to\infty} e^{-2t_n}\ext_\tray{q_n}{t_n}(F)=\sum_j\frac{i(G_j,F)^2}{i(G_j,H(q))}.\]
	\end{theorem}
	This generalizes a previous result proven by Walsh in \cite[Theorem 1]{Walsh}, where the same is shown for $q_n$ constant.

\subsection{Outline of the paper and a note for the reader interested in surfaces without boundary}
	The paper is structured as follows. In \cref{se:metricdefinitions} we introduce the necessary metric notions used in the paper. We follow in \cref{se:horofunctionmetric} by proving the results related to the more general metric setting, such as showing that the horofunction compactification is finer than the visual one. In \cref{se:backgroundteichmuller} we give a short review of the necessary background on Teichmüller spaces. In \cref{se:continuitybusemanteich} we determine which quadratic differentials are infusible, and find which surfaces admit infusible quadratic differentials, getting a proof of \cref{th:homeomorphictovisualcomp}. In \cref{se:shapeoffibers} we characterize the points in the closure of Busemann points, and get some bounds on the dimension of the fibers of the map $\fibermap$. In \cref{se:nondensity} we show that Busemann points are not dense. In \cref{se:globaltopology} we determine which surfaces result in the map $\fibermap$ having a section, and prove that the horoboundary is path connected. Finally, in \cref{se:formulas} we use the previous results to obtain estimates regarding asymptotic values of extremal lengths. 
	
	Some of the most dense parts of this paper are due to the added complexity of considering surfaces with boundary. As such, the reader focused in surfaces with empty boundary might want to omit the corresponding sections on a first reading. One of the largest related parts starts after the remark following \cref{th:maxcondition} and ends before the start of \cref{se:horocycles}. The other sizable part starts with \cref{pr:continuoussection2} and ends at the start of the proof of \cref{th:teichconnected}, where we note that the proof is significantly simpler in the case of surfaces without boundary.
	
\subsection*{Acknowledgments}
	The author would like to thank Maxime Fortier Bourque and Vaibhav Gadre for many helpful discussions and corrections.

\section{Metric definitions}\label{se:metricdefinitions}

\subsection{Compactifications}

	A compactification of a space serves, among other things, as a way of characterizing convergence to infinity. Formally, a \emph{compactification} of a topological space $X$ is a pair $(f,\overline{X})$, where $\overline{X}$ is a compact topological space and $f:X\to \overline{X}$ is an embedding such that $f(X)$ is dense in $\overline{X}$. The boundary of a compactification $\partial \overline{X}=\overline{X}-X$ describes the different ways of converging to infinity provided by that compactification. We shall usually identify the points in $X$ with the ones in $\overline{X}$ via the map $f$, and say that a sequence $(x_n)\subset X$ converges in $\overline{X}$ if $f(x_n)$ converges.
	
	A compactification $(f_1,X_1)$ is \emph{finer} than another one $(f_2,X_2)$ if there exists a continuous map $\overline{f}_2:X_1\to X_2$ such that $\overline{f}_2\circ f_1 =f_2$. Since $f_2(X)$ is dense in $X_2$, the continuous extension $\overline{f}_2$ is surjective. Furthermore, we can restrict the map $\overline{f}_2$ to the boundary to get a surjective map $\left.\overline{f}_2\right\vert_{\partial X_1}:\partial X_1\to \partial X_2$, which can be seen as a projection. Having a compactification finer than another ones means, from an intuitive point of view, that the finer compactification catalogs more ways of converging to infinity than the other one. Namely, any sequence in $X$ converging in the finer compactification converges also in the coarser one, while the opposite may not be true.
	
	We say that two compactifications are isomorphic if each one is finer than the other one. The following Lemma found in \cite[Lemma 17]{Walsh} coincides with the intuitive notion of finer compactifications.
	
	\begin{lemma}\label{le:walshcompact}
		Let $(f_1,X_1)$ and $(f_2,X_2)$ be two compactifications of $X$ such that $f_2$ extends continuously to an injective map $\overline{f}_2:X_1\to X_2$. Then the two compactifications are isomorphic.
	\end{lemma}

	We will usually refer to the space $\overline{X}$ as the compactification when the embedding is clear from the context.
	Since the images of $X$ by the embedding are dense, the extensions we get to compare the compactifications are unique. That is, we have the following result
	\begin{lemma}
		Let $(f_1,X_1)$ and $(f_2,X_2)$ be two compactifications of $X$ such that $X_1$ is finer than $X_2$. Then the extension $\overline{f_2}:X_1\to X_2$ is unique.
	\end{lemma}
	\begin{proof}
		For any $x\in X$ we have $\overline{f_2}(f_1(x))=f_2(x)$. Hence, the image of $\overline{f_2}$ is determined on a dense subset of $X_1$, so by continuity it is determined on $X_1$.
	\end{proof}

\subsection{Visual compactification of proper, uniquely geodesic, straight spaces.} \label{se:visualdefinition}

	Let $(X,d)$ be a metric space. We shall say that a map $\gamma$ from an interval $I\subset \R$ to $X$ is a \emph{geodesic} if it is an isometric embedding, that is, if $d(\gamma(t),\gamma(s))=|t-s|$. We shall consider two geodesics to be equal if their image is equal and have the same orientation. A space is \emph{uniquely geodesic} if for any two distinct points $a,b\in X$ there is a unique geodesic starting at $a$ and ending at $b$. 

	Furthermore, we say that the space is \emph{proper} if the closed balls $D(x,r)=\{p\in X\mid d(p,x)\le r\}$ are compact. 
	
	If geodesic segments can be extended uniquely, that is, if for any geodesic segment $\gamma_1$ there is a unique bi-infinite geodesic $\gamma_2$ such that $\gamma_1\cap \gamma_2=\gamma_1$, we say that the space is \emph{straight}. 
	
	Let then $X$ be a proper, uniquely geodesic, straight space and let $D_b$ be the set of infinite geodesic rays starting at $b$, with the topology given by uniform convergence on compact sets. Furthermore, denote $S_b^1=\{x\in X \mid d(x,b)=1\}$ the sphere of radius 1 around $b$. 
	\begin{lemma}\label{le:geodesicsequalsphere}
		The map from $D_b$ to $S_b^1$ defined by sending $\gamma\in D_b$ to $\gamma(1)$ is a homeomorphism.
	\end{lemma}
	\begin{proof}
		Since the topology on $D_b$ is given by uniform convergence on compact sets, the point $\gamma(1)$ varies continuously with respect to $\gamma$. 
		
		On the other hand, since the space is straight and has unique geodesics, given any point $a\in S_b^1$ there is a unique geodesic ray starting at $b$ and passing through $a$. This is the inverse to the map obtained by evaluating the geodesics. To see that the relation is continuous we consider a sequence $(a_n)\subset S_b^1$ converging to some $a$, and denote $(\gamma_n)$ and $\gamma$ the associated geodesics. Assume $\gamma_n$ does not converge to $\gamma$. Then we have a subsequence without $\gamma$ as an accumulation point. For any $t>0$, the geodesic segments $\left.\gamma\right\vert_{[0,t]}$ are contained in the ball of radius $t$, which is compact, as $X$ is proper. As these are geodesics we have equicontinuity, so by Arzelà--Ascoli we can take a subsequence converging uniformly to some path $\gamma'$. Since the distance function is continuous, $\gamma'$ is a geodesic. Furthermore, $\gamma'(1)=\lim_{n\to\infty} \gamma_n(1)=\lim_{n\to\infty} a_n= a$. By uniqueness of geodesics, $\gamma'$ and $\gamma$ are equal when restricted to $[0,1]$, which by straightness implies they are equal. Hence, $\gamma_n$ converges to $\gamma$ uniformly on the compact $[0,t]$.
	\end{proof}
	
	Following a similar reasoning it is possible to show the following, still under the same hypotheses on $X$.
	\begin{lemma}\label{le:straightspacesareballs}
		The space $X$ is homeomorphic to $D_b\times [0,\infty)/D_b\times\{0\}$.
	\end{lemma}
	\begin{proof}
		We define the map $C:D_b\times [0,\infty)/D_b\times\{0\}\to X$ given by $C(\theta,r)=\theta(r)$. This is well defined, as $C(\theta,0)=b$ for any $\theta\in D_b$. Furthermore, this is a bijection, since for every $x\in X-\{b\}$ there is a unique geodesic ray from $b$ to $x$. The map is continuous, as the topology on $D_b$ is given by uniform convergence on compact sets. To see that the inverse is continuous consider a sequence $a_n\in X$ converging to some $a\in X$. If $a=b$, then $d(a_n,b)\to 0$, so we have continuity. Otherwise we denote $r_n=d(a_n,b)$ and $r=d(a,b)$. We have $r_n\to r$, so denoting $(\gamma_n)$ and $\gamma$ the unique geodesic in $D_b$ such that $\gamma_n(r_n)=a_n$ and $\gamma(r)=a$ and applying Arzelà--Ascoli's theorem in the same way as in \cref{le:geodesicsequalsphere}, we have that $\gamma_n$ converges to $\gamma$.
	\end{proof}
			
	The space $D_b\times [0,\infty)/D_b\times\{0\}$ can be included into the compact space $D_b\times [0,\infty]/D_b\times\{0\}$, which can be written as $\left(D_b\times [0,\infty)/D_b\times\{0\}\right)\cup D_b\times\{\infty\}$. Using the homeomorphism from \cref{le:straightspacesareballs}, we can use this inclusion to give a compact topology on the space $X\cup D_b$. The \emph{visual compactification} is defined as the pair $(i,X\cup D_b)$, where $i$ is the inclusion $i:X\to X\cup D_b$ and the topology on the space $X\cup D_b$ is the one we just defined. We shall denote $X\cup D_b$ as $\vbc{X}_b$, or $\vbc{X}$ when the basepoint is not relevant to the discussion.

\subsection{Horofunction compactification}\label{se:horodefinition}

	The second compactification that will play a part in this paper is slightly more involved and difficult to visualize.
	
	Let $X$ be a proper, uniquely geodesic, straight metric space. Given a basepoint $b\in X$, one can embed $X$ into the space of continuous functions from $X$ to $\R$ via the map $h:X\to C(X)$ defined by
	\[h(x)(\cdot ):=d(x,\cdot)-d(x,b).\]
	The topology given to $C(X)$ is that of uniform convergence on compact sets. The map $h$ is indeed continuous, as the distance function is continuous. Furthermore, $h$ is injective, as $h(x)$ has a strict global minimum at $x$. It can also be proven that since $X$ is proper, $h$ is an embedding. For more details about this construction see \cite[Section 2]{Walsh3}. Furthermore, the properness of $X$ implies it is second countable, so the closure of $h(X)$ is compact, Hausdorff and second countable. We shall denote the closure of $h(X)$ on $C(X)$ as $\hbc{X}$. The \emph{horofunction compactification} is defined as the pair $(h,\hbc{X})$. We call the set $\hbd{X}=\hbc{X}-X$ the \emph{horofunction boundary} or \emph{horoboundary}, and we call its members \emph{horofunctions}. If we want to specify the chosen basepoint we write $\hbc{X}_b$. However, it is possible to see that quotienting the compactification by letting $f\sim g$ whenever the difference is constant we get an isomorphic compactification, showing that the horofunction compactification does not depend on the basepoint.

	Usually the easier points to identify in the horoboundary are the Busemann points. These are the ones that can be reached as a limit along almost geodesics, which is a slight weakening of the notion of geodesic by allowing an additive constant approaching $0$. That is, a path $\gamma:[0,\infty)\to X$ is an \emph{almost geodesic} if for each $\varepsilon>0$,
	\[\left|d(\gamma(0),\gamma(s))+d(\gamma(s),\gamma(t))-t\right|<\varepsilon\]
	for all $s$ and $t$ large enough, with $s\le t$.
	Rieffel \cite{Rieffel} proved that every almost geodesic converges to a limit in $\hbd{X}$. A horofunction is called a \emph{Busemann point} if there exists an almost geodesic converging to it. We shall denote the Busemann point associated in this way to the almost geodesic $\gamma$ by $B_\gamma$.

\section{Horofunction compactification of proper, uniquely geodesic, straight metric spaces.}\label{se:horofunctionmetric}

\subsection{The relation between the horofunction compactification and the visual compactification}

Fix a uniquely geodesic, proper and straight metric space $(X,d)$ and a basepoint $b\in X$. We will assume $X$ satisfies these hypotheses through this section.
For each geodesic ray $\gamma\in \vbd{X}$ starting at $b$ there is an associated Busemann point $B_\gamma\in \hbd{X}$. We can extend this map to all the visual compactification by setting it as the identification with the map $h$ on $X$ given by the horofunction compactification. That is, we define the \emph{Busemann map} $B:\vbc{X}\to \hbc{X}$ by setting $B(\gamma)=B_\gamma$ for $\gamma\in \vbd{X}$ and $B(x)=h(x)$ for $x\in X$. The relevance of this map can be seen with the following result.

\begin{lemma}\label{le:visualfineriffBcontinuous}
	The visual compactification $(i, \vbc{X})$ is finer than the horofunction compactification $(h,\hbc{X})$ if and only if the map $B$ is continuous.
\end{lemma}
\begin{proof}
	We have that $B(i(x))=h(x)$, so $B$ is an extension of $h$ to $\vbc{X}$. Hence, if $B$ is continuous, then the visual compactification is finer than the horofunction compactification.
	
	On the other hand, if the visual compactification is finer than the horofunction compactification, then we have a continuous map $f:\vbc{X}\to \hbc{X}$. For every $x\in X$, we have $f(i(x))=h(x)=B(i(x))$. Furthermore, for any ray $\gamma$ starting at the basepoint we have $f(\gamma)=\lim_{t\to\infty} f(i(\gamma(t))=\lim h(\gamma(t))=B(\gamma)$. Hence, $B=f$, and $B$ is continuous.
\end{proof}

In general, the Busemann map may not be surjective nor continuous. However, we have the following.

\begin{proposition}\label{pr:bussmapinjective}
	For a proper, uniquely geodesic, straight metric space $(X,d)$ the Busemann map is injective.
\end{proposition}
\begin{proof}
	For each $x\in X$, the associated function $h(x)$ has a global minimum at $x$, while $B_\gamma$ is unbounded below for every $\gamma\in \vbd{X}$. Hence, in the interior of $\vbc{X}$ the map is injective and $B(X)\cap B(\vbd{X})=\emptyset$. Assume we have $\gamma, \gamma'\in \vbd{X}$ such that $\gamma\neq \gamma'$ and $B(\gamma)=B(\gamma')=\xi$. Then, for a given sequence $t_n\to\infty$ we have $\lim_{n\to\infty}h(\gamma(t_n))=\lim_{n\to\infty}h(\gamma'(t_n))=\xi$. For any $t\in\R$ and any $n$ such that $t_n>t$ we have 
	\[h(\gamma(t_n))(\gamma(t))=d(\gamma(t_n),\gamma(t))-d(\gamma(t_n),\gamma(0))=t_n-t-t_n=-t,\]
	and similarly for $\gamma'$.
	Hence $\xi(\gamma(t))=\xi(\gamma'(t))=-t$ for all $t$.
	
	Fix now a $t>0$. We have
	\begin{align*}
		-t =\xi(\gamma'(t))=&\lim_{n\to\infty}(d(\gamma'(t),\gamma(t_n))-d(b,\gamma(t_n)))\\
		=&\lim_{n\to\infty}(d(\gamma'(t),\gamma(t_n))-t_n).
	\end{align*}
	That is, there is a sequence $\varepsilon_n$ with $\varepsilon_n\to 0$ such that 
	\[t_n-t+\varepsilon_n \ge d(\gamma'(t),\gamma(t_n))\ge t_n-t-\varepsilon_n.\]
	for every $n$.
	
	By straightness we can extend $\gamma$ in the negative direction towards $\gamma(-s)$ for some $s>0$. We shall now show that the geodesic $\gamma$ does not minimize the distance between $\gamma(-s)$ and $\gamma(t_n)$ for $n$ big enough. Since the space is straight, the geodesic segment between $\gamma(-s)$ and $b$ can be extended uniquely, so concatenating it with the segment between $b$ and $\gamma'(t)$ does not result in a geodesic. Hence, the distance between $\gamma'(t)$ and $\gamma(-s)$ is strictly smaller than $s+t$. That is, there is some $\delta>0$ such that $d(\gamma(-s),\gamma'(t))<t+s-\delta$. As shown in \cref{fi:badtriangles} we get a path going from $\gamma(-s)$, to $\gamma(t_n)$, passing through $\gamma'(t)$ that has length less than $t+s-\delta+t_n-t+\varepsilon_n=t_n+s-\delta+\varepsilon_n$. Hence, taking $n$ big enough so that $\varepsilon_n<\delta$ we get that the geodesic segment between $\gamma(-s)$ and $\gamma(t_n)$ is not minimizing. This is a contradiction, from which we conclude that $\gamma=\gamma'$. Therefore, $B$ is injective.
	
	\begin{figure} \centering
		\begin{tikzpicture}[scale=3]

			\node at (1/3,1/9) [circle,fill,inner sep=1pt,label=below:$b$](b){};
			\node at (0.5,0.8) [circle,fill,inner sep=1pt,label=above:$\gamma'(t)$](away){};
			\node at (3,1) [circle,fill,inner sep=1pt,label=right:$\gamma(t_n)$](upper){};
			\node at (-1/3,-1/9) [circle,fill,inner sep=1pt,label=left:$\gamma(-s)$](back){};
			\draw (b) -- node[label=right:$t$]{} (away);
			\draw (b) -- node[label=below:$t_n$]{} (upper);
			\draw (upper) --  node[label=above:$<t_n-t+\varepsilon_n$]{} (away);
			\draw (back) -- node[label=below:$s$]{} (b);
			\draw (back) --  node[label=left:$<t+s-\delta$]{} (away);

		\end{tikzpicture}
		\caption{The triangles involved in the proof of Proposition \ref{pr:bussmapinjective}.}
		\label{fi:badtriangles}
	\end{figure}
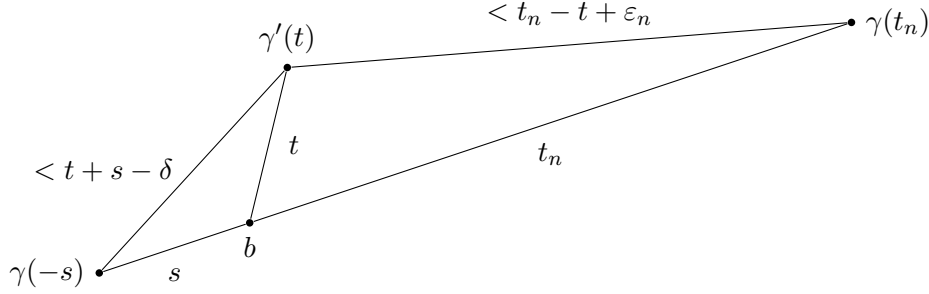
\end{proof}

	Hence, given a Busemman point $\xi$ in $B(\vbd{X})$ we have a unique associated geodesic ray $\gamma\in\vbd{X}$ such that $\xi(\gamma(t))=-t$ for all $t$. Our next aim is to build a similar relation for all other horofunctions. Our approach is similar to the one used by Walsh in \cite[Section 7]{Walsh}.
	
	We say that a geodesic $\gamma$ is an \emph{optimal geodesic} for a certain horofunction $\xi\in \hbc{X}$ if $\xi(\gamma(t))-\xi(\gamma(0))=-t$ for all $t\in \R$. We shall now see that each function in the horoboundary has at least one optimal geodesic.

\begin{lemma}\label{le:fixingvaluealongpath}
	Let $X$ be a proper, uniquely geodesic, straight metric space and let $\xi\in \hbd{X}$ be a horofunction. Suppose that
	$(x_n)\subset X$ converges to $\xi$, with $x_n=\gamma_n(t_n)$, $\gamma_n\in\vbd{X}$ and $(\gamma_n)$ converging to $\gamma$ as $n\to \infty$. Then $\xi(\gamma(t))=-t$ for every $t\in \R$. That is, $\gamma(t)$ is an optimal geodesic for $\xi$.
\end{lemma}
\begin{proof}
	Fix $t$. We have that \[\xi(\gamma(t))=\lim_{n\to\infty}(d(\gamma(t),\gamma_n(t_n))-d(b,\gamma_n(t_n)))=\lim_{n\to\infty}(d(\gamma(t),\gamma_n(t_n))-t_n).\]
	As $n$ goes to infinity, $\gamma_n$ converges to $\gamma$. Hence by the given topology on the visual boundary, the maps $\gamma_n(\cdot)$ converge uniformly on compact sets to the geodesic $\gamma(\cdot)$. In particular, denoting $d(\gamma(t),\gamma_n(t))=\varepsilon_n$ we have $\varepsilon_n\to 0$. 
	We get then \cref{fi:squeezedtriangle}, so by the triangle inequality, \[|d(\gamma(t),\gamma_n(t_n))-(t_n-t)|=|d(\gamma(t),\gamma_n(t_n))-d(\gamma_n(t),\gamma_n(t_n))|\le\varepsilon_n,\]
	and so $\xi(\gamma(t))=-t$.
	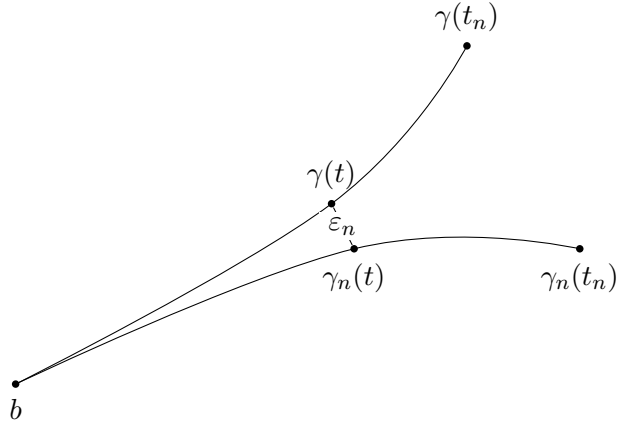
\begin{figure} \centering
		\begin{tikzpicture}[scale=3]
			\node [circle,fill,inner sep=1pt,label=below:$b$] (b) at (0,0) {};
			\node [circle,fill,inner sep=1pt,label=below:$\gamma_n(t_n)$] (c) at (2.5,0.6) {};
			\node [circle,fill,inner sep=1pt,label=above:$\gamma(t_n)$] (d) at (2,1.5) {};
			\node [circle,fill,inner sep=1pt,label=above:$\gamma(t)$] (e) at (1.4,0.8) {};
			\node [circle,fill,inner sep=1pt,label=below:$\gamma_n(t)$] (f) at (1.5,0.6) {};
			\draw plot[smooth, tension=.7] coordinates {(b) (f) (c)};
			\draw plot[smooth, tension=.7] coordinates {(b) (e) (d)};
			\draw (e) --  node[fill=white,inner ysep=2pt]{$\varepsilon_n$} (f);
		\end{tikzpicture}
		\caption{
			In the proof of Lemma \ref{le:fixingvaluealongpath}, $\gamma_n$ converges to $\gamma$, so $\gamma_n(t)$ converges to $\gamma(t)$, and hence the distance between $\gamma_n(t_n)$ and $\gamma_n(t)$ gets arbitrarily close to the distance between $\gamma_n(t_n)$ and $\gamma(t)$.
		}
		\label{fi:squeezedtriangle}
	\end{figure}
\end{proof}

Since $\vbd{X}$ is compact, for any horofunction $\xi\in\hbd{X}$ and sequence $(x_n)\subset X$ converging to $\xi$ we can take a subsequence such that the hypotheses of \cref{le:fixingvaluealongpath} are satisfied, so each $\xi\in\vbd{X}$ does have at least one optimal geodesic.

If $\xi$ has another optimal geodesic $\gamma'$ with $\gamma'(0)=\gamma(0)$ we have at least two geodesics along which $\xi(\gamma(t))=\xi(\gamma'(t))=-t$ for all $t$. Following a reasoning similar to the one in the proof of \cref{pr:bussmapinjective}, we get a contradiction. This time, however, we have to be a bit more careful about the distances, as instead of two fixed rays we have a fixed ray and a sequence converging to a distinct fixed ray.

\begin{proposition}\label{pr:optimalpath}
	Let $\xi\in \hbd{X}$ and $b\in X$. Then there is a unique optimal geodesic for $\xi$ passing through $b$.
\end{proposition}
\begin{proof}
	Let $(x_n)=(\gamma_n(t_n))$ be a sequence converging to $\xi$, with $(\gamma_n) \subset \vbd{X}$, and take a subsequence such that $\gamma_n$ converges to some geodesic $\gamma$. By \cref{le:fixingvaluealongpath}, $\gamma$ is an optimal geodesic. Assume that we have a different optimal geodesic $\gamma'$ passing through $b$.
	
	Using that $h(\gamma_n(t_n))$ converges pointwise to $\xi$ we have
	\begin{align*}
		-t =\xi(\gamma'(t))=&\lim_{n\to\infty}(d(\gamma'(t),\gamma_n(t_n))-d(b,\gamma_n(t_n)))\\
		=&\lim_{n\to\infty}(d(\gamma'(t),\gamma_n(t_n))-t_n).
	\end{align*}
	
	Hence, there is a sequence $\varepsilon_n$ with $\varepsilon_n\to 0$ such that 
	\[t_n-t+\varepsilon_n \ge d(\gamma'(t),\gamma_n(t_n)) \ge t_n-t-\varepsilon_n.\]
	We proceed by showing that for $n$ big enough there is some $s>0$ such that the geodesic $\gamma_n$ does not minimize the distance between $\gamma_n(-s)$ and $\gamma_n(t_n)$. As in the proof of Proposition \ref{pr:bussmapinjective}, by applying the triangle inequality between \(\gamma'(t),\gamma(-s)\) and \(b\) we have $d(\gamma'(t),\gamma(-s))<s+t$. Fix $s>0$ and pick $\delta>0$ such that $d(\gamma'(t),\gamma(-s))< t+s-\delta$. Since $\gamma_n$ converges to $\gamma$ uniformly on compact sets, $\gamma_n(-s)$ converges to $\gamma(-s)$.
	Hence, $d(\gamma'(t),\gamma_n(-s))$ converges to $d(\gamma'(t),\gamma(-s))$. 
	Then for $n$ big enough we have $d(\gamma'(t),\gamma_n(-s)) < t+s-\delta$. Consider then $n$ big enough so that $\varepsilon_n\le \delta/2$ as well. The triangle between $\gamma'(t),\gamma_n(-s)$ and $\gamma_n(t_n)$ gives
 \[
 d(\gamma_n(-s),\gamma_n(t_n))\le d(\gamma_n(-s),\gamma'(t))+d(\gamma'(t),\gamma_n(t_n))< (t+s-\delta)+(t_n-t+\epsilon_n)<t_n+s.
 \]
 This is a contradiction, which proves the uniqueness of $\gamma$.
\end{proof}

Given a basepoint $b\in X$ we can now define a map $\fibermap_b:\hbc{X}\to\vbc{X}_b$ by sending any $\xi\in\hbd{X}$ to the unique optimal geodesic $\gamma$ of $\xi$ with $\gamma(0)=b$, and by sending $h(x)$ to $x$ for any $x\in X$. This map is indeed an extension of the relation we had established for Busemann points in $\B(\vbd{X})$, since if $\xi=B(\gamma)$ for $\gamma\in D_b$ then $\gamma$ is an optimal geodesic of $\xi$, giving us $\fibermap_b(B(\gamma))=\gamma$.

We will often write $\fibermap$ instead of $\fibermap_b$ whenever the basepoint is not relevant to the discussion.
To prove that $\fibermap$ is continuous, we first have to see the following result.
\begin{proposition}\label{pr:welldefined}
	Let $(x_n)\subset X$ be a sequence converging to $\xi\in\hbd{X}$. Then, $(x_n)$ has a unique accumulation point in the visual compactification. Further, this accumulation point depends only on $\xi$.
\end{proposition}
\begin{proof}
	Since $\vbd{X}$ is compact, $(x_n)$ has accumulation points in the visual compactification. If $(x_n)$ has two accumulation points we can take two subsequences converging to two different geodesics, which by Lemma \ref{le:fixingvaluealongpath} are optimal geodesics, contradicting Proposition \ref{pr:optimalpath}.
	
	If there is another sequence $(y_n)$ converging to $\xi$ with a different accumulation point the result follows by merging both sequences and repeating the reasoning.
\end{proof}

	Hence, $\fibermap$ can be alternatively defined by sending any $\xi\in \hbd{X}$ to the unique accumulation point in $\vbc{X}$ of the sequences converging to $\xi$ in $\hbc{X}$, and by sending $h(x)$ to $x$ for any $x\in X$. By \cref{pr:welldefined}, this definition is equivalent to the previous one.

	By this second definition of the map $\fibermap$, we see how it is mostly related to the convergence of sequences, so using a diagonal sequence argument we can prove its continuity.
\begin{proposition}\label{pr:fibermapcontinuous}
	The map $\fibermap$ is continuous.
\end{proposition}
\begin{proof}
	Take a sequence $(\xi_n)\subset \hbc{X}$ converging to $\xi$. If $\xi\in h(X)$ we have that, as $h(X)$ is open, $\xi_n\in h(X)$ for $n$ big enough. Hence, $\fibermap(\xi_n)=h^{-1}(\xi_n)$, which converges to $h^{-1}(\xi)$, as $h$ is a homeomorphism with its image.
	
	If $\xi\in \hbd{X}$ we split the sequence into two subsequences, one contained in $h(X)$ and one contained in $\hbd{X}$. The one contained in $h(X)$ converges to $\xi$, so by definition of $\fibermap$ and we have $\fibermap(\xi)=\lim_{n\to\infty} h^{-1}(\xi_n)$.
	
	Assume then that $(\xi_n)\subset \hbd{X}$ converges to $\xi$. We want to see that $\gamma_n=\fibermap(\xi_n)$ converges to $\gamma=\fibermap(\xi)$. For each $\xi_n$ we can take a sequence $\left(h(\gamma_n^m(t_n^m))\right)_m$ converging, as $m\to\infty$ to $\xi_n$. By Proposition \ref{pr:welldefined} the sequence $\gamma_n^m(t_n^m)$ converges to $\gamma_n$. Let $\gamma'$ be an accumulation point of $\gamma_n$. Take a convergent subsequence of $\gamma_n$ converging to $\gamma'$, and relabel it as $\gamma_n$. Let $(V_n)$ be a nested sequence of open neighborhoods of $\xi$ in $\hbc{X}$ such that $\xi_n\in V_n$ and $\bigcap_n V_n = \{\xi\}$ and let $(W_n)$ be a nested sequence of open neighborhoods of $\gamma'$ in $\vbc{X}$ such that $\gamma_n\in W_n$ and $\bigcap_n W_n= \{\gamma'\}$. We can take such sequences of sets, as both spaces are metrizable.
	
	For each $n$, there exists $m(n)$ big enough so that $\gamma_n^{m(n)}\in W_n$ and $h(\gamma_n^{m(n)}(t_n^{m(n)}))\in V_n$. By the first condition on $m(n)$, we have that $\gamma_n^{m(n)}$ converges to $\gamma'$. By the second condition, $h(\gamma_n^{m(n)}(t_n^{m(n)}))$ converges to $\xi$, so by the the definition of $\fibermap$ and Proposition \ref{pr:welldefined} the sequence $\gamma_n^{m(n)}$ converges to $\fibermap(\xi)=\gamma$. Hence, $\gamma=\gamma'$, so the only accumulation point of $(\gamma_n)$ is $\gamma$ and by compactness of $\vbd{X}$ the sequence $(\gamma_n)$ converges to $\gamma$.
\end{proof}

By combining Propositions \ref{pr:welldefined} and \ref{pr:fibermapcontinuous} we get that $\fibermap$ is the map announced at the introduction, giving us a proof of Theorem \ref{th:projectionfunction}. As mentioned in the introduction, this map shows that the horofunction compactification is finer than the visual compactification. By using the Busemann map to insert the visual boundary inside the horoboundary, we can consider the map $\fibermap$ as a projection. 

One straightforward consequence of the continuity of $\fibermap_b$ is as follows.
\begin{corollary}\label{co:convergencegeodesics}
	Let $\gamma$ be a geodesic ray, not necessarily starting at the basepoint $b\in X$. Then, $\gamma$ converges in the visual compactification of $X$ based at $b$.
\end{corollary}
\begin{proof}
	The ray $\gamma$ converges in the horofunction compactification to $B_\gamma$. Since $\fibermap_b$ is continuous, the ray also converges in the visual compactification based at $b$ to $\fibermap_b(B_\gamma)$.
\end{proof}
For Teichmüller spaces with the Teichmüller metric this result was first proved by Walsh \cite[Theorem 7]{Walsh}.

By Lemma \ref{le:visualfineriffBcontinuous}, the visual compactification is finer than the horofunction compactification if and only if the Busemann map is continuous. Hence, since the horofunction compactification is always finer than the visual compactification, we obtain an isomorphism whenever this is the case, resulting in \cref{pr:horobocompfiner}.

\subsection{The fiber structure}
To get a better picture of the shape of the horoboundary we shall study the shape of the preimages of the projection $\fibermap$ restricted to the boundary. That is, for a given point $\gamma$ in the visual boundary we are interested in finding out information about the fiber $\fibermap^{-1}(\gamma)$. We first prove the following lemma, which we will use to get bounds on the values of $\fibermap^{-1}(\gamma)$.

\begin{lemma}\label{le:strictinequality}
	Fix a geodesic ray $\gamma \in \vbd{X}$ and $p\in X$ not in the bi-infinite extension of the geodesic ray $\gamma$. Then, the function $h(\gamma(\cdot))(p)$, with domain $[0,\infty)$, is strictly decreasing.
\end{lemma}
\begin{proof}
	Take $t,s\ge0$ with $s<t$. By the triangle inequality we have 
	\[d(\gamma(t),p)\le d(\gamma(s),p)+ d(\gamma(t),\gamma(s))=d(\gamma(s),p)+t-s.\]
	Further, we have strict inequality, as equality would give us two different paths with the same length between $\gamma(t)$ and $p$, with one of them being geodesic. Hence,
	\begin{align*}
		h(\gamma(t))(p)&=d(\gamma(t),p)-d(\gamma(t),b)\\
		&<d(\gamma(s),p)+t-s-t\\
		&=h(\gamma(s))(p).
	\end{align*}
\end{proof}

The set $C(X)$ can be partially ordered by saying that $f\ge g$ whenever $f(x)\ge g(x)$ for all $x\in X$. If $f\ge g$ and $f\neq g$ then we write $f>g$. If $p=\gamma(r)$ for some $r$ and $s<t$ we have $h(\gamma(s))(p)=h(\gamma(t))(p)=-r$ for $r\le s$ and $-s=h(\gamma(s))(p)>h(\gamma(t))(p)=-\min(r,t)$ otherwise. Hence, adding the previous lemma we have $h(\gamma(s))>h(\gamma(t))$ whenever $s<t$. By attempting to extend this relation to the horofunction boundary we get that Busemann points are maximal in their fibers.

\begin{proposition}\label{pr:upperbound}
	Let $\gamma\in \vbd{X}$ and $\xi\in \fibermap^{-1}(\gamma)$. Then, $\xi\le B(\gamma)$.
\end{proposition}

\begin{proof}
	Choose any sequence $(x_n)\subset X$ such that $h(x_n)$ converges to $\xi$. Since $\xi\in\fibermap^{-1}(\gamma)$ the sequence $(x_n)$ converges to $\gamma$ in $\vbc{X}$, so we can write $x_n=\gamma_n(t_n)$ with $t_n$ converging to infinity and $\gamma_n$ converging to $\gamma$.

	Fix $p\in X$ and let $\varepsilon>0$. Denote $s_n=\sup\{t: d(\gamma(t),\gamma_n(t))<\varepsilon \text{ and } t<t_n\}$. The geodesics $\gamma_n$ converge to $\gamma$ uniformly on compact sets, so $s_n\to\infty$ as $n\to\infty$. Hence, by definition of the Busemann point and since $d(\gamma_n(s_n),\gamma(s_n))<\varepsilon$,
	\[B_\gamma(p)=
	\lim_{n\to\infty}h(\gamma(s_n))(p)\ge\limsup_{n\to\infty}h(\gamma_n(s_n))(p)-2\varepsilon .\]
	Furthermore, $s_n\le t_n$, so by Lemma \ref{le:strictinequality},
	\begin{align*}
		\xi(p)=
		\lim_{n\to\infty}h(\gamma_n(t_n))(p)\le\limsup_{n\to\infty}h(\gamma_n(s_n))(p)\le B_\gamma(p)+2\varepsilon.
	\end{align*}
	Since $\varepsilon$ can be arbitrarily small we get the proposition.
\end{proof}

While it might not be possible to get a similar unique minimum in each fiber, we can get the following result.

\begin{proposition}\label{pr:lowerbound}
		Let $\gamma\in \vbd{X}$ and $\xi\in \fibermap^{-1}(\gamma)$. Furthermore, let $(x_n)\subset X$ be a sequence converging to $\xi$ with $x_n=\gamma_n(t_n)$. For any $p$, define $\eta(p)=\liminf_{n\to\infty} B(\gamma_n)(p)$. Then, $\xi\ge \eta$.
\end{proposition}

\begin{proof}
	The proof follows a similar reasoning as the last one.
	
	Fix $p\in X$, choose a subsequence so $B(\gamma_n)(p)$ converges to $\eta(p)$ and let $(\varepsilon_m)$ be a sequence of positive numbers coverging to $0$. For each $\varepsilon_m$, take $n(m)$ big enough so that $B(\gamma_{n(m)})(p)\ge \eta(p)-\varepsilon_m$. Further, take $s_{m}$ bigger than $t_{n(m)}$, and big enough so that
	\[h(\gamma_{n(m)}(s_m))(p)\ge B(\gamma_{{n(m)}})(p)-\varepsilon_m.\]
	Such an $s_m$ always exists by the definition of $B(\gamma_{n(m)})$.
	In particular, we have that 
	\[\liminf_{m\to\infty}h(\gamma_{n(m)}(s_m)(p)\ge \eta(p).\]
	By Lemma \ref{le:strictinequality} we have
	\begin{align*}
		\xi(p)=\lim_{m\to\infty}h(\gamma_{n(m)}(t_{n(m)})(p)\ge\liminf_{m\to\infty} h(\gamma_{n(m)}(s_m))(p)\ge \eta(p).
	\end{align*}
\end{proof}

The intuition one might get from these propositions is that approaching $\gamma$ ``through the boundary'', that is, through the furthest way possible from the interior of $X$, gives a lower bound on the possible values of approaching through other angles, and approaching $\gamma$ in a straight way, that is, through the geodesic, gives an upper bound. Hence, when these two ways of approaching $\gamma$ are the same, every other possible angle of approach should also yield the same limit. Following this reasoning we get our next result, announced in the introduction.

\continuityatqintro*

\begin{proof}
	$(1)\implies(2)$:
	Take $\xi\in \fibermap^{-1}(\gamma)$. By Proposition \ref{pr:upperbound} we have $\xi\le B(\gamma)$. Since $\bussmap$ is continuous at $\gamma$ when restricted to the boundary we have that for any $\gamma_n\to \gamma$ the horofunctions $B(\gamma_n)$ converge to $B(\gamma)$. Hence, by Proposition \ref{pr:lowerbound}, $\xi\ge B(\gamma)$, so $\xi=B(\gamma)$ and we have (2).
	
	$(2)\implies (3)$: Take then any $(x_n)\subset \vbc{X}$ converging to $\gamma$, consider the sequence $(B(x_n))\subset \hbc{X}$ and let $\eta$ be an accumulation point. By the definition of $\fibermap$ we have $\eta\in \fibermap^{-1}(\gamma)$, so $\eta=B(\gamma)$ since we assumed that $\fibermap^{-1}(\gamma)$ is a singleton. This shows that $B$ is continuous at $\gamma$.
	
	Finally, it is clear that $(3) \implies (1)$.
\end{proof}

The relation obtained in Lemma \ref{le:strictinequality} can be exploited further. Indeed, trying to carry it to the boundary in a more delicate manner we can see that the fibers are path connected.

\begin{proposition}\label{pr:pathconnected}
	Let $\gamma \in \vbd{X}$. For any $\xi\in \fibermap^{-1}(\gamma)$ there exists a continuous path from $B(\gamma)$ to $\xi$ contained in $\fibermap^{-1}(\gamma)$.
\end{proposition}
\begin{proof}
	Take a sequence $(x_n)\subset  X$ converging to $\xi$ in the horofunction compactification, and write $x_n=\gamma_n(u_n)$.
	As we have seen in the proof of Proposition \ref{pr:lowerbound}, we can take a sequence $(l_n)\subset \mathbb{R}$ with $\gamma_n(l_n)$ converging to $B_\gamma$ such that $l_n<u_n$ for all $n$.
	For each $n$ we have a path $\tilde{\alpha}^n(t)$ connecting $\gamma_n(l_n)$ and $\gamma_n(u_n)$ by setting $\tilde{\alpha}^n(t)=\gamma_n(tu_n+(1-t)l_n)$ for $t\in[0,1]$.
	We would like to carry this path to the limit, getting a path between $\xi$ and $B(q)$. However, directly taking such a limit might result in some discontinuities, so we have to choose a parametrization carefully. 
	
	To find a good parametrization we shall use a certain functional as a control. We want the functional to carry discontinuities and strict increases in the path of functions to discontinuities and strict increases in the value of the functional. Since $X$ is proper, it is separable, so let $(p_i)_{i\in \N}$ be a countable dense set in $X$. We define the functional $I:\hbc{X}\to \R$ given by 
	\[I(f)=\sum_{i\in\N} \frac{f(p_i)}{2^i d(b,p_i)}.\]
	Since $|f(x)|\le d(b,x)$ for all $f\in\hbc{X}$, the summation in the definition of $I(f)$ is absolutely convergent, so $I(f)$ is defined, finite, continuous with respect to $f$, and for any two $f,g\in \hbc{X}$ we have $I(f+g)=I(f)+I(g)$. Furthermore, since $(p_n)$ is dense and we are taking continuous functions, we have that the functional translates strict inequalities. That is, $f>g$ implies $I(f)>I(g)$. Hence, if $I(f)=0$ and $f\ge 0$ we have $f=0$.
	
	We define then the function $F_n(t)=I(h(\gamma_n(t))$. By continuity of $I$ this function is continuous, and by Lemma \ref{le:strictinequality} it is strictly decreasing with respect to $t$. That is, we have continuous strictly decreasing functions $F_n:[l_n,u_n]\to [F_n(u_n),F_n(l_n)]$. Hence, we can define implicitly the continuous parametrizations $s_n:[0,1]\to [l_n,u_n]$ by taking the unique value $s_n(t)$ such that
	\[F_n(s_n(t))=(1-t)F_n(l_n) + t F_n(u_n).\]
	Denote the $F_n(s_n(t))$ as $E_n(t)$. By the continuity of $I$ we have that $E_n(t)$ converges to $(1-t)I(B_\gamma) + t I(\xi)$ as $n\to\infty$, which we denote $E(t)$.
	
	Take now a countable dense set $(t^k)_{k\in \N}\subset [0,1]$ containing $0$ and $1$. We are now ready to start defining the path $\alpha:[0,1]\to \fibermap^{-1}(\gamma)$, and we begin defining it for the dense set $(t^k)$.
	For $k=1$ we define $\alpha(t^1)$ as an accumulation point of $h(\gamma_n(s_n(t^1)))$. 
	Denote $(\gamma_{m^1(n)})$ the subsequence of $\gamma_n$ such that $h(\gamma_{m^1(n)}(s_{m^1(n)}(t^1)))$ converges to $\alpha(t^1)$. 
	Define inductively $\alpha(t^k)$ and $(\gamma_{m^k(n)})$ by taking an accumulation point and a corresponding converging subsequence of $h(\gamma_{m^{k-1}(n)}(s_{m^{k-1}(n)}(t^k)))$. By the continuity of $I$ we have \[I(\alpha(t^k))=\lim_{n\to\infty} (F_{m^k(n)}(s_{m^k(n)}(t^k)))=E(t^k).\]
			
	For each pair $i>j$ we have that $m^i(n)$ is a subsequence of $m^j(n)$, so $h(\gamma_{m^i(n)}(s_{m^i(n)}(t^j)))$ converges to $\alpha(t^j)$. 
	Assume $t^i>t^j$. 
	By Lemma \ref{le:strictinequality} we have that $h(\gamma_{m^i(n)}(s_{m^i(n)}(t^i)))<h(\gamma_{m^i(n)}(s_{m^i(n)}(t^j)))$, so $\alpha(t^i)\le \alpha(t^j)$. 
	
	We now have to prove that the definition we have given for $\alpha$ on $(t^k)$ can be extended continuously to $[0,1]$. Fix any $t\notin (t^k)$ and take a subsequence of $t^k$, labeled $t^{k_n}$, such that $t^{k_n}\to t$. 
	We shall now see that $\alpha(t^{k_n})$ converges to a function which does not depend on the chosen subsequence, and define $\alpha(t)$ as that limit. We can split and reorder the sequence $(t^{k_n})$ into $(t^+_n)$ and $(t^-_n)$ satisfying $t^+_n>t^+_{n+1}>t>t^-_{n+1}>t^-_n$. 
	The associated $\alpha(t^\pm_n)$ are ordered, so for any $p\in X$ the sequence $\alpha(t^\pm_n)(p)$ is an increasing (or decreasing) sequence of of values in $\R$, bounded above (or below) by $\alpha(0)(p)$ (or $\alpha(1)(p)$). Hence, both sequences converge pointwise, which implies uniform convergence on compact sets, as these functions are 1-Liptschitz. Furthermore, these limits do not depend on the chosen sequence, since if we had any other we could intercalate them and the sequences would still converge. Denote then $\alpha^+$ the limit associated to $t^+_n$, and $\alpha^-$ the limit associated to $t^-_n$. Since $\alpha(t^+_n)<\alpha(t^-_m)$ for all $n,m$ we have $\alpha^+\le \alpha^-$. For each $\alpha(t^k)$ we have $I(\alpha(t^k)) = E(t^k)$. Hence by the continuity of $I$ we have that 
	\[ I(\alpha^+)= E(t)=I(\alpha^-).\]
	That is, we have
	\[I(\alpha^--\alpha^+)=0.\]
	Since $\alpha^-$ and $\alpha^+$ are continuous and $\alpha^--\alpha^+\ge 0$ we have $\alpha^-=\alpha^+.$ We thus define $\alpha(t)$ to be either one. The same reasoning shows that $\alpha$ is continuous.
\end{proof}

We would like to remark that several choices where made in the proof of the previous lemma, and the obtained path may not be unique.

We can use the previous result to observe that the horoboundary is connected if and only if the visual boundary is connected.
\begin{proof}[Proof of Proposition \ref{pr:finslerconnected}]
	Assume that the visual boundary is not connected. Then we have $U, V\subset \vbd{X}$ nonempty and open such that $U\cap V=\emptyset$ and $U\cup V=\vbd{X}$. As $\fibermap$ is continuous, the sets $\fibermap^{-1}(U)$ and $\fibermap^{-1}(V)$ are open, so the horoboundary is not connected.
	
	For the other implication, assume that the visual boundary is connected while the horoboundary is not connected. Then we have $U, V \subset \hbd{X}$ nonempty and open such that $U\cap V=\emptyset$ and $U\cup V=\hbd{X}$. Since fibers are path connected, each of them is contained in only one of $U$ or $V$, so $\fibermap(U)$ and $\fibermap(V)$ are disjoint. Since $U\cup V=\hbd{X}$ we have $\fibermap(U)\cup \fibermap(V)=\vbd{X}$, and since both $U$ and $V$ are nonempty, so are the images. Hence, both images cannot be open at the same time, as $\vbd{X}$ is connected. Therefore, these sets cannot be both closed. Assume $\fibermap(U)$ is not closed. We then have a sequence $(\gamma_n)\subset \fibermap(U)$ converging to a point in $\fibermap(V)$. Again, since $U\cup V=\hbd{X}$, we have that $U=\fibermap^{-1}\fibermap(U)$ and $V=\fibermap^{-1}\fibermap(V)$. Hence, any lift of the sequence $(\gamma_n)$ to $\fibermap^{-1}\fibermap(U)$ is contained in $U$ and, since $\hbd{X}$ is compact, has accumulation points which, by the continuity of the projection map, are be contained in $\fibermap^{-1}\fibermap(V)=V$. Hence, $U$ is not closed and we get a contradiction.
\end{proof}

\subsection{An alternative definition of the horofunction compactification}
\label{se:alternativehorofunctiondefinition}
Under what a priori seem to be more restrictive hypotheses on the space $X$ it is possible to characterize the horofunction compactification as a subset of the product of all of its visual compactifications. We detail the construction in this section.

The new extra hypotheses are both related to the differentiability of the distance function.
We say a that a uniquely geodesic metric space $X$ is \emph{$C^1$ along geodesics} if given a point $p\in X$ and a geodesic segment $\gamma$ that does not intersect $p$, the distance function $d(\gamma(t),p)$ is first differentiable and the value of the derivative depends continuously on both $t$ and $p$. Furthermore, the space $X$ has \emph{constant distance variation} if for any two distinct geodesics $\gamma,\eta$ with $\gamma(0)=\eta(0)$ we have either
\begin{equation}\label{eq:derivative}
		\left.\frac{d}{dt}d(\gamma(t),\eta(s))\right\vert_{t=0}=\left.\frac{d}{dt}d(\gamma(t),\eta(1))\right\vert_{t=0}
\end{equation}
for all $s>0$, or $\left.\frac{d}{dt}d(\gamma(t),\eta(s))\right\vert_{t=0}$ does not exist for any $s>0$.

Many commonly studied metric spaces have constant distance variation. For example, spaces with bounded curvature, either above of below, have constant distance variation, as explained in the book by Burago--Burago--Ivanov \cite[Section 4]{CourseMetricGeometryBook}. Importantly to our case, Teichmüller spaces with the Teichmüller distance satisfy both hypotheses. Earle \cite{EarleDifferentiable} proved that the distance function is $C^1$ by providing a formula for its derivative. Applying the formula to \eqref{eq:derivative} we get that the derivative depends only on the tangential vector to $\gamma$ at $0$ and the unit area quadratic differential associated to $\eta$ at $0$, so we also have constant distance variation. Furthermore, Teichmüller spaces with the Teichmüller distance are also straight and proper, so the results from this section can be applied to them.

Consider the product of all the possible visual compactifications obtained by changing the basepoint,
\[
	E=\prod_{b\in X} \vbc{X}_b,
\]
with the usual product topology. See the book by Munkres \cite[Chapters 2.19 and 5.37]{Munkres} for some background on infinite products of topological spaces. Denote $\pi_b$ the projection from $E$ to $\vbc{X}_b$. By definition of the product topology, the diagonal inclusion $i:X\hookrightarrow E$ such that by $\pi_b(i(x))=x$ for every $x,b\in X$ is continuous, and has continuous inverse restricted to $i(X)$ given by $\pi_b$. Hence, $i(X)$ is homeomorphic to $X$. That is, $i$ is an embedding. Furthermore, by Tychonoff's theorem the product is compact, as each factor of the product is compact. Hence the closure $\overline{i(X)}$, which we shall denote $\Vbc{X}$, is compact. The pair $(i,\Vbc{X})$ is then a compactification of $X$, which tracks the information given by the visual boundary at each point. That is, a sequence in $X$ converges in the topology of $\Vbc{X}$ if and only if it converges for every possible visual compactification $\vbc{X}_b$. The main interest of this compactification comes from the following result.

\begin{theorem}\label{th:horocompisvisualfromeverypoint}
	Let $X$ be a proper, uniquely geodesic, straight metric space which is $C^1$ along geodesics and has constant distance variation. Then $(i,\Vbc{X})$ is isomorphic to $(h,\hbc{X})$.
\end{theorem} 

Denote $\fibermap_b$ the continuous map from $\hbc{X}$ to $\vbc{X}_b$ given by \cref{th:projectionfunction}. The isomorphism between $\hbc{X}$ and $\Vbc{X}$ is defined by recording the value of each possible $\fibermap_b$ within $\Vbc{X}$. That is, we define $\bigfibermap:\hbc{X}\to\Vbc{X}$ in such a way that $\pi_b\circ\bigfibermap:=\fibermap_b$ for each $b\in X$. The only property required to prove that $\bigfibermap$ is an isomorphism not following directly from previous results is the injectivity. By Proposition \ref{pr:optimalpath} we know that if $f\in\fibermap_b^{-1}(\gamma)$ then $\gamma$ is an optimal geodesic of $f$. That is, $f(\gamma(t))-f(\gamma(s))=-(t-s)$. Hence, if $f,g\in\fibermap_b^{-1}(\gamma)$, then they differ by a constant along the geodesic $\gamma$. If $f$ and $g$ are horofunctions in the preimage of a point by $\bigfibermap$, then they differ by a constant along infinitely many geodesics, which cover $X$. However, the constant might depend on the geodesic, so we need a way to connect these constants. We proceed by strengthening Proposition \ref{pr:optimalpath} to show that any two functions in $\fibermap_b^{-1}(\gamma)$ also have the same directional derivatives at points in $\gamma$, which allows us to connect the geodesics.
Precisely, we prove the following.
\begin{proposition}\label{pr:firstderivativehorofunction}
	Let $X$ be a proper, uniquely geodesic, straight metric space which is $C^1$ along geodesics and has constant distance variation. Furthermore, let $\gamma$ be a geodesic ray starting at $b$, and let $\alpha$ be a geodesic starting at some point on $\gamma$. Then, $\left.\frac{d}{dt} f\circ\alpha (t) \right\vert_{t=0}$ exists and its value is the same for all $f\in \fibermap^{-1}_b(\gamma)$.
\end{proposition}
\begin{proof}
	For any $b'\in\gamma$ we have that $\gamma$ is an optimal geodesic of $f$ passing through $b'$. Denoting $\gamma_{b'}$ the geodesic ray starting at $b'$ with the same bi-infinite extension as $\gamma$ we have that $f\in\fibermap_{b'}^{-1}(\gamma_{b'})$, by Proposition \ref{pr:optimalpath}. Hence, we can assume that $\alpha(0)=b$ by changing the basepoint if necessary. Let $x_n$ be a sequence converging to $f$. Furthermore, let $\eta^n_t$ be the geodesic from $\alpha(t)$ to $x_n$ and $g_n(t)$ be the value of $\left.\frac{d}{ds}h(x_n)\circ \alpha(s)\right\vert_{s=t}$. By the definition of the map $h$ we have $g_n(t)=\left.\frac{d}{ds}d(\alpha(s),x_n)\right\vert_{s=t}$. By the constant distant variation we have $g_n(t)=\left.\frac{d}{ds}d(\alpha(s),\eta^n_t(1))\right\vert_{s=t}$, which since $X$ is $C^1$ along geodesics depends continuously on $\eta^n_t(1)$ and $t$. 
	
	By Proposition \ref{pr:welldefined} the geodesics $\eta^n_t$ converge as $n\to\infty$ to some geodesics $\eta_t$, so $\eta^n_t(1)$ converges to $\eta_t(1)$. Since the space is $C^1$ along geodesics, the value of $\left.\frac{d}{ds}d(\alpha(s),\eta^n_t(1))\right\vert_{s=t}$ depends continuously on $\eta^n_t(1)$, and so $g_n$ converges pointwise to $g(t)=\left.\frac{d}{ds}d(\alpha(s),\eta_t(1))\right\vert_{s=t}$. 
	
	Take some $\delta>0$ and assume the convergence is not uniform on $[-\delta,\delta]$. Then there is some $\varepsilon>0$ such that for each $n$ there is at least one $t_n\in [-\delta,\delta]$ such that $|g_n(t_n)-g(t_n)|>\varepsilon$. Since $[-\delta,\delta]$ is compact we can take a converging subsequence such that $t_n$ converges to some $T\in[-\delta,\delta]$. Hence, the point $\eta_{t_n}^n(1)$ does not converge to $\eta_T(1)$, so by properness of $X$ we can take a subsequence such that $\eta_{t_n}^n(1)$ converges to some $p\in X$ different from $\eta_T(1)$. Let $\beta$ be the geodesic starting at $\alpha(T)$ passing through $p$. The geodesics $\eta^n_{t_n}$ converge uniformly to $\beta$, and $\beta\neq\eta_T$. For any fixed $t>0$ we have, following the same reasoning than in the proof of Proposition \ref{pr:welldefined},
	\[
		f(\beta(t))-f(\beta(0))=\lim_{n\to\infty}d(x_n,\beta(t))-d(x_n,\beta(0))=-t.
	\]
	Hence, $\beta$ is an optimal geodesic of $f$ passing through $\alpha(T)$. However, $f\in \fibermap^{-1}_{\alpha(T)}(\eta_T)$, so $\eta_T$ is also an optimal geodesic passing through $\alpha(T)$, contradicting Proposition \ref{pr:optimalpath}.
	
	Hence, the convergence of $(h(x_n)\circ\alpha)'=g_n$ to $g$ is uniform on $[-\delta,\delta]$. Therefore, $f$ is differentiable and $f'(0)=g(0)=\left.\frac{d}{ds}d(\alpha(s),\gamma(1))\right\vert_{s=0}$, which is the same for all $f\in\fibermap^{-1}(\gamma)$.
\end{proof}

\begin{proof}[Proof of Theorem \ref{th:horocompisvisualfromeverypoint}]
	Each $\fibermap_b$ is continuous, so by the definition of the product topology the map $\bigfibermap$ is continuous. Hence, by \cref{le:walshcompact} to see that $\bigfibermap$ is an isomorphism it is enough to show that $\bigfibermap$ is injective.

	Let $f,g\in \hbc{X}$ be such that $\bigfibermap(f)=\bigfibermap(g)$. If there is some $b\in X$ such that $\pi_b\circ \bigfibermap(f)\in X$ then $f=h(\pi_b\circ \bigfibermap(f))=g$. Assume then $\pi_b\circ \bigfibermap(f)\in\vbd{X}_b$ for all $b\in X$. By Proposition \ref{pr:firstderivativehorofunction} they have the same directional derivatives at every point. Let $\alpha$ be a geodesic from a fixed basepoint $b$ to any other point. We have $(f\circ \alpha)'=(g\circ\alpha)'$, so $f-g$ is constant along $\alpha$, and hence everywhere, since any point can be connected to $b$ by a geodesic. Hence, $f$ and $g$ are the same horofunctions.
\end{proof}

By the definition of the convergence in the product topology, this characterization gives us the following equivalence for the convergence to points in the horoboundary.
\begin{corollary}\label{co:convergingtohoriffconvergingeveryvis}
	Let $X$ be a proper, uniquely geodesic, straight metric space, $C^1$ along geodesics and with constant distance variation. 
	A sequence $(x_n)\subset X$ converges in the horofunction compactification if and only if the sequence converges in all the visual compactifications.
\end{corollary}
Restricting the result to the Teichmüller metric we get \cref{co:convergingtohoriffconvergingeveryvisintro} announced in the introduction.

\section{Background on Teichmüller spaces}\label{se:backgroundteichmuller}
	
	A \emph{surface with marked points} $S$ is a pair $(\Sigma,P)$, where $\Sigma$ is a compact, orientable surface with possibly empty boundary, and $P\subset \Sigma$ is a finite, possibly empty, set of points, where we allow points to be on the boundary. The \emph{Teichmüller space} $\T(S)$ is the set of equivalence classes of pairs $(X,f)$ where $X$ is a Riemann surface and $f:\Sigma \to X$ is an orientation-preserving homeomorphism. Two pairs $(X,f)$ and $(Y,g)$ are equivalent if there is a conformal diffeomorphism $h:X\to Y$ such that $g^{-1}\circ h \circ f$ is isotopic to identity rel $P$.
	
	The \emph{Teichmüller distance} between two points $[(X,f)],[(Y,g)]\in \T(S)$ is defined as the value $\frac{1}{2} \log \inf K$, where the infimum is taken over all $K\ge 1$ such that there exists a $K$-quasiconformal homeomorphism $h:X\to Y$ with
	$g^{-1}\circ h \circ f$ isotopic to identity rel $P$. Together with the smooth structure provided by the Fenchel--Nielsen coordinates  $\T(S)$ satisfies all the metric properties discussed in the previous section. That is, $\T(S)$ with the Teichmüller distance is a proper, uniquely geodesic and straight metric space which is $C^1$ along geodesics and has constant distance variation. See \cite[Part 2]{primer} for some background on the Teichmüller metric and the Fenchel--Nielsen coordinates.

	A \emph{quadratic differential} on a Riemann surface $X$ is a map $q:TX\to \C$ such that $q(\lambda v)=\lambda^{2}q(v)$ for every $\lambda \in \C$ and $v\in TX$. Considering only holomorphic quadratic differentials with finite area $\int_X |q|$ we get a characterization of the cotangent space to the Teichmüller space based at $[(X,f)]$.
	Given a point $p\in \T(S)$ and a quadratic differential $q\in T_p^*\T(S)$ there is a unique geodesic $\gamma$ such that $\gamma(0)=p$ and $\gamma'(0)=|q|/q$. We shall denote such a geodesic as $\tray{q}{\cdot}$ and denote the associated Busemann points as $B(q)$ or $B_q$.
	 
 \subsection{Measured foliations}\label{se:measuredfoliations}
	A \emph{multicurve} on $S$ is an embedded $1$-dimensional submanifold of $\Sigma \backslash P$ with boundary in $\partial \Sigma \backslash P$ such that 
	\begin{itemize}
		\item no circle component bounds a disk with at most 1 marked point;
		\item no arc component bounds a disk with no interior marked points and at most 1 marked point on $\partial \Sigma$ and
		\item no two components are isotopic to each other in $\Sigma$ rel $P$.
	\end{itemize}
	Each of the components is called \emph{curve.} A \emph{weighted multicurve} is a multicurve together with a positive weight associated to each curve. We shall consider (weighted) multicurves up to isotopy rel $P$. If a simple curve is a circle we shall denote it \emph{closed curve}, and \emph{proper arc} otherwise.
 	
	A \emph{measured foliation} on $S$ is a foliation with isolated prong singularities, where we allow 1-prong singularities at marked points, equipped with an invariant transverse measure $\mu_F$ \cite[Exposé 5]{FLP}. Denoting $\alpha_i$ and $w_i$ the components and the weights of $\alpha$ respectively, the intersection number $i(\alpha,F)$ is defined as $\inf\sum_i w_i \int_{\alpha_i} |\mu_F|d\alpha_i$, where the infimum is taken over all representatives of $\alpha$. Two measured foliations $F$ and $G$ are \emph{equivalent} if $i(\alpha,F)=i(\alpha,G)$ for every multicurve $\alpha$. We shall always consider measured foliations up to this equivalence relation. The set of measured foliations is usually denoted as $\MF$, and its topology is defined in such a way that a sequence $(F_n)\subset \MF$ converges to $F$ if and only if $i(\alpha,F_n)$ converges to $i(\alpha,F)$ for every multicurve $\alpha$.
	
	Given a quadratic differential one can define the \emph{vertical foliation} as the union of \emph{vertical trajectories}, that is, maximal smooth paths $\gamma$ such that $q(\gamma'(t))<0$ for every $t$ in the interior of the domain. 
	This foliation can be equipped with the transverse measure given by $|\operatorname{Re} \sqrt{q}|$.
	This measured foliation is called the \emph{vertical measured foliation} of $q$, and shall be denoted as $V(q)$.
	This map is actually a homeomorphism.
	As such, given a measured foliation $F$ and a complex structure $X$ there is a unique quadratic differential $q_{F,X}$ on $X$ such that $V(q_{F,X})=F$. We call this quadratic differential the \emph{Hubbard--Masur} differential associated to $F$ on $X$ \cite{Hubbard}.
	Furthermore, for each $\lambda>0$ we have $q_{\lambda F,X}=\lambda q_{F,X}$. Similarly, the \emph{horizontal foliation} $H(q)$ can be defined as the union of maximal smooth paths $\gamma$ such that $q(\gamma'(t))>0$, with the transverse measure $|\operatorname{Im}\sqrt{q}|$.
	
	It is possible to associate a measured foliation to each weighted multicurve by thickening each proper arc and closed curve to a rectangle or cylinder respectively with width equal to the weight of the curve, and then collapsing the rest of the surface.
	The intersection numbers are maintained by this construction. 
	This association is injective, and hence we shall consider the set of weighted multicurves as a subset of the measured foliations, and use both expressions of weighted multicurve indistinctly.
	
	By removing the critical graph, a measured foliation is decomposed into a finite number of connected components, each of which is either a thickened curve, or a minimal component which does not intersect the boundary, in which every leaf is dense \cite[Chapter 24.3]{Strebel}. Each transverse measure within the minimal components can be further decomposed into a sum of finitely many projectively distinct ergodic measures. A foliation $F'$ is an indecomposable component of $F$ if it is either a thickened curve or a minimal component with a transverse measure that cannot be decomposed as a sum of more than one projectively distinct ergodic measure. Every foliation can be decomposed uniquely into a union of indecomposable foliations. For a surface of genus $g$ with no boundaries nor marked points Papadopoulos shows \cite{Papadopoulos} that the maximum number of indecomposable components for any foliation is $3g-3$. It is possible to get an upper bound  for foliations on surfaces with boundary and marked points by swapping the marked points for boundaries and using the doubling trick we will explain in \cref{se:doublingtrick}.

	It was shown by Thurston that for surfaces without boundary it is possible to achieve a dense subset by restricting to simple closed curves, see Fathi--Laudenbach--Poénaru \cite{FLP} for a reference. When there are boundaries the picture gets slightly more complicated, but it has been shown by Kahn, Pilgrim and Thurston in \cite[Proposition 2.12]{kahn} that multicurves can be seen as a dense subset. More precisely, they show the following.
	 
	\begin{proposition}[Kahn--Pilgrim--Thurston]\label{pr:curvesdense}
		Let $F$ be a measured foliation in $S$ not containing proper arcs. Then there exists a sequence of multicurves composed solely of closed curves approaching $F$.
	\end{proposition}
	
	The result can be extended to any foliation by cutting along the proper arcs and approaching the foliation in the resulting surfaces by multicurves. Then, joining the multicurves from the proposition with the proper arcs and the adequate weights we get a sequence of multicurves converging to our original foliation.
 
 \subsection{Extremal length}

	Given a marked conformal structure on $S$, that is, a point $X\in \T$, the \emph{extremal length} of $F$ on $X$ is defined as 
	\[\ext_X(F):=\int_X |q_{F,X}|.\]
	The map $\ext:\MF(S)\times \T(S)\to \R$ is continuous and homogeneous of degree 2 in the first variable.

	Given two points $x,y\in\T(S)$ we can define the function
	\[
		K_{x,y}:=\sup_{F\in P_b} \frac{\ext_x(F)}{\ext_y(F)},
	\]
	where $P_b$ is the set of measured foliations $F$ satisfying $\ext_b(F)=1$. As revealed by Kerckhoff's formula \cite{Kerckhoff}, the value $(1/2)\log K_{x,y}$ coincides with the usual definition of the Teichmüller distance $d(x,y)$.

\subsection{The doubling trick}\label{se:doublingtrick}
	
	Let $X$ be a Riemann surface with nonempty boundary. Denote by $\overline{X}$ the mirror surface, obtained by composing each atlas of $X$ with the complex conjugation. Gluing $X$ to $\overline{X}$ along the corresponding boundary components we obtain the \emph{conformal double} $X^d=X\cup \overline{X}/\sim$ of $X$. Note that $X^d$ has empty boundary. Given a foliation $F$ or a quadratic differential $q$ on $X$, we can repeat the same process, obtaining the corresponding conformal doubles $F^d$ and $q^d$ on $X^d$. For a more detailed treatment of this argument see \cite[Section II.1.5]{Abikoff}.
	
	The main interest of the conformal doubles is that these are surfaces without boundary, so most of the results relating to Teichmüller theory of surfaces without boundary can be translated to surfaces with boundary. We have the following.
	
	\begin{proposition}
		Let $X$ be a Riemann surface with boundary, and $F$ be a foliation on $X$. Then, 
		\[
			\ext_{X^d}(F^d)=2 \ext_X(F).
		\]
	\end{proposition}
	\begin{proof}
		We have $q_{F^d,X^d}=q_{F,X}^d$, so the result follows, as $\int_{X^d}|q_{F,X}^d|=2\int_X|q_{F^d,X^d}|$.
	\end{proof}

	\begin{figure} \centering
		
		\begin{tikzpicture}
			
			\draw[smooth] (-3,0) to  [out=90, in = 180] (0,2.5) to [out=0,in=90] (3,0);
			
			\draw[smooth] (-2,0) to  [out=90,in=90] (-0.5,0);
			
			\draw[smooth] (0.5,0) to  [out=90,in=90] (2,0);
			
			\draw(-3,0) arc(180:0:0.5 and 0.25);
			\draw(-3,0) arc(180:360:0.5 and 0.25);
			
			\draw(-0.5,0) arc(180:0:0.5 and 0.25);
			\draw(-0.5,0) arc(180:360:0.5 and 0.25);
			
			\draw(2,0) arc(180:0:0.5 and 0.25);
			\draw(2,0) arc(180:360:0.5 and 0.25);
			
			\begin{scope}[xshift=-40,yshift=50]
				\draw[smooth] (0.4,0.1) .. controls (0.8,-0.25) and (1.2,-0.25) .. (1.6,0.1);
				\draw[smooth] (0.5,0) .. controls (0.8,0.2) and (1.2,0.2) .. (1.5,0);
				\draw  (1,0) ellipse (0.8 and 0.4);
			\end{scope}
			
			\draw[smooth] (0,0.25) to  [out=90, in = 90] (-2.5,0.25);
			
			\begin{scope}[xshift=40,yshift=35,rotate=270]
				\draw(0.5,0.5) .. controls (-0.2,0.1) .. (-0.5,-0.5); 
				\draw(-0.5,0.5) .. controls (-0.2,0.1) .. (0.5,-0.5); 
				\draw(0.5,0.4) .. controls (-0.1,0.1) .. (0.5,-0.4); 
				\draw(-0.4,0.5) .. controls (-0.2,0.2) .. (0.4,0.5); 
				\draw(-0.5,0.4) .. controls (-0.25,0.1) .. (-0.5,-0.4); 
				\draw(-0.4,-0.5) .. controls (-0.2,0.1) .. (0.4,-0.5); 
			\end{scope}
			
			\begin{scope}[xscale=1,yscale=-1,yshift=20]
				\draw[smooth] (-3,0) to  [out=90, in = 180] (0,2.5) to [out=0,in=90] (3,0);
				
				\draw[smooth] (-2,0) to  [out=90,in=90] (-0.5,0);
				
				\draw[smooth] (0.5,0) to  [out=90,in=90] (2,0);
				
				\draw[dashed](-3,0) arc(180:0:0.5 and 0.25);
				\draw(-3,0) arc(180:360:0.5 and 0.25);
				
				\draw[dashed](-0.5,0) arc(180:0:0.5 and 0.25);
				\draw(-0.5,0) arc(180:360:0.5 and 0.25);
				
				\draw[dashed](2,0) arc(180:0:0.5 and 0.25);
				\draw(2,0) arc(180:360:0.5 and 0.25);
				
				\begin{scope}[xshift=-40,yshift=50,yscale=-1]
					\draw[smooth] (0.4,0.1) .. controls (0.8,-0.25) and (1.2,-0.25) .. (1.6,0.1);
					\draw[smooth] (0.5,0) .. controls (0.8,0.2) and (1.2,0.2) .. (1.5,0);
					\draw  (1,0) ellipse (0.8 and 0.4);
				\end{scope}
				
				\draw[smooth] (0,-0.25) to (0,0.25) to  [out=90, in = 90] (-2.5,0.25) to (-2.5,-0.25);
				
				\begin{scope}[xshift=40,yshift=35,rotate=270]
					\draw(0.5,0.5) .. controls (-0.2,0.1) .. (-0.5,-0.5); 
					\draw(-0.5,0.5) .. controls (-0.2,0.1) .. (0.5,-0.5); 
					\draw(0.5,0.4) .. controls (-0.1,0.1) .. (0.5,-0.4); 
					\draw(-0.4,0.5) .. controls (-0.2,0.2) .. (0.4,0.5); 
					\draw(-0.5,0.4) .. controls (-0.25,0.1) .. (-0.5,-0.4); 
					\draw(-0.4,-0.5) .. controls (-0.2,0.1) .. (0.4,-0.5); 
				\end{scope}
				
			\end{scope}
			
		\end{tikzpicture}
	\caption{Visual representation of the doubling trick.}
	\label{fi:doublingtrick}
	\end{figure}
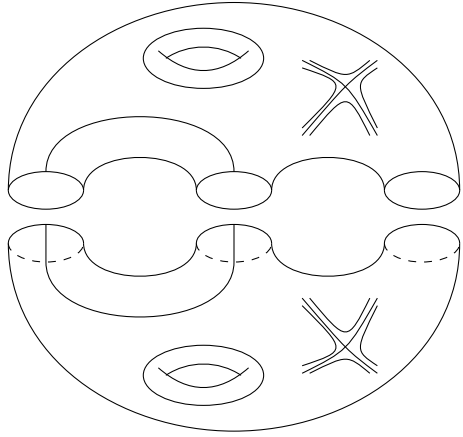
	 
\subsection{The Gardiner--Masur compactification}\label{se:GMboundaryWalshpaper}
	
	For a surface $S$ with marked points and empty boundary we can embed $\T(S)$ into the space of continuous functions from the set $\Curv$ of simple closed curves on $S$ to $\R$ via the map $\phi:\T(S)\to P(\R^\Curv)$ defined by
	\[
		\phi(X)=\left[\ext_X(\alpha)^{1/2}\right]_{\alpha\in\Curv},
	\]
	where the square brackets indicate a projective vector.
	Gardiner and Masur show \cite{Gardiner} that this map is indeed an embedding, and that $\phi(\T(S))$ is precompact. The Gardiner--Masur compactification of a surface without boundary is then defined as the pair $(\phi,\overline{\phi(\T(S))})$. 
	
	Alternatively, after choosing a basepoint $b\in \T(S)$, it is also possible to consider the map $\E:\T(S)\to C(\MF)$ defined by
	\[
		\E(X)(\cdot):=\left(\frac{\ext_X(\cdot)}{K_{b,X}}\right)^{1/2},
	\]
	This map is quite similar to the original map $\phi$, the differences being that $\E$ considers all measured foliations instead of just the closed curves, and normalizes instead of projectivizing. Walsh proves \cite{Walsh} that, for surfaces without boundary, the map $\E$ defines a compactification in the same way that $\phi$ does, and in fact this compactification is isomorphic to the one defined by $\phi$.
	
	The compactification defined by $\E$ fits better our goal, so we shall define the Gardiner--Masur compactification of Teichmüller spaces of surfaces with boundary as the one obtained by using $\E$. With this in mind, we first need the following result.

	\begin{proposition}
		Let $S$ be a compact surface with possibly boundary and marked points. Then the map $\E:\T(S)\to C(\MF)$ is injective.
	\end{proposition}
	\begin{proof}
		Assume we have $x,y\in \T(S)$ with $\E(x)(F)=\E(y)(F)$ for all $F\in \MF$. Then,
		\[
			K_{x,y}=\sup_{F\in P_b}\frac{\ext_x(F)}{\ext_y(F)}=\frac{K_{b,x}}{K_{b,y}}
		\]
		and
		\[
			K_{y,x}=\sup_{F\in P_b}\frac{\ext_y(F)}{\ext_x(F)}=\frac{K_{b,y}}{K_{b,x}}=K_{x,y}^{-1}.
		\]
		However, $K_{y,x}=K_{x,y}$, since the Teichmüller distance is symmetric. Hence, $K_{x,y}=1$ and, by Kerckhoff's formula, $d(x,y)=1/2 \log K_{x,y}=0$.
	\end{proof}
	
	Miyachi shows \cite{Miyachi3} that the set $E(S):=\{\E(X)\mid X\in\T(S)\}$ is precompact when $S$ is a surface without boundary. Given a surface with boundary $S$, denote $\MF^d(S)$ the set of measured foliations on $S^d$ obtained by doubling the foliations $\MF(S)$. 
	The set $E(S^d)|_{\MF^d(S)}=\{\E(X)|_{\MF^d(S)}\mid X\in\T(S^d)\}$, obtained by restricting the functions in $E(S^d)$ to $\MF^d$, is precompact. Furthermore, we can embed $E(S)$ into $E(S^d)|_{\MF^d(S)}$ by sending $f\in E(S)$ to $f^d\in E(S^d)|_{\MF^d(S)}$ defined by $f^d(F^d)=f(F)$. Hence, $E(S)$ is precompact.
	
	We define the \emph{Gardiner--Masur compactification} for a surface with boundary as the closure $\overline{E}$ of $E(S)$, together with the map $\E$. We shall be using the same characterization for surfaces without boundary.
	
	One of the relevant features of the Gardiner--Masur compactification is that it coincides with the horofunction compactification. Indeed, Liu--Su \cite{LiuSu} and Walsh \cite{Walsh} prove that for surfaces without boundary these two compactifications are isomorphic. In the following, we shall extend the relevant results to surfaces with boundary. We begin with the driving theorem from Walsh's paper.
	
	\begin{theorem}[Extension of {\cite[Theorem 1]{Walsh}} to surfaces with boundary]
		Let $\tray{q}{\cdot}:\R_+\to\T(S)$ be the Teichmüller ray with initial unit-area quadratic differential $q$, and let $F$ be a measured foliation. Then,
		\[
		\lim_{t\to\infty} e^{-2t}\ext_{\tray{q}{t}}(F)=\sum_j\frac{i(G_j,F)^2}{i(G_j,H(q))},	
		\]
		where the $\{G_j\}$ are the indecomposable components of the vertical foliation $V(q)$, and $H(q)$ is the horizontal foliation.
	\end{theorem}
	\begin{proof}
		If $S$ does not have boundary the result follows from Walsh's paper. Assume then that $S$ has boundary. Let $p$ be the number of proper arcs of $V(q)$, and reorder the components so $G_j$ is a proper arc for $j\le p$. The conformal double $G_j^d$ is indecomposable whenever $G_j$ is a proper arc, and decomposes into two components otherwise, as it is not incident to the boundary of $S$. Denote $G^1_j$ and $G^2_j$ the two components of $G_j$ for $j>p$. We have 
		\begin{multline*}
			2\lim_{t\to\infty} e^{-2t}\ext_{\tray{q}{t}}(F)=\lim_{t\to\infty} e^{-2t}\ext_{\tray{q^d}{t}}(F^d)\\=
			\sum_{j\le p}\frac{i(G^d_j,F^d)^2}{i(G^d_j,H(q)^d)}
			+\sum_{i\in \{1,2\}}\sum_{j> 		p}\frac{i(G^i_j,F^d)^2}{i(G^i_j,H(q)^d)}.
		\end{multline*}
		
		For foliations $G,F\in \MF(S)$ we have $i(G^d,F^d)=2i(G,F)$. Hence, $i(G^d_j,F^d)=2i(G_j,F)$. Using the symmetry, $i(G^1_j,F^d)=i(G^2_j,F^d)$, so for $j>p$ we have $i(G^1_j,F^d)=i(G_j,F)$. Using these identities we get the result.
	\end{proof}
	Following the same reasoning we can extend as well the next result.
	\begin{lemma}[Extension of {\cite[Lemma 3]{Walsh}} to surfaces with boundary]\label{le:walshlowerbound}
		Let $q$ be a unit area quadratic differential. Then,
		\[
			e^{-2t}\ext_{\tray{q}{t}}(F)\ge \sum_j\frac{i(G_j,F)^2}{i(G_j,H(q))},
		\]
		where $t\in \R_+$ and $\{G_j\}$ are the indecomposable components of the vertical foliation $V(q)$.
	\end{lemma}
	
	Most of the results in Walsh's paper use the previous theorem. In particular, we have the following.
	\begin{corollary}[Extension of {\cite[Corollary 1]{Walsh}} to surfaces with boundary]\label{co:walshbusemanshape}
		Let $q$ be a quadratic differential and denote by $G_j$ the indecomposable components of its vertical foliation. Then, the Teichmüller ray $\tray{q}{\cdot}$ converges in the Gardiner--Masur compactification to
		\[\left(\sum_j \frac{i(G_j,\cdot)^2}{i(G_j,H(q)}\right)^{1/2}.\]
	\end{corollary}
	The relation between the Gardiner--Masur compactification and the horoboundary compactification is given by the map $\Xi:\overline{E}\to \hbc{\T(S)}$ defined by 
	\[
		\Xi(f)(x):=\frac{1}{2}\log\sup_{F\in \mathcal{P}}\frac{f(F)^2}{\ext_x(F)}.
	\]
	
	The following result can be extended to surfaces with boundary by repeating the proof found in Walsh's paper in this context.
	\begin{theorem}[Extension of {\cite[Lemma 21]{Walsh}} to surfaces with boundary]
		The map $\Xi$ is an isomorphism between the compactifications $(\E,\overline{E})$ and $(h,\hbc{\T(S)})$.
	\end{theorem}

	Directly from the definition of $\Xi$ we have the following
	\begin{corollary}\label{co:orderpreserved}
		Let $f,g\in \overline{E}$. If $f\ge g$ then $\Xi(f)\ge \Xi(g)$.
	\end{corollary}
			
	We shall denote the representation of the Busemann point $B(q)$ in the Gardiner--Masur compactification as $\E(q)$. By \cref{co:walshbusemanshape} we have an explicit representation of $\E(q)$.
	As we have seen in \cref{pr:horobocompfiner,pr:continuityatqintro}, the continuity of the Busemann map has some interesting implications, and it is enough to look for continuity of the map restricted to the boundary. Related to this question we have the following result, which can also be derived by the same proof found in Walsh's paper, applied to this context.
	\begin{theorem}[Extension of {\cite[Theorem 10]{Walsh}} to surfaces with boundary]\label{th:Busemanncontinuityifstrong}
		Let $(q_n)$ be a sequence of quadratic differentials based at $b\in \T(S).$ Then $B(q_n)$ converges to $B(q)$ if and only if both of the following hold:
		\begin{enumerate}
			\item $(q_n)$ converges to $q$;
			\item for every subsequence $(G^n)_n$ of indecomposable elements of $\MF$ such that, for each $n\in \N$, $G^n$ is a component of $V(q_n)$, we have that every limit point of $G^n$ is indecomposable.
		\end{enumerate}
	\end{theorem}
	
	In view of this theorem, we say that a sequence of quadratic differentials $(q_n)$ converges \emph{strongly} to $q$ if it does so in the sense described by the theorem. 
	
	Finally, while the following result may be extendable to surfaces with boundary, we only use it in the context of surfaces without boundary, so we shall not be working on finding an extension.
	\begin{theorem}[{\cite[Theorem 3]{Walsh}}]
		For the Teichmüller space of a surface without boundary with the Teichmüller metric,
		for any basepoint $X\in \T(S)$, all Busemann points can be expressed as $B(q)$ for some quadratic differential $q$ based at $X$.
	\end{theorem}

\section{Horoboundary convergence for Teichmüller spaces}\label{se:continuitybusemanteich}
\subsection{Continuity of the Busemann map}
 We begin by using \cref{pr:continuityatqintro} to determine when the Busemann map is continuous. Recall that a sequence $(q_n)$ converges to $q$ strongly if and only if the sequence satisfies the conditions of Theorem \ref{th:Busemanncontinuityifstrong}. That is, a sequence $(q_n)$ converges to $q$ strongly if and only if the associated Busemann points $B(q_n)$ converge to $B(q)$. With this in mind we introduce the following notion.
	
	\begin{definition}
		Let $q$ be a quadratic differential. We say that $q$ is \emph{infusible} if any sequence of quadratic differentials converging to $q$ converges strongly. We say that $q$ is \emph{fusible} if it is not infusible.
	\end{definition}
	
	In other words, we say that $q$ is fusible when it can be approached by a sequence of quadratic differentials $(q_n)$ such that there is some sequence $(G^n)$ of measured foliations with each $G^n$ being an indecomposable component of $V(q_n)$, with $(G^n)$ having at least one decomposable accumulation point. The following statement follows directly from this definition, \cref{pr:continuityatqintro} and Walsh's result.

	\begin{proposition}
		Let $q$ be a unit area quadratic differential. The Busemann map $\bussmap$ is continuous at $q$ if and only if $q$ is infusible.
	\end{proposition}
	\begin{proof}
		If $q$ is fusible then we have a sequence converging to $q$ but not strongly. Hence, by Theorem \ref{th:Busemanncontinuityifstrong} the sequence $(B(q_n))$ does not converge to $B(q)$, and so the Busemann map is not continuous at $q$.
		
		If $q$ is infusible we have that any sequence $(q_n)$ converging to $q$ does so strongly, and so we have that $B(q_n)$ converges to $B(q)$, so $B$ is continuous at $q$ when restricted to the boundary. By \cref{pr:continuityatqintro} this implies that $B$ is continuous at $q$.
	\end{proof}

	We shall now find a criterion on the vertical foliation to determine when a unit area quadratic differential is infusible. 
	
	\begin{definition}
		Let $F$ be a measured foliation on a surface $S$ and let $G$ be one of its indecomposable components. We say that $G$ is a \emph{boundary annulus} if it is an annulus parallel to a boundary with no marked points, and a \emph{boundary component} if it is a boundary annulus or a proper arc. If $G$ is not a boundary component, we shall call it an \emph{interior component}. 
		Each of the connected components of the surface obtained after removing the proper arcs shall be called \emph{interior part}. If each of these interior parts has at most one interior component, then we say that $F$ is \emph{internally indecomposable}.
		If $F$ is not internally indecomposable we say that it is \emph{internally decomposable}.
	\end{definition}

	For surfaces without boundary, a foliation $F$ is internally indecomposable if and only if it is indecomposable, as we do not have boundary components. Given these definitions we can state our main result of this section
	
	\begin{theorem}\label{th:maxcondition}
		Let $q$ be a quadratic differential. Then $q$ is infusible if and only if its vertical foliation $V(q)$ is internally indecomposable.
	\end{theorem}
	
	This result is somewhat straightforward whenever $S$ does not have boundary, as in order to have a sequence $(q_n)$ that converges to $q$ but not strongly we need a sequence of components of $V(q_n)$ converging to a decomposable component of $V(q)$, but if $S$ is closed and $V(q)$ is internally indecomposable, then $V(q)$ only has one indecomposable component. Conversely, if $V(q)$ has more than one indecomposable component, as $S$ does not have boundary $V(q)$ can be approached by a sequence of simple closed curves, so the associated sequence of quadratic differentials converges to $q$ but not strongly.
	
	For surfaces with boundary the proof is more involved, as simple closed curves are no longer dense. However, the density of multicurves from Proposition \ref{pr:curvesdense} allows us to follow a slightly similar reasoning. We begin by proving some results regarding the shape that foliations have to take when approaching a foliation with boundary components, namely, boundary components have to be eventually included in the approaching foliations.
	
	\begin{proposition}\label{pr:notsplittingboundary}
		Let $(F_n)$ be a sequence of measured foliations converging to a measured foliation $F$, let $G$ be the union of the boundary components of $F$ and let $H$ be such that $F=H+G$. Then, for $n$ big enough, $F_n=H_n+a_n G$, with $a_n$ converging to $1$ and $H_n$ converging to $H$.
	\end{proposition}

	In particular, the proper arcs of the limiting foliation have to be included in the approaching foliations. Hence, we will be able to separate the surface along these proper arcs into the interior parts of the limiting foliation, and study the convergence in each of these parts.

	We say that a subset of a boundary component is a \emph{boundary arc} if it is homeomorphic to an open interval or a circle, does not contain marked points and, if it is homeomorphic to an open interval, it is delimited by marked points.

	Repeating the argument by Chen--Chernov--Flores--Fortier Bourque--Lee--Yang \cite{Fortier2} to a more general setting we get the following characterization of foliations on simple surfaces, which we shall use to solve the simpler cases.
	\begin{lemma}\label{le:polygonsfinitelymanyarcs}
		Let $S$ be a sphere with one boundary component possibly containing boundary marked points and one interior marked point. Then every indecomposable foliation on $S$ is a proper arc and there are finitely many distinct proper arcs.
	\end{lemma}
	\begin{proof}
		Assuming that there is some foliation $F$ with a recurrent leaf to some part of $S$ we get a contradiction, as explained in the proof of \cite[Lemma 4.1]{Fortier2}. Hence, each indecomposable foliation is a curve. Any closed curve in $S$ is contractible to the marked point. Hence, a each indecomposable foliation is a proper arc.

        A proper arc in $S$ must have two endpoints, which must be contained in the boundary arcs in the boundary component of $S$. Denote $b_1$ and $b_2$ these two boundary arcs, which might be the same. We aim to show that there are at most two classes of arcs with endpoints in $b_1$ and $b_2$. Fix three proper arcs with endpoints on $b_1$ and $b_2$. Any intersection between these arcs can be removed by doing isotopies moving the endpoints along the arcs $b_1$ or $b_2$. Hence, these arcs can be isotoped to not intersect each other. Since there is only one interior marked point, two of these arcs delimit a rectangle with no marked interior marked points, so are isotopic. Hence, there are at most two different proper arcs between $b_1$ and $b_2$. There are finitely marked points in the boundary component, so there are finitely many boundary arcs. Therefore, there are finitely many pairs of boundary arcs, and since we have at most two proper arcs per pair, there are also finitely many different proper arcs.
	\end{proof}

	We shall first see the proposition for the case where $G$ contains a proper arc and we are approaching with a sequence of indecomposable foliations.
	\begin{lemma}\label{le:notsplittingboundary4}
		Let $S$ be a surface and let $(F_n)$ be a sequence of indecomposable foliations on $S$ converging to a measured foliation $G$. Then $G$ is either a multiple of a proper arc $\gamma$, in which case $F_n$ is also a multiple of $\gamma$ for $n$ big enough, or $G$ does not contain a proper arc.
	\end{lemma}
	\begin{proof}
		Assume $G$ contains a proper arc $\gamma$ with weight $w>0$ and denote $b$ one of the boundary arcs where $\gamma$ is incident.
		
		Our first step is seeing that, for $n$ big enough, $F_n$ intersects $b$. We shall do this by finding different test curves $\beta$ depending on the shape of $b$.
		If the boundary component containing $b$ has at most one marked point, we consider $\beta$ to be a curve parallel to that boundary component as in \cref{fi:samplecurves1}.
		Otherwise we consider $\beta$ to be the curve defined by taking a small arc starting at the boundary arc next to $b$, concatenating with a curve parallel to $b$, and concatenating another segment with endpoint in the boundary arc after $b$, as shown in \cref{fi:samplecurves2}.
		
		\begin{figure} \centering
			\begin{subfigure}[b]{0.45\textwidth}
				\centering
				\resizebox{\linewidth}{!}{
				\begin{tikzpicture}[scale=2]
					\node [label=right:$b$] at (-0.5,0) {};
					\draw  (0,0) ellipse (0.5 and 0.5);
					
					\node [label=left:$\beta$] at (-0.7,0) {};
					\draw  (0,0) ellipse (0.7 and 0.7);
					
					\node [label=above:$\gamma$] at (1,0) {};
					\draw (0.5,0) to [out=0, in =120] (1,0) to [out=-60,in=200] (1.5,0);		
				\end{tikzpicture}}  
				\caption{}
				\label{fi:samplecurves1}
			\end{subfigure}
			\begin{subfigure}[b]{0.45\textwidth}
				\centering
				\resizebox{\linewidth}{!}{
				\begin{tikzpicture}[scale=2]
					\node [label=above:$b$] at (0,-0.5) {};
					\draw  (0,0) ellipse (0.5 and 0.5);
					
					\node[circle,fill,inner sep=1pt] at (120:0.5) {};
					\node[circle,fill,inner sep=1pt] at (80:0.5) {};
					\node[circle,fill,inner sep=1pt] at (230:0.5) {};
					
					\node [label=left: ] at (-0.7,0) {};
					
					\node [label=above:$\beta$] at (50:0.7) {};
					\draw (90:0.5) to (90:0.7);
					\draw (220:0.5) to (220:0.7);
					\draw (90:0.7) arc (90:-140:0.7) ;
					
					\node [label=above:$\gamma$] at (1,0) {};
					\draw (0.5,0) to [out=0, in =120] (1,0) to [out=-60,in=200] (1.5,0);
				\end{tikzpicture}}  
				\caption{}
				\label{fi:samplecurves2}
			\end{subfigure}
			\caption{Sample curves used in the proof of \cref{le:notsplittingboundary4}}
			\label{fi:samplecurvesboth}
		\end{figure}
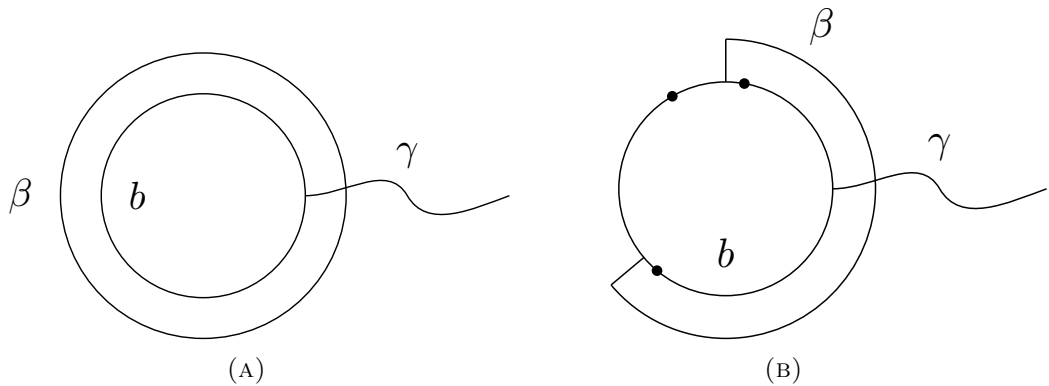
		
		If the curve $\beta$ is contractible then $S$ is a sphere with one boundary component and at most one interior marked point, so by \cref{le:polygonsfinitelymanyarcs} the result follows. Assume then that $\beta$ is not contractible. We have $i(\gamma,\beta)>0$, so $i(G,\beta)>0$ and hence $i(F_n,\beta)>0$ for $n$ big enough, which implies that $F_n$ intersects $b$. Hence, since $F_n$ is indecomposable, it is a weighted proper arc, which we denote $w_n\gamma_n$, where $w_n>0$ is the weight at $\gamma_n$ is a proper arc.
		
		Denote $b_1$ and $b_2$ the boundary arcs where $\gamma$ has its endpoints, and denote $\beta_1$ and $\beta_2$ the associated test curves shown in \cref{fi:samplecurvesboth}. If both endpoints are in the same boundary arc we set $b_2$ and $\beta_2$ as null curves. We shall now find a multicurve $A$ surrounding $\gamma$, $b_1$ and $b_2$ such that any leaf of $G$ intersecting $A$ but not $\gamma$ has an endpoint in either $b_1$ or $b_2$. The multicurve $A$ is chosen so that, together with the boundaries where $\gamma$ has its endpoints, delimits the smallest surface containing $\gamma$. The precise shape of $A$ depends on whether the endpoints of $\gamma$ are in the same boundary component or not, and the distribution of marked points in these boundaries.
				
		If both endpoints of $\gamma$ are in different boundary components we proceed differently according to the distribution of marked points at these boundaries. If each of the boundaries contains at most one marked point then we define $A$ as the curve shown in \cref{fi:boundarycurves1}. If one of the boundary components has two or more marked points, but the other has at most one marked point we define $A$ as the arc shown in \cref{fi:boundarycurves2}. Finally, if each of the boundaries contains at least two marked marked points we define $A$ as the multicurve formed by the curves $A_1$ and $A_2$ as shown in \cref{fi:boundarycurves3}.
		
		If both endpoints $\gamma$ are in the same boundary we also proceed differently according to the distribution of marked points. In all cases $A$ is defined as a multicurve formed by two curves. If each possible segment within the boundary component joining the two endpoints has at most one marked points we proceed as in \cref{fi:boundarycurves4}. If one of these segments has two or more marked points, while the other has at most one we proceed as in \cref{fi:boundarycurves5}. Finally, if both of these segments have two or more marked points we proceed as in \cref{fi:boundarycurves6}.
		
		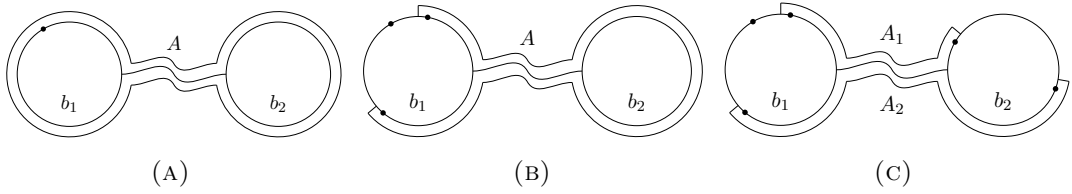
\begin{figure} \centering
			\begin{subfigure}[b]{0.30\textwidth}
				\centering
				\resizebox{\linewidth}{!}{
					\begin{tikzpicture}[scale=2]
						\node [label=above:$b_1$] at (0,-0.5) {};
						\draw  (0,0) ellipse (0.5 and 0.5);
						
						\node[circle,fill,inner sep=1pt] at (120:0.5) {};
						
						\node [label=above:$A$] at (1,0.1) {};
						\draw (-10:0.6) arc (-10:-350:0.6) ;
						\draw (10:0.6) to [out=0, in =120] (1,0.1) to [out=-60,in=180] ($(170:0.6)+(2,0)$);
						\draw ($(170:0.6)+(2,0)$) arc (170:190-360:0.6) ;
						\draw (-10:0.6) to [out=0, in =120] (1,-0.1) to [out=-60,in=180] ($(190:0.6)+(2,0)$);
						
						\draw (0.5,0) to [out=0, in =120] (1,0) to [out=-60,in=180] (1.5,0);		
						
						\node [label=above:$b_2$] at (2,-0.5) {};
						\draw  (2,0) ellipse (0.5 and 0.5);
				\end{tikzpicture}}  
				\caption{}
				\label{fi:boundarycurves1}
			\end{subfigure}
			\begin{subfigure}[b]{0.30\textwidth}
				\centering
				\resizebox{\linewidth}{!}{
				\begin{tikzpicture}[scale=2]
					\node [label=above:$b_1$] at (0,-0.5) {};
					\draw  (0,0) ellipse (0.5 and 0.5);
					
					\node[circle,fill,inner sep=1pt] at (120:0.5) {};
					\node[circle,fill,inner sep=1pt] at (80:0.5) {};
					\node[circle,fill,inner sep=1pt] at (230:0.5) {};
					
					\node [label=above:$A$] at (1,0.1) {};
					\draw (90:0.5) to (90:0.6);
					\draw (220:0.5) to (220:0.6);
					\draw (90:0.6) arc (90:10:0.6) ;
					\draw (-10:0.6) arc (-10:-140:0.6) ;
					\draw (10:0.6) to [out=0, in =120] (1,0.1) to [out=-60,in=180] ($(170:0.6)+(2,0)$);
					\draw ($(170:0.6)+(2,0)$) arc (170:190-360:0.6) ;
					\draw (-10:0.6) to [out=0, in =120] (1,-0.1) to [out=-60,in=180] ($(190:0.6)+(2,0)$);
					
					\draw (0.5,0) to [out=0, in =120] (1,0) to [out=-60,in=180] (1.5,0);		
					
					\node [label=above:$b_2$] at (2,-0.5) {};
					\draw  (2,0) ellipse (0.5 and 0.5);
				\end{tikzpicture}}  
				\caption{}
				\label{fi:boundarycurves2}
			\end{subfigure}
			\begin{subfigure}[b]{0.30\textwidth}
			\centering
			\resizebox{\linewidth}{!}{
			\begin{tikzpicture}[scale=2]
				\node [label=above:$b_1$] at (0,-0.5) {};
				\draw  (0,0) ellipse (0.5 and 0.5);
				
				\node[circle,fill,inner sep=1pt] at (120:0.5) {};
				\node[circle,fill,inner sep=1pt] at (80:0.5) {};
				\node[circle,fill,inner sep=1pt] at (230:0.5) {};
				
				\node[circle,fill,inner sep=1pt] at ($(150:0.5)+(2,0)$) {};
				\node[circle,fill,inner sep=1pt] at ($(-20:0.5)+(2,0)$) {};

				\node [label=above:$A_1$] at (1,0.1) {};
				\draw (90:0.5) to (90:0.6);
				\draw (90:0.6) arc (90:10:0.6) ;
				\draw (10:0.6) to [out=0, in =120] (1,0.1) to [out=-60,in=180] ($(170:0.6)+(2,0)$);
				\draw ($(170:0.6)+(2,0)$) arc (170:140:0.6) ;
				\draw ($(140:0.5)+(2,0)$) to ($(140:0.6)+(2,0)$);
				
				\node [label=below:$A_2$] at (1,-0.1) {};
				\draw ($(190:0.6)+(2,0)$) arc (190:350:0.6) ;
				\draw (-10:0.6) to [out=0, in =120] (1,-0.1) to [out=-60,in=180] ($(190:0.6)+(2,0)$);
				\draw ($(350:0.5)+(2,0)$) to ($(350:0.6)+(2,0)$);
				\draw (220:0.5) to (220:0.6);
				\draw (-10:0.6) arc (-10:-140:0.6) ;
				
				\draw (0.5,0) to [out=0, in =120] (1,0) to [out=-60,in=180] (1.5,0);		
				
				\node [label=above:$b_2$] at (2,-0.5) {};
				\draw  (2,0) ellipse (0.5 and 0.5);
			\end{tikzpicture}}  
			\caption{}
			\label{fi:boundarycurves3}
		\end{subfigure}
			\caption{Construction of the curves $A_1$ and $A_2$ whenever $\gamma$ has endpoints in different boundary components in the proof of \cref{le:notsplittingboundary4}}
		\end{figure}
				
		\begin{figure} \centering
			\begin{subfigure}[b]{0.30\textwidth}
				\centering
				\resizebox{\linewidth}{!}{
				\begin{tikzpicture}[scale=2]
					\node [label=above:$b$] at (0,-0.5) {};
					\draw  (0,0) ellipse (0.5 and 0.5);
					
					\draw (20:0.5) to [out=20, in =0] (0.5,1) to [out= 180, in = 10] (-0.5,1.5) to [out=190,in=160] (160:0.5);
					
					\node [label=above:$A_1$] at (0,0.5) {};
					\draw (30:0.6) to [out=30, in =0] (0.5,0.9) to [out= 180, in = 10] (-0.45,1.4) to [out=190,in=150] (150:0.6);
					
					\node [label=above:$A_2$] at (0,1.4) {};
					\draw (10:0.6) to [out=10, in =0] (0.5,1.1) to [out= 180, in = 10] (-0.5,1.6) to [out=190,in=170] (170:0.6);
					
					\draw (10:0.6) arc (10:-190:0.6) ;
					
					\draw (30:0.6) arc (30:150:0.6) ;
					
					\node[circle,fill,inner sep=1pt] at (120:0.5) {};
				\end{tikzpicture}}  
				\caption{}
				\label{fi:boundarycurves4}
			\end{subfigure}
			\begin{subfigure}[b]{0.30\textwidth}
				\centering
				\resizebox{\linewidth}{!}{
				\begin{tikzpicture}[scale=2]
					\node [label=right:$b$] at (-0.5,0) {};
					\draw  (0,0) ellipse (0.5 and 0.5);
					
					\draw (20:0.5) to [out=20, in =0] (0.5,1) to [out= 180, in = 10] (-0.5,1.5) to [out=190,in=160] (160:0.5);
					
					\node [label=above:$A_1$] at (0,0.5) {};
					\draw (30:0.6) to [out=30, in =0] (0.5,0.9) to [out= 180, in = 10] (-0.45,1.4) to [out=190,in=150] (150:0.6);
					\draw (30:0.6) arc (30:150:0.6) ;
					
					\node [label=above:$A_2$] at (0,1.4) {};
					\draw (10:0.6) to [out=10, in =0] (0.5,1.1) to [out= 180, in = 10] (-0.5,1.6) to [out=190,in=170] (170:0.6);
					\draw (10:0.6) arc (10:-70:0.6) ;
					\draw (-70:0.6) to (-70:0.5) ;
					\draw (170:0.6) arc (170:230:0.6) ;
					\draw (230:0.6) to (230:0.5) ;
					
					\node[circle,fill,inner sep=1pt] at (220:0.5) {};
					\node[circle,fill,inner sep=1pt] at (300:0.5) {};
				\end{tikzpicture}}  
				\caption{}
				\label{fi:boundarycurves5}
			\end{subfigure}
			\begin{subfigure}[b]{0.30\textwidth}
				\centering
				\resizebox{\linewidth}{!}{
				\begin{tikzpicture}[scale=2]
					\node [label=right:$b_1$] at (-0.5,0) {};
					\node [label=left:$b_2$] at (0.5,0) {};
					\draw  (0,0) ellipse (0.5 and 0.5);
					
					\draw (20:0.5) to [out=20, in =0] (0.5,1) to [out= 180, in = 10] (-0.5,1.5) to [out=190,in=160] (160:0.5);
					
					\node [label=above:$A_1$] at (-0.5,0.5) {};
					\draw (30:0.6) to [out=30, in =0] (0.5,0.9) to [out= 180, in = 10] (-0.45,1.4) to [out=190,in=150] (150:0.6);
					\draw (30:0.6) arc (30:40:0.6) ;
					\draw (40:0.6) to (40:0.5) ;
					\draw (120:0.6) arc (120:150:0.6) ;
					\draw (120:0.6) to (120:0.5) ;
					
					\node [label=above:$A_2$] at (0,1.4) {};
					\draw (10:0.6) to [out=10, in =0] (0.5,1.1) to [out= 180, in = 10] (-0.5,1.6) to [out=190,in=170] (170:0.6);
					\draw (10:0.6) arc (10:-70:0.6) ;
					\draw (-70:0.6) to (-70:0.5) ;
					\draw (170:0.6) arc (170:230:0.6) ;
					\draw (230:0.6) to (230:0.5) ;
					
					\node[circle,fill,inner sep=1pt] at (220:0.5) {};
					\node[circle,fill,inner sep=1pt] at (300:0.5) {};
					
					\node[circle,fill,inner sep=1pt] at (30:0.5) {};
					\node[circle,fill,inner sep=1pt] at (60:0.5) {};	
					\node[circle,fill,inner sep=1pt] at (95:0.5) {};
					\node[circle,fill,inner sep=1pt] at (130:0.5) {};
				\end{tikzpicture}}  
				\caption{}
				\label{fi:boundarycurves6}
			\end{subfigure}
			\caption{Construction of the curves $A_1$ and $A_2$ whenever $\gamma$ has endpoints in the same boundary component in the proof of \cref{le:notsplittingboundary4}}
		\end{figure}
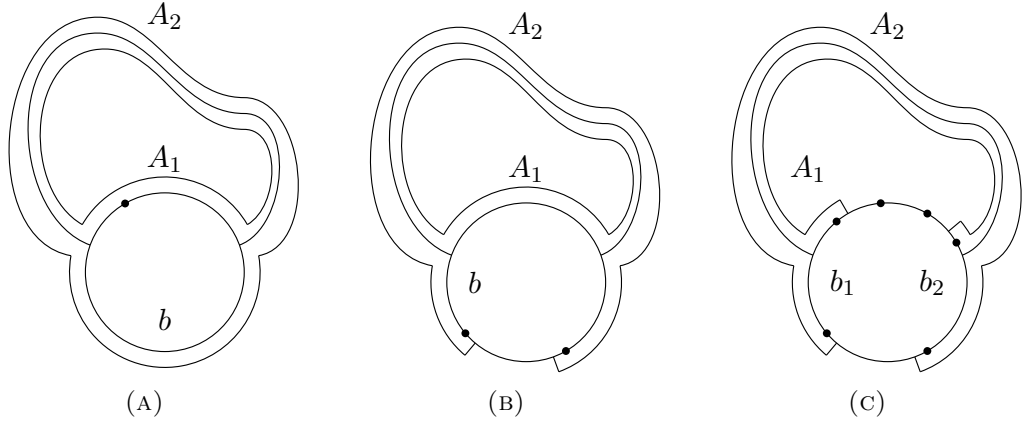
				
		In any of the cases above if a component of $A$ is non essential we remove it from $A$. The following argument also applies whenever $A$ is a null curve. Put $A$ and $G$ in minimal position and denote $P$ the surface containing $\gamma$, delimited by $A$ and the boundary components where $\gamma$ has its endpoints. Let $\alpha$ be a connected component of a non critical leaf of $G$ restricted to $P$ intersecting $A$. Since $G$ contains $\gamma$ the proper arc $\alpha$ cannot intersect $\gamma$. Furthermore, by observing the possible configurations, if $\alpha$ has one endpoint in $A_1$, the other one cannot be in $A_2$, as whenever we have both $A_1$ and $A_2$, these are separated within $P$ by the proper arc $\gamma$. Furthermore, if both endpoints are in $A_1$ then $\alpha$ can be isotoped to not intersect $A$. Therefore, the other endpoint of $\alpha$ is in either $b_1$ or $b_2$. Hence, $i(G,\beta_1)+i(G,\beta_2)\ge i(G,A)+w \, i(\gamma,\beta_1)+w \, i(\gamma,\beta_2)>i(G,A)$. Since $w_n \gamma_n$ converges to $G$, this last inequality implies that for $n$ big enough,
		\[i(\gamma_n,\beta_1)+i(\gamma_n,\beta_2)>i(\gamma_n,A).\]
		
		Fix $n$ such that $\gamma_n$ satisfies the previous inequality. Assume $\gamma_n$ has just one endpoint inside $P$. Then, $i(\gamma_n,\beta_1)+i(\gamma_n,\beta_2)=1$, so $i(\gamma_n,A)=0$ and $\gamma_n$ cannot leave $P$. If $\gamma_n$ has both endpoints in $P$ then $i(\gamma_n,\beta_1)+i(\gamma_n,\beta_2)=2.$ Furthermore, if $\gamma_n$ leaves $P$, then it has to reenter at some point, resulting in $i(\gamma_n,A_1+A_2)=2$. Hence, $\gamma_n$ stays inside $P$.
		
		The weights $w_n$ do not converge to $0$, as $w_ni(\gamma_n,\beta)$ converges to $i(G,\beta)$, but $i(\gamma_n,\beta)\le 2$. Since $\gamma$ is contained in $G$ we have $i(G,\gamma)=0$. Therefore, for any $\epsilon>0$ and $n$ big enough we have $w_ni(\gamma_n,\gamma)<\epsilon$, so for $n$ big enough $i(\gamma_n,\gamma)=0$. Since $\gamma_n$ does not intersect $\gamma$ and stays inside $P$, $\gamma_n$ can be isotoped to stay inside one of the components obtained after removing $\gamma$ from $P$. Denote $C$ such component. The component $C$ has either one or two boundary components and no interior marked points or one boundary component and one interior marked point. By \cref{le:polygonsfinitelymanyarcs} the only case where we do not have finitely many different proper arcs is when $C$ has two boundary components. However, in that case one of the boundary components is associated to a curve in $A$, so $\gamma_n$ does not intersect it and that boundary can be treated as a marked point. Hence, in all cases there are finitely many possible proper arcs, and so $\gamma_n$ is a multiple of $\gamma$ for $n$ big enough.
	\end{proof}
	
	When the boundary component is an annulus we have to be a bit more careful, so we start by proving it for approaching curves.
	\begin{lemma}\label{le:notsplittingboundary5}
		Let $S$ be a surface and let $(w_n\gamma_n)$ be a sequence of weighted curves on $S$ converging to a foliation $G$, where $(w_n)$ are the weights and $(\gamma_n)$ are the curves. Then $G$ is either a multiple of a boundary annulus $\gamma$, in which case $\gamma_n$ is $\gamma$ for $n$ big enough, or $G$ does not contain a boundary annulus.
	\end{lemma}
	\begin{proof}
		If $S$ is a polygon with at most one interior marked point, then $G$ cannot contain a boundary annulus. If $S$ is a cylinder then, since we have a boundary annulus, at least one of the boundaries must not contain marked points. Hence, the number of curves is finite, as there is only one possible closed curve, and for counting the proper arcs we can consider the boundary without marked points as a marked point and apply \cref{le:polygonsfinitelymanyarcs}. In that case, the conclusion follows.
		
		Assume then that $S$ is neither a disk with at most one interior marked point nor a cyclinder with no interior marked points. Then there is a pair of pants $P$ in $S$ containing $\gamma$ where each boundary component of $P$ is either non contractible or contractible to a marked point. Denote $B_1$ the boundary component parallel to $\gamma$ and $B_2$ and $B_3$ the other two boundary components of $P$. Furthermore, assume that $G$ contains $\gamma$ with weight $w$.
		
		Begin by assuming that $B_2$ and $B_3$ are not contractible to marked points. Let $C$ be the proper arc contained in $P$ with both endpoints in $B_1$. Put $B_2$, $B_3$ and $C$ in a minimal position with respect to $G$, and consider a connected component of a noncritical leaf of $G$ intersecting $C$ restricted to $P$. This noncritical leaf either is isotopic to $\gamma$, or to the curves $F$, $E$ and $D$ shown in \cref{fi:curvelabeling}. Since the leaves of $G$ do not intersect, there cannot be leaves isotopic to $E$ and leaves isotopic to $D$ at the same time. Breaking symmetry, assume there are no leaves isotopic to $D$. Then, $i(C,G)=i(C,\gamma)+i(B_3,G)>i(B_3,G)\ge i(B_2,G)$. Doing the same reasoning assuming that there are no leaves isotopic to $E$ we get $i(C,G)>\max(i(B_2,G),i(B_3,G))$. Hence, since $w_n\gamma_n$ converges to $G$, $\gamma_n$ has to satisfy
		\[
			i(C,\gamma_n)>\max(i(B_2,\gamma_n),i(B_3,\gamma_n))
		\]
		for $n$ big enough.
		
		For each $n$ put $B_3$, $B_2$ and $C$ in a minimal position with respect to $\gamma_n$, and consider the restriction of $\gamma_n$ to $P$. Assume $\gamma_n$ is not $\gamma$. Then, the curves on the restriction of $\gamma_n$ to $P$ intersecting $C$ are isotopic to either $E,F$ and $D$, but not $\gamma$. As before, this restriction cannot contain curves isotopic to $E$ and curves isotopic to $D$ for the same $n$, so assuming there are no curves isotopic to $D$ we have $i(C,\gamma_n)=i(B_3,\gamma_n)$ which is a contradiction. Doing the same reasoning assuming that there are no curves isotopic to $E$ also gives a contradiction. Hence, $\gamma_n$ is $\gamma$ for $n$ big enough.
		
		If $B_2$ or $B_3$ are contractible to marked points we have $i(G,B_2)$ or $i(G,B_3)$ is $0$, and a similar reasoning yields the same result.
		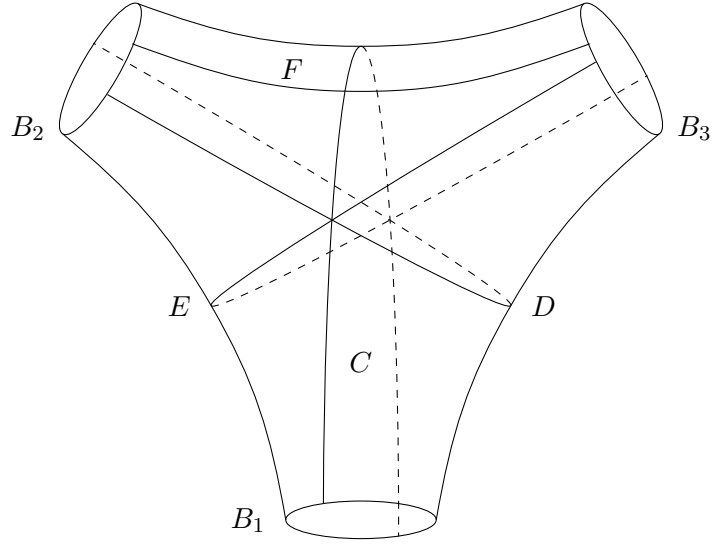
\begin{figure} \centering

			\begin{tikzpicture}
				\draw[rotate=150]  (0:4) ellipse (0.25 and 1);
				\draw[rotate=30]    (0:4) ellipse (0.25 and 1);
				\draw[rotate=270]  (0:4) ellipse (0.25 and 1);

				\draw[smooth] [rotate=30] (4,1) to  [out=170,in=-30] (60:2.3) to [out=150,in=-50] ($ (120:4)+({120-90}:1)$);
				\draw[smooth] [rotate=150] (4,1) to  [out=170,in=-30] (60:2.3) to [out=150,in=-50] ($ (120:4)+({120-90}:1)$);
				\draw[smooth] [rotate=270] (4,1) to  [out=170,in=-30] (60:2.3) to [out=150,in=-50] ($ (120:4)+({120-90}:1)$);
				
				\node[label=left:$F$] at (-0.5,2){};
	
				\draw[smooth] [rotate=30] (3.8,0.5) to  [out=170,in=-30] (60:1.7) to [out=150,in=-50] ($ (120:3.8)+({120-90}:0.5)$);

				\node[label=above:$B_3$] at (10:4.5) {};
				
				\node[label=above:$B_2$] at (170:4.5) {};
				
				\node[label=left:$B_1$] at (-1,-4){};
				
				\node[label=above:$C$] at (270:2.3) {};
				\draw [rotate=90] (0:2.3) arc(0:88:6.3 and 0.5);
				\draw[dashed] [rotate=90] (0:2.3) arc(0:-92:6.3 and 0.5);

				\node[label=left:$E$] at (210:2.3) {};
				\draw[dashed]  [rotate=210] (0:2.3) arc(0:92.3:6.3 and 0.25);
				\draw [rotate=210] (0:2.3) arc(0:-87.7:6.3 and 0.25);
				
				\node[label=right:$D$] at (330:2.3) {};
				\draw[dashed] [rotate=330] (0:2.3) arc(0:92.3:6.3 and 0.25);
				\draw [rotate=330] (0:2.3) arc(0:-87.7:6.3 and 0.25);
			\end{tikzpicture}
			\caption{Curve labeling for the proof of Lemma \ref{le:notsplittingboundary5}}
			\label{fi:curvelabeling}
		\end{figure}
	\end{proof}
	
	\begin{proof}[Proof of Proposition \ref{pr:notsplittingboundary}]
		Let $(F_n)$ be a sequence of measured foliations converging to $F$. As pointed out before, Proposition \ref{pr:curvesdense} can be extended to get sequences of weighted multicurves $(\gamma_n^m)_m$ converging to each $F_n$. Denote $\gamma_{n,1}^m,\gamma_{n,2}^m,\ldots,\gamma_{n,k(n,m)}^m$ the weighted curves of $\gamma_n^m$. For each $n$ we take a subsequence such that $k(n,m)$ is constant with respect to $m$, and $\gamma^{m}_{n,i}$ converges for each $i$ as $m\to\infty$. Denoting $F_{n,i}$ the limit of $\gamma^m_{n,i}$ as $m\to\infty$, we can write $F_n=\sum F_{n,i}$. 
		
		Denote $\beta_j$ the boundary components of $F$. That is, $\sum \beta_j=G$. Furthermore, denote $b_{n,j}$ and $b_{n,j}^m$ the weights of $\beta_j$ on $F_n$ and $\gamma_n^m$, where we set the weight to be $0$ if $\beta_j$ is not contained in the foliation. It is clear that if $b_{n,j}=0$ then $b_{n,j}^m\to 0$, as we must have $b_{n,j}\ge \liminf_{m\to\infty} b_{n,j}^m$. If $b_{n,j}>0$ for some $n$, then $F_{n,i}$ contains $\beta_j$ for some $i$. Hence, by Lemmas \ref{le:notsplittingboundary4} and \ref{le:notsplittingboundary5} we have $F_{n,i}$ and $\gamma_{n,i}^m$ are both multiples of $\beta_j$ for $m$ big enough. Then, since each of the multicurves in $\gamma_n^m$ has to be different, $\beta_j$ is not contained in any other foliation $F_{n,i}$ for that given $n$, so $F_{n_i}=b_{n,j}\beta_j$ and $\gamma_{n,i}^m$ can be written as $b_{n,i}^m\beta_j$ for $m$ big enough, with $b_{n,i}^m$ converging to $b_{n,j}$ as $m\to\infty$.
		
		Assume for some $j$ we have $b_{n,j}$ not converging to $1$. We can then take a subsequence such that $b_{n,j}$ converges to some $\lambda\neq 1$. Denote $\delta=|1-\lambda|/2$. For each $n$, there exists some $m_0(n)$ big enough so that $|1-b_{n,j}^m|>\delta$ for all $m\ge m_0(n)$. We can then take a diagonal sequence $\gamma_n^{m(n)}$ converging to $F$ with $m(n)\ge m_0(n)$. However, following the previous reasoning we get that $\gamma_n^{m(n)}$ should contain $\beta_{j}$ for $n$ big enough, and the weight should converge to the weight in $G$, that is, to $1$. However, $|1-b_{n,j}^{m(n)}|>\delta$, giving us a contradiction. Hence, $b_{n,j}$ converges to $1$ for all $j$. Let then $a_n=\min_j(b_{n,j})$. Since $b_{n,j}\ge a_n$ we can define $H_n=F_n-a_n G$ and we have $F_n=H_n+a_nG$. Finally, $a_n\to 1$ as $n\to\infty$, so the proposition is proved.
	\end{proof}
	
	\begin{proposition}\label{pr:ifpart}
		Let $q$ be a unit area quadratic differential such that $V(q)$ is internally indecomposable. Then $q$ is infusible.
	\end{proposition}
	
	\begin{proof}
		Assume $q$ is fusible, that is, we have a sequence of quadratic differentials $(q_n)$ converging to $q$ but not strongly. Let $F_i^n$ be the indecomposable components of $V(q_n)$. To have non-strong convergence we must have at least one sequence of indecomposable components converging to a decomposable component $G$, which we assume is $(F_1^n)_n$. Let $\beta$ be a boundary component of $V(q)$. By Proposition \ref{pr:notsplittingboundary} for $n$ big enough a multiple of $\beta$ must be contained in $V(q_n)$. Furthermore, $\beta$ cannot be contained in $G$. Since $G$ cannot contain boundary components, it must contain at least two interior components. On the other hand, since $V(q)$ is internally indecomposable, each interior part obtained by removing the proper arcs contains at most one interior component. Hence, for $n$ big enough $F_1^n$ must intersect at least two interior parts, that is, $F_1^n$ must cross at least one proper arc. However, for each proper arc $\gamma$ there is some $n$ big enough such that $\gamma$ is contained in the foliation $V(q_n)$, so $F_1^n$, a component of $V(q_n)$, intersects the foliation $V(q_n)$, giving us a contradiction.
	\end{proof}

	To prove the other direction we shall first see the following lemma.
	\begin{lemma}\label{le:interesctingcurves}
		Let $S$ be a compact surface with with possibly nonempty boundary and finitely many marked points, let $k\ge 2$ and let $\alpha=\{\alpha_1,\alpha_2,\ldots ,\alpha_k\}$ be a collection of non intersecting closed curves on $S$. Furthermore, let $p$ be the number of curves in $\alpha$ parallel to a boundary. Then there exists a collection of $\max(\lceil(p/2)\rceil),1)$ non intersecting curves intersecting each $\alpha_i$.
	\end{lemma}
	Our main interest in the lemma is that the amount of curves needed is strictly smaller than the amount of closed curves in $\alpha$. This will allow us, by doing Dehn twists along the closed curves in $\alpha$, to create a sequence of foliations converging to a foliation with strictly more components, which can be translated to a sequence of quadratic differentials that converge but not strongly. The proof of this lemma is based on a reasoning found in \cite[Proposition 3.5]{primer}.
	\begin{proof}
		We start by replacing all boundaries of $S$ without parallel curves in $\alpha$ by marked points. Let then $\alpha'$ be a completion of $\alpha$ to a pair of pants decomposition. Glue the remaining boundaries pairwise until we have at most one left. After cutting the surface along the closed curves that were not parallel to boundaries we get a collection of $\lceil p/2 \rceil$ tori with one boundary component and some spheres with $b$ boundary components and $n$ marked points, with $b+n=3$ and $b\ge 1$. If $p$ is odd, one of these spheres has a boundary of $S$ as a boundary. We join the boundaries of each of these surfaces with non intersecting arcs, as shown in \cref{fi:pathslaying}, that is, in such a way that each boundary component has two arcs incident to it. We can then paste these surfaces back together in order to obtain a collection $\beta_1,\beta_2,\ldots,\beta_l$ of pairwise disjoint curves in $S$. If $p$ is odd this collection contains precisely one proper arc, as we only have two endpoints coming from the boundary we did not paste. If $p$ is even the collection does not contain any proper arc. By the bigon criterion each $\beta_j$ is in minimal position with respect to each $\alpha_i$, and each $\alpha_i$ intersects either one or two of the $\beta_j$. Furthermore, since we did not cut along the original boundaries we pasted from $S$, each $\alpha_i$ parallel to a boundary of $S$ intersects precisely one of the $\beta_j$. Suppose we have $\beta_j$ and $\beta_{j'}$ intersecting a curve $\kappa\in \alpha'$ and that $\beta_j$ and $\beta_{j'}$ are distinct. Since we have at most one proper arc, at least one of $\beta_j$ and $\beta_{j'}$ is a closed curve. Hence, doing a half twist about $\kappa$, $\beta_j$ and $\beta_{j'}$ become a single curve. Since this process does not create any bigons, the resulting collection is still in minimal position with $\alpha$. Continuing this way we obtain a single curve $\gamma$ intersecting each curve in $\kappa$. Furthermore, $\gamma$ intersects each pasted boundary once. Cutting along the pasted boundaries, we get the curves from the lemma. If $p$ is odd, $\beta$ is a proper arc, so each cut along a pasted boundary increases the curve count by one, totalling $ (p+1)/2$ curves. If $p$ is even, $\beta$ is a closed curve, so the first cut transforms it into a proper arc, and the following ones increase the curve count by one, giving a total of $\max(p/2,1)$ curves.
		\begin{figure} \centering
		\begin{tikzpicture}
			\draw[smooth] (2.5,-0.85) to[out=180,in=30] (2,-1) to[out=210,in=-30] (0,-1) to[out=150,in=-150] (0,1) to[out=30,in=150] (2,1) to[out=-30,in=180] (2.5,0.85);
			\draw[smooth] (0.4,0.1) .. controls (0.8,-0.25) and (1.2,-0.25) .. (1.6,0.1);
			\draw[smooth] (0.5,0) .. controls (0.8,0.2) and (1.2,0.2) .. (1.5,0);
			\draw (2.5,-0.85) arc(270:90:0.3 and 0.85);
			\draw (2.5,-0.85) arc(270:450:0.3 and 0.85);
			\draw[smooth] (2.26,-0.5) to[out=180,in=30] (2,-0.6) to[out=210,in=270] (0,0) to[out=90,in=150] (2,0.6) to[out=-30,in=180] (2.26,0.5);
			
			\draw (3.5,-0.85) arc(270:90:0.3 and 0.85);
			\draw[dashed] (3.5,-0.85) arc(270:450:0.3 and 0.85);
			\draw (6,-0.85) arc(270:90:0.3 and 0.85);
			\draw (6,-0.85) arc(270:450:0.3 and 0.85);
			\draw (3.5,-0.85) to (6,-0.85);
			\draw (3.25,-0.5) to (5.75,-0.5);
			\draw (3.25,0.5) to (5.75,0.5);
			\draw[smooth] (3.5,0.85) to[out=0,in=260] (4.75,1.5) to [out=280,in=180] (6,0.85) ;
			\node at (4.75,1.5) [circle,fill,inner sep=0.7pt]{};
			
			\draw (7,-0.85) arc(270:90:0.3 and 0.85);
			\draw[dashed] (7,-0.85) arc(270:450:0.3 and 0.85);
			\draw (9.5,-0.85) arc(270:90:0.3 and 0.85);
			\draw (9.5,-0.85) arc(270:450:0.3 and 0.85);
			\draw (7,-0.85) to (9.5,-0.85);
			\draw[smooth] (7,0.85) to[out=0,in=260] (7.75,1.5)  ;
			\draw[smooth] (8.75,1.5) to[out=280,in=180] (9.5,0.85)  ;
			\draw (7.75,1.5) arc(180:0:0.5 and 0.15);
			\draw (7.75,1.5) arc(-180:0:0.5 and 0.15);
			\draw (6.75,-0.5) to (9.25,-0.5);
			\draw[smooth] (6.75,0.5) to[out=0, in = 270] (8,1.37);
			\draw[smooth] (9.25,0.5) to[out=180, in = 270] (8.5,1.37);
			
			\draw (10.5,-0.85) arc(270:90:0.3 and 0.85);
			\draw[dashed] (10.5,-0.85) arc(270:450:0.3 and 0.85);
			\draw[smooth](10.5,0.85) to [out=0, in = 220] (12.5,1.5) to [out=240, in = 120] (12.5,-1.5) to [out=140, in = 0] (10.5,-0.85);  
			\draw[smooth](10.25,-0.5)  to [out=0, in = 230] (12.06,0) ;  
			\draw[smooth,dashed](12.06,0)  to [out=100, in = -20] (11.5,0.95) ;  
			\draw[smooth](11.5,0.95) to [out=200, in = 0] (10.25,0.5) ;  
			\node at (12.5,1.5) [circle,fill,inner sep=0.7pt]{};
			\node at (12.5,-1.5)  [circle,fill,inner sep=0.7pt]{};
		\end{tikzpicture}
			\caption{Laying out of curve segments for the proof of Lemma \ref{le:interesctingcurves}}
			\label{fi:pathslaying}
		\end{figure}
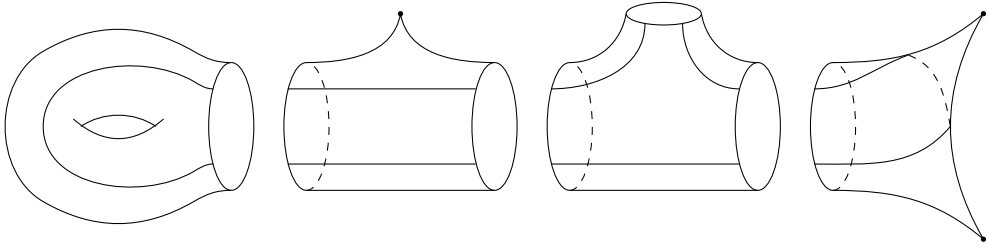
	\end{proof}

	\begin{proposition}\label{pr:onlyifpart}
		Let $F$ be an internally decomposable measured foliation. Then, $F$ can be approached by a sequence of weighted multicurves with fewer components than $F$.
	\end{proposition}
	\begin{proof}		
		By the extension to Proposition \ref{pr:curvesdense}, we have a sequence of weighted multicurves $\gamma^n$ converging converging to $F$, with the only proper arcs being the ones contained in $F$. Cutting the surface along the proper arcs of $\gamma^n$ and quotienting these proper arcs to points we get $k$ many surfaces $Z_1,Z_2,\ldots ,Z_k$ with boundary. Let $\gamma^n_i$ be the restriction of $\gamma^n$ to $Z_i$, and let $F_i$ be the limit of $\gamma^n_i$. The foliation $F$ is the union of the foliations $F_i$ and the proper arcs. 
		
		Fix some $i$ such that $F_i$ is nonempty, and let $\alpha_1,\ldots, \alpha_b$ be the closed curves parallel to the boundaries of $Z_i$. Let $a^n_1,\ldots, a^n_b$ be the weights of $\alpha_1,\ldots,\alpha_b$ in $\gamma^n_i$. We can take a subsequence such that $a^n_j$ converges for each $j$ to some $a_j$. If $a_j>0$, the closed curve $\alpha_j$ is contained in $F_i$. If $a_j=0$, then the weights $a^n_j$ can be set to $0$ on the multicurves $\gamma^n_i$ while leaving the limit intact. Hence, we can assume that $a^n_j=0$ for all $j$ such that $a_j=0$. Let $p$ and $u$ be the number of closed curves with $a_j>0$ parallel to boundaries with or without marked points respectively. Since we have removed all the closed curves with $a_j=0$, the multicurve $\gamma^n_i$ contains precisely $p$ and $u$ closed curves parallel to boundaries with or without marked points for $n$ big enough. Denote by $B$ the set of closed curves parallel to boundary components without marked points. Applying Lemma \ref{le:interesctingcurves} to the multicurve $\gamma^n_i$ minus $B$ we get $\max(\lceil(p/2)\rceil),1)$ curves $\beta^n_i$ intersecting all closed curves in $\gamma^n_i$ except the ones parallel to boundaries without marked points. Doing the appropriate Dehn twists along the closed curves of $\gamma^n_i$ and rescaling to the curves $\beta^n_i$, and adding with the corresponding weights the curves in $B$, we get a sequence converging to $\gamma^n_i$ with $\max(\lceil(p/2)\rceil),1)+u$ many components. As such, taking a diagonal sequence we can get a sequence of multicurves converging to $F_i$ with each multicurve containing $\max(\lceil(p/2)\rceil),1)+u$ components.
		
		Finally, since $F$ is internally decomposable, there is at least one $F_i$ with at least 2 interior components, so one of these multicurves has strictly less components than the limiting foliations, and we have non-strong convergence.
	\end{proof}

	Theorem \ref{th:maxcondition} follows by combining Propositions \ref{pr:ifpart} and \ref{pr:onlyifpart}.

	We do not need $S$ to have a lot of topology to find internally decomposable foliations. In fact, determining which surfaces do not support internally decomposable foliations we get the following result.
	
	\begin{proposition}\label{pr:continuousBusemann}
		Let $S_{g,b_m,b_u,p}$ be a surface of genus $g$ with $b_m$ and $b_u$ boundaries with and without marked points respectively and $p$ interior marked points. Then the Busemann map is continuous if and only if $3g+2b_m+b_u+p\le 4$.
	\end{proposition}

	We shall split the proof in the following two lemmas
	
	\begin{lemma}\label{le:nonexceptionalimpliesfusible}
		Let $S_{g,b_m,b_u,p}$ be a surface with $3g+2b_m+b_u+p > 4$. Then it admits an internally decomposable foliation.
	\end{lemma}
	\begin{proof}
		A multicurve consisting of two interior closed curves generates an internally decomposable foliation, so we just have to find such a pair for each possible surface satisfying the hypothesis. If $S$ has genus at least 2 we can take a multicurve consisting of $2$ non separating closed curves. If $S$ is a torus with at least 2 boundaries or marked points, or a boundary with marked points, we can take a non separating closed curve and a separating closed curve around 2 boundaries or marked points, or around a boundary with marked points. If $S$ is a sphere with at least 5 marked points or boundaries, we can take a closed curve around two interior points or boundaries, and a closed curve around two different interior points or boundaries. If $S$ is a sphere with 1 boundary with marked points and at least 3 other boundaries or interior points we can take a closed curve around the boundary with marked points, and a closed curve around two other interior points or boundaries. Lastly, if $S$ is a sphere with 2 boundaries with marked points and another interior marked point or boundary we take a closed curve around each boundary with marked points.
	\end{proof}

	\begin{lemma}\label{le:exceptionalimpliesnonfusible}
		Let $S_{g,b_m,b_u,p}$ be a surface with $3g+2b_m+b_u+p\le 4$. Then every foliation on $S$ is internally indecomposable.
	\end{lemma}
	\begin{proof}
		Assume we have an internally decomposable foliation on $S_{g,b_m,b_u,p}$. Then we can get an internally decomposable foliation on $S_{g,0,0,b_u+p+2b_m}$ by removing the boundary components, replacing the boundaries without marked points with marked points and each boundary with marked points for 2 marked points. Furthermore, if we have at least one marked point, we can get an internally decomposable foliation in $S_{g,0,0,b_u+p+2b_m+k}$, $k\in \N$, by replacing a marked point with a $k+1$ marked points.
				
		Hence, we only need to prove that a torus with one marked point and a sphere with 4 marked points do not admit internally decomposable foliations. However, since these do not have boundaries, a foliation being internally decomposable translates to a foliation having at least two indecomposable components. 
		
		Assume the torus with one marked point admits a foliation with two indecomposable components. We can replace the marked point with a boundary, and add to the foliation a boundary component parallel to that boundary. Considering the doubled surface explained in \cref{se:doublingtrick} we get a closed surface of genus $2$ without boundaries nor marked points, with at least 5 indecomposable components. Recall that the maximum number of indecomposable components for a foliation on a surface of genus $g$ is $3g-3$, so for genus $2$ the maximum is $3$, giving us a contradiction. A similar process applies for the sphere with 4 marked points.
	\end{proof}
				
	\begin{proof}[Proof of Proposition \ref{pr:continuousBusemann}]
		The Busemann map is continuous at every point in the interior of Teichmüller space, as it is the identity when restricted in there and $\vbd{X}$ is closed. Hence, we only need to prove continuity or discontinuity at the points on the boundary. By Lemma \ref{le:nonexceptionalimpliesfusible} if $3g+2b_m+b_u+p>4$ then $S$ admits an internally decomposable foliation $F$, so by \cref{th:maxcondition} the Hubbard--Masur quadratic differential associated to $F$ at the basepoint $X$ is fusible and hence the Busemann map is not continuous at that point. On the other hand, if $3g+2b_m+b_u+p\le 4$ then by \cref{le:exceptionalimpliesnonfusible} for any quadratic differential $q$, the vertical foliation $V(q)$ is internally indecomposable, so again by \cref{th:maxcondition} every quadratic differential is infusible an $B$ is continuous at every boundary point.
	\end{proof}
		
	By combining Proposition \ref{pr:continuousBusemann} with Proposition \ref{pr:horobocompfiner}, we get the precise classification of surfaces with horofunction compactification isomorphic to visual compactification announced in Theorem \ref{th:homeomorphictovisualcomp} from the introduction.

	\begin{proof}[Proof of Theorem \ref{th:homeomorphictovisualcomp}]
		As shown in Proposition \ref{pr:horobocompfiner}, the visual compactification and the horofunction compactification are isomorphic if and only if the Busemann map is continuous, so the theorem follows by applying Proposition \ref{pr:continuousBusemann}.
	\end{proof}

\subsection{Criteria for convergence}\label{se:horocycles}

	One straightforward consequence of the horofunction compactification being finer than the visual compactification is the following criterion regarding the convergence of sequences in the horofunction compactification.

	\begin{corollary}\label{co:convergencepaths}
		Let $(x_n)\subset \T(S)$ be a sequence. If $(x_n)$ converges to a quadratic differential $q$ in the visual compactification, then all accumulation points of $(x_n)$ in the horofunction compactification are contained in $\fibermap^{-1}(q)$. In particular, if $V(q)$ is internally indecomposable, then $(x_n)$ converges in the horofunction compactification.
		
		Furthermore, if $(x_n)$ does not converge in the visual compactification, then it does not converge in the horofunction compactification.
	\end{corollary}
	\begin{proof}
		If $x_n$ converges in the visual compactification to a quadratic differential $q$ then by the continuity of $\fibermap$ all its accumulation points are in $\fibermap^{-1}(q)$. If $V(q)$ is internally indecomposable, then by \cref{th:maxcondition} the quadratic differential $q$ is infusible, so the Busemann map is continuous at $q$ and by \cref{pr:continuityatqintro} the fiber $\fibermap^{-1}(q)$ is a singleton. Therefore $x_n$ converges to $\fibermap^{-1}(q)$, as that is the only accumulation point of $x_n$ and the horofunction compactification is compact.
		
		On the other hand, if $x_n$ converges to $\xi$ in the horofunction compactification, by continuity of $\fibermap$, $x_n$ converges to $\fibermap(\xi)$ in the visual compactification.
	\end{proof}
	
	A frequent topic in the study of compactifications of Teichmüller spaces is the convergence of certain measure-preserving paths. We shall see now how the previous results can be applied in that study.
	
	Let $X\in\T(S)$ be a point in Teichmüller space and $q$ be a unit quadratic differential based at $X$. It is a well known fact that there exists a unique orientation-preserving isometric embedding $\iota :\H\to \T(S)$ from the hyperbolic plane $\H$ to the Teichmüller space such that $\iota(i)=X$ and $\iota^{*}(q)=i$, see the work of Herrlich--Schmithüsen \cite{Herrlich} for a detailed explanation. The path $\iota(i+t)$ for $t\in \R_+$ is called the \emph{horocycle} generated by $q$. Since $\iota$ is an isometric embedding, $h(X)(p)=d(\iota^{-1}X,\iota^{-1}p)-d(\iota^{-1}X,\iota^{-1}b)$ for $X,b,p\in \iota(\H)$. That is, if we restrict the evaluations of horofunctions to the image of the Teichmüller disc, the value coincides with the values in the hyperbolic plane. Hence, since the path $i+t$ is a horocycle of the Busemann point obtained by moving along the geodesic $e^ti$ along the hyperbolic plane, the path $\iota(i+t)$ is also a horocycle of the corresponding Busemann point $B(q)$, obtained by moving along the geodesic $\iota(e^t i)$.
	
	Since $\iota$ is an isometric embedding, the geodesic between $X$ and $\iota(i+t)$ is contained in $\iota(\H)$. Furthermore, the pushforward and pullback maps are continuous, so denoting $q_t$ the unit quadratic differential spawning the geodesic between $X$ and $\iota(i+t)$, we have $\lim_{t\to\infty} \iota^* (q_t)=i$, and $\iota_*(i)=q$, so $\lim_{t\to\infty} q_t=q$. The distance between $\iota(i+t)$ and $X$ grows to infinity, so any horocycle path generated by some $q$ based at $X$ converges to $q$ in the visual compactification based at $X$. Hence, horocycles generated by infusible quadratic differentials converge in the horofunction compactification, which had been previously shown by Jiang--Su \cite{jiang} and Alberge \cite{Vincent} in the context of surfaces without boundary.
	\begin{corollary}\label{co:convergencehorocycles}
		Let $S$ be a compact surface with possibly nonempty boundary and finitely many marked points and let $q$ be an infusible quadratic differential based at any $X\in \T(S)$. Then the horocycle generated by $q$ converges in the horofunction compactification.
	\end{corollary}
	\begin{proof}
		The horocycle path converges to $q$ in the visual compactification based at $X$, so by \cref{co:convergencepaths} all accumulation points in the horofunction compactification are contained in $\fibermap^{-1}_X(q)$. Furthermore, since $q$ is infusible, $\fibermap^{-1}_X(q)$ is a singleton, so the horocycle path has a unique accumulation point in the horofunction compactification, and hence it converges.
	\end{proof}
	
	On the other hand, Fortier Bourque found some diverging horocycles in the horofunction compactification.
	\begin{theorem}[Fortier Bourque {\cite[Theorem 1.1]{Fortier}}]
		Let $S$ be a closed surface of genus $g$ with $p$ marked points, such that $3g+p\ge5$. Then there is some fusible quadratic differential $q$ based at some basepoint $X\in\T(S)$ such that the associated horocycle path does not converge in the horofunction compactification.
	\end{theorem}
	\cref{co:convergencepaths} gives an upper limit on the set of accumulation points, as it has to be contained in $\fibermap^{-1}_X(q)$.
	
	Furthermore, by \cref{co:convergingtohoriffconvergingeveryvis} we have that a path converges in the horofunction compactification if and only if it converges in each visual compactification. Hence, such a divergent horocycle also diverges in some visual compactification. That is, we get Corollary \ref{co:divergencehorocycles}. 
	This contrasts with the behavior of Teichmüller rays, which by \cref{co:convergencegeodesics} or \cite[Theorem 7]{Walsh} converge in all visual compactifications.
	
\section{Dimension of the fibers}\label{se:shapeoffibers}

	Our first approach in determining the shape of the fibers is looking at the limits of Busemann points, which by Proposition \ref{pr:lowerbound} give us bounds on the elements of $\fibermap^{-1}(q)$. For a given quadratic differential $q$ and a foliation $G$ we define $\wals^q(G)$ as the map from measured foliations to $\R$ given by 
	\[
	\wals^q(G)=\frac{i(G,\cdot)^2}{i(G,H(q))},
	\]
	if $i(G,H(q))>0$, and $\wals^q(G)=0$ otherwise.
	By the extension of Walsh's \cref{co:walshbusemanshape} describing Busemann points in the Gardiner--Masur compactification, we see that the element $\E_q=\Xi^{-1} B_q$ has the form $\sqrt{\sum_i \wals^q(V_i)}$, where $V_i$ are the indecomposable components of $V(q)$. Hence, a reasonable path to follow for understanding the limits of Busemann points is understanding the limits of $\wals^q$ as $q$ varies.
	\begin{lemma}\label{le:limith}
		Let $q_n$ be a sequence of quadratic differentials on $X$ converging to $q$, and let $V_j^n$, $1\le j\le c(n)$ be the indecomposable components of $V(q_n)$. Let $G^n$ be a sequence of non zero measured foliations of the form $\sum \alpha_j^n V_j^n$, converging to a measured foliation $G$. Then
		\[
		\lim_{n\to\infty} \wals^{q_n}(G^n) = \wals^q (G)
		\]	
		if $G$ is non zero and $\lim_{n\to\infty} \wals^{q_n}(G^n)=0$ if $G$ is zero, where the convergence is pointwise in both cases.
	\end{lemma}
	\begin{proof}
		For any measured foliation $F$ we have $\wals^{q_n}(G^n)(F)=\frac{i(G^n,F)^2}{i(G^n,H(q_n))}$, so if $G$ is non zero the lemma follows by continuity of the intersection number.

		If $G$ is zero  the result follows from applying the same proof than in \cite[Lemma 27]{Walsh}.
	\end{proof}
	
	Denote $\B$ the set of Busemann points, $\Bclose$ its closure and $\Bclose(q)$ the intersection $\Bclose\cap \fibermap^{-1}(q)$. We can use the previous lemma to show that the elements of $\Bclose(q)$ satisfy certain properties.
	
	\begin{proposition}\label{pr:busemanclosureshape}
		Let $S$ be a closed surface with possibly marked points, $\xi\in \Bclose(q)$ and $V_i$, $i\in\{1,\ldots,k\}$ be the indecomposable components of $V(q)$. Denote $x_i=\frac{i(V_i,\cdot)}{i(V_i,H(q))}$. Then, the square of the representation of $\xi$ in the Gardiner--Masur compactification, $(\Xi^{-1} \xi)^2$, is a homogeneous polynomial of degree $2$ in the variables $x_i$, whose coefficients sum to $1$.
	\end{proposition}
	
	Recall that we are using a normalized version of the Gardiner--Masur compactification. Under the projectivized version the sum of the coefficients cannot have any fixed value.
	
	\begin{proof}
		Since the surface does not have boundary, all Busemann points are of the form $B(q')$ for some quadratic differential of unit area $q'$. Consider a sequence $(q_n)$ such that $B(q_n)$ converges to $\xi$ and $q_n$ converges to $q$. Let $c(n)$ be the number of indecomposable vertical components of $V(q_n)$, and let $V^n_j$, $0<j\le c(n)$ be those components. We know that $c(n)$ is bounded by some number depending on the topology of the surface. Take a subsequence such that $c(n)$ is equal to some constant $c$ and $V^n_j$ converges for each $j$. The sum $\sum_{j=1}^c V^n_j$ converges as $n\to\infty$ to $\sum_{i=1}^k V_i$, so the limit of each $V^n_j$ has to be of the form $\sum_{i=1}^k \alpha_j^i V_i$. Furthermore, $\sum_{j=1}^c \alpha_j^i=1$, since 
		\[
			\sum_{i=1}^k V_i=V(q)=\lim_{n\to\infty} V(q_n)=\lim_{n\to\infty}\sum_{j=1}^cV_{j}^n=\sum_{j=1}^c\sum_{i=1}^k\alpha_j^iV_i=\sum_{i=1}^k\left(\sum_{j=1}^c\alpha_j^i\right)V_i.
		\]
		The element associated to the Busemann point $B(q_n)$ in the Gardiner--Masur compactification satisfies
		\[
		\E_{q_n}^2=\sum_{j=1}^c \wals^{q_n}(V^n_j).
		\]
		Hence, applying Lemma \ref{le:limith} we get the following expressions for the square of the limit of Busemann points:
		\[
		(\Xi^{-1}\xi)^2=\sum_{j=1}^c \wals^q\left(\sum_{i=1}^k \alpha_j^i V_i\right)=\sum_{j=1}^c\frac{\left(\sum_{i=1}^k \alpha_j^i i(V_i,H(q)) x_i\right)^2}{\sum_{i=1}^k \alpha_j^i i(V_i,H(q))}.
		\]
		That is, we get a homogeneous polynomial of degree 2 in the variables $x_i$. Since $q$ has unit area, the sum of the coefficients is
		\[
		\sum_{j=1}^c \sum_{i=1}^k \alpha_j^i i(V_i,H(q))=\sum_{i=1}^k i(V_i, H(q))=1,
		\]
		which completes our claim.
	\end{proof}

	By Proposition \ref{pr:upperbound}, the Busemann point $B(q)$ gives an upper bound on all functions in $\fibermap^{-1}(q)$. While Proposition \ref{pr:lowerbound} does not give us a lower bound directly, we can use Lemma \ref{le:walshcompact} to get one. For a unit area quadratic differential $q$, let $Z_j$ be the interior parts of $V(q)$, and denote $G_j$ the union of interior indecomposable components within $Z_j$. Further, let $P_i$ be the boundary components of $V(q)$. We define the \emph{minimal point} at $q$ as
	\[
		M(q)=\Xi \left(\sum_i \wals^q(P_i)+\sum_j \wals^q\left(G_j\right)\right)^{1/2}.
	\]
	
	\begin{proposition}\label{pr:lowerboundGM}
		Let $q$ be a quadratic differential. Then, for any $\xi\in\fibermap^{-1}(q)$, we have
		\[\Xi^{-1}\xi\ge \Xi^{-1} M(q)\]
		in the Gardiner--Masur compactification.
		Furthermore, $M(q)\in \fibermap^{-1}(q)$ whenever each $G_j$ has at most two annuli parallel to the boundaries of $Z_j$ with marked points.
	\end{proposition}
	
	In the context of surfaces without boundary the previous result has been also proven by Liu--Shi in \cite[Lemma 3.10]{LiuShi}. In such context we have $M(q)=\Xi i(V(q),\cdot)^2$, which by the proposition is always contained in $\fibermap^{-1}(q)$.

	The minimality is essentially derived from the following well-known inequality.
	\begin{lemma}[Titu's lemma]\label{le:elementaryinequality}
		For any positive reals $a_1,\ldots, a_n$ and $b_1,\ldots, b_n$ we have 
		\[
		\sum_j\frac{a_j^2}{b_j}\ge \frac{\left(\sum_j a_j\right)^2}{\sum_j b_j}.
		\]
	\end{lemma}
	\begin{proof}
		The inequality can be written as
		\[
		\sum_i b_i\sum_j\frac{a_j^2}{b_j}\ge \left(\sum_j a_j\right)^2,
		\]
		so the result follows after applying the Cauchy--Schwartz inequality.
	\end{proof}
	
	The implication this lemma has for our discussion is that $\wals^q(\cdot)$ is convex, in the sense that for any $G=\sum_i G_i$ and any measured foliation $F$ we have 
	\[
		\sum_i \wals^q(G_i)(F) \ge \wals^q(G)(F).
	\]

	\begin{proof}[Proof of Proposition \ref{pr:lowerboundGM}]
		If $q$ is infusible then each $G_j$ is indecomposable, so $M(q)=B(q)$, the fiber $\fibermap^{-1}(q)$ has one point and the proposition is satisfied.
		
		Consider then $q$ fusible and $\xi\in \fibermap^{-1}(q)$. Let $(x_n)=(\tray{q_n}{t_n})\subset \T$ converging to $\xi$. By Lemma \ref{le:walshlowerbound} we have $\Xi^{-1}(h(x_n))\ge\Xi^{-1} B(q_n)$. Hence, $\Xi^{-1}\xi\ge \liminf_{n\to\infty}\Xi^{-1} B(q_n)$.
		
		Given a measured foliation $F$, take a subsequence so that  \[\liminf_{n\to\infty}\Xi^{-1} B(q_n)(F)=\lim_{n\to\infty}\Xi^{-1} B(q_n)(F).\]
		The foliations $V(q_n)$ converge to $V(q)$, so by Proposition \ref{pr:notsplittingboundary} for $n$ big enough all boundary components $P_i$ are contained within $V(q_n)$. Hence, for $n$ big enough the foliations $V(q_n)$ can be split to the interior parts $Z_j$ by cutting along the proper arcs. Denote $G_j^n$ the interior components of the foliation $V(q_n)$ restricted to $Z_j$. Let $G_{j,k}^n$ be the indecomposable components of $G_j^n$. The sequence $G_j^n$ converges to $G_j$, so we can take a subsequence such that each $G_{j,k}^n$ converges to some foliation $G_{j,k}$ with $\sum_k G_{j,k}=G_j$. Applying Lemma \ref{le:limith} we have
		\begin{multline*}
			\lim_{n\to\infty}\Xi^{-1} B(q_n)(F)=\lim_{n\to\infty}\sum_i^n \wals^{q_n}(P_i)+\sum_j \sum_k \wals^{q_n}\left(G^n_{j,k}\right)\\=\sum_i \wals^q(P_i)+\sum_j \sum_k \wals^q\left(G_{j,k}\right).
		\end{multline*}
		Hence, applying Lemma \ref{le:elementaryinequality} to the second sum we get the first part of the proposition.
		
		To observe that the limit is actually reached we can repeat the proof of Proposition \ref{pr:onlyifpart} and observe that a proper arc for each interior part is enough to approach the foliation whenever each interior part of the foliation has at most two annuli parallel to boundaries with marked points.
	\end{proof}

	By \cref{co:orderpreserved} this lower bound is carried to the horofunction representation and by Proposition \ref{pr:upperbound} we have an upper bound. Hence, we have the chain of inequalities
	\[M(q) \le \xi \le B(q),\]
	for any $\xi\in \fibermap^{-1}(q)$.
	As we see in the next proposition, this chain can be translated as well to the Gardiner--Masur compactification.
		
	\begin{proposition}\label{pr:upperboundGM}
		Let $\xi\in\fibermap^{-1}(q)$. Then,
			\[\Xi^{-1}\xi \le \Xi^{-1} B(q).\]
	\end{proposition}
	\begin{proof}
		We have a sequence of points $\tray{q_n}{t_n}$ converging to $\xi$, with $q_n$ converging to $q$.
		By Lemma \ref{le:fixingvaluealongpath} we have $\xi(\tray{q}{t})=-t$. Further, $\tray{q_n}{t_n}$ converges in the Gardiner--Masur compactification to the function $f(G)^2=\lim_{n\to\infty} e^{-2t_n} \ext_\tray{q_n}{t_n}(G)$, and we have $\Xi f (x) = \xi(x)$. Hence,
		\[\frac{1}{2}\log\frac{f(F)}{\ext_\tray{q}{t}(F)}\le \frac{1}{2}\log\sup_{G\in P} \frac{f(G)}{\ext_\tray{q}{t}(G)}=-t.
		\]
		Upon exponentiating and reordering the terms, we get
		\[\lim_{n\to\infty} e^{-2t_n} \ext_\tray{q_n}{t_n}(F)=f^2(F)\le e^{-2t}\ext_\tray{q}{t}(F)
		\]
		for all $t$. Letting $t\to\infty$, the right hand side converges to $(\Xi^{-1}B(q)(F))^2,$ so we get the proposition.
	\end{proof}
		
		 Using these bounds we can further refine the characterization of points in $\Xi^{-1}\fibermap^{-1}(q)$.
	\begin{proposition}\label{pr:firstderivative}
		Let $q$ be a quadratic differential, let $V_i$, $i\in\{1,\ldots,k\}$ be the indecomposable components of $V(q)$ and let $x_i(F)=\frac{i(V_i,F)}{i(V_i,H(q))}$. Given $f\in \Xi^{-1}\fibermap^{-1}(q)$ and $c>0$ we have, for all $F\in \MF$,
		\[ 
		f^2(F)=c^2+2 c \sum_i i(V_i,H(q))(x_i(F)-c)+\sum_{i,j} O \left((x_i(F)-c)(x_j(F)-c)\right).
		\]
		In particular, as a function of the values $x_i(F)$ at the point $x_i=c$ for all $i$, $f^2(x_1,\ldots ,x_k)$ takes value $c^2$, is differentiable and satisfies $\frac{\partial}{\partial x_i} f^2(x_1,\ldots,x_k) = 2 c\:i(V_i,H(q)).$
	\end{proposition}
	\begin{proof}
		We have that $\left(\Xi^{-1} M(q)\right)^2\le f^2\le \left(\Xi^{-1} B(q)\right)^2$. Denoting $a_i=i(V_i,H(q))$ and $x_i=x_i(\cdot)$ we have by Lemmas \ref{le:elementaryinequality} and \ref{pr:lowerboundGM} that $\left(\sum a_ix_i\right)^2\le \left(\Xi^{-1} M(q)\right)^2$. Writing the bounds on $f^2$ in terms of the variables $x_i$, we obtain 
		\[
		\left(\sum a_ix_i\right)^2\le f^2\le \sum a_ix^2_i.
		\]
		Adding that $\sum a_i=1$, we have that $f^2$ is bounded below by the arithmetic mean, and above by the quadratic mean. Rewritting both sides as a polynomial in $x_i-c$, we get
		\[
		c^2+2c\sum a_i(x_i-c)+\left(\sum a_i (x_i-c)\right)^2\le f^2\le c^2+2c\sum a_i(x_i-c)+\sum a_i (x_i-c)^2,
		\]
		so the first part of the proposition is satisfied. Subbing in the value $x_i(F)=c$ we get the second part.
	\end{proof}
	
	By Propositions \ref{pr:optimalpath} and \ref{pr:firstderivativehorofunction} all members of $\fibermap^{-1}(q)$ share their values along $\tray{q}{\cdot}$, as well as the directional derivatives at the points of the geodesic. For a given $q$ we have $x_i(\lambda H(q))=\lambda$ for all $i$ and all $\lambda>0$. Hence, Proposition \ref{pr:firstderivative} shows a similar relation for the representations of the elements of $\fibermap^{-1}(q)$ in the Gardiner--Masur compactification, as they share their value, as well as some derivatives, at all foliations of the form $\lambda H(q)$. 
	
	As shown by Fortier Bourque \cite{Fortier}, the Gardiner--Masur boundary contains extremal length functions, so we can use Proposition \ref{pr:firstderivative} to get some information on the differentials of these functions. Namely, we recover in a more restricted setting the following result, proven in \cite[Theorem 1.1]{Miyachi2}.
	\begin{theorem}[Miyachi]\label{th:differentialextremallenght}
		Let $G_t$, $t\in[0,t_0]$ be a path in the space of measured foliations on $X$ which admits a tangent vector $\dot{G}_0$ at $t=0$ with respect to the canonical piecewise linear structure. Then, the extremal length $\ext(G,X)$ is right-differentiable at $t=0$ and satisfies
		\[
			\left.\frac{d}{dt^+} \ext(G_t,X)\right\vert_{t=0}=2 i(\dot{G}_0, F_{G_0,X}),
		\]
		where $F_{G_0,X}$ is the horizontal foliation of the Hubbard--Masur differential associated to $G_0$ on $X$.
	\end{theorem}
	
	The concrete extremal length functions in the Gardiner--Masur boundary we are going to use are given by the following theorem.
	
	\begin{theorem}[Fortier Bourque]\label{th:maxsresult}
		Let $\{w_1,\ldots,w_k\}$ be weights with $w_i>0$, let $\phi_n=\tau_1^{\lfloor n w_1 \rfloor}\circ \dots \circ \tau_k^{\lfloor n w_k \rfloor}$ be a sequence of Dehn multitwist around a multicurve $\{\alpha_1,\ldots ,\alpha_k\}$ in a surface $S$ and let $X\in \T(S)$. Then the sequence $\phi_n(X)$ converges to 
		\[
			\left[\ext^{1/2}\left(\sum_{i=1}^k w_i i(F,\alpha_i)\alpha_i,X\right)\right]_{F\in \MF(S)}
		\]
		in the projective Gardiner--Masur compactification as $n\to\infty.$
	\end{theorem}
	
	The precise statement of this result is slightly weaker \cite[Corollary 3.4]{Fortier}, but the same proof yields this extension.
		
	Fix a multicurve $\{\alpha_1,\ldots ,\alpha_k\}$, weights $\{w_1,\ldots, w_k\}$ and denote $\alpha=\sum w_i \alpha_i$. Furthermore, normalize the weights $\{w_1,\ldots, w_k\}$ so that there is a unit area quadratic differential $q$ such that $V(q)=\alpha$. Denote $V_i$ the vertical components of $V(q)$. That is, $V_i=w_i\alpha_i$. We are able to recover Miyachi's formula when $i(V_i,H(q))=w_i$ for all $i$. The sequence $\phi_n(X)$ converges in the visual compactification based at $X$ to $q\in T_X\T(S)$. By Theorem \ref{th:maxsresult} the function $f(F)=\lambda^{1/2}\ext^{1/2}\left(\sum_{i=1}^k w_i i(F,\alpha_i)\alpha_i,X\right)$ is in $\Xi^{-1}\fibermap^{-1}(q)$ for some $\lambda>0$. We have $i(F,\alpha_i)=x_i(F)i(V_i,H(q))/w_i$. So, assuming $i(V_i,H(q))=w_i$ we can write
	\[
		f^2(F)=\lambda\ext\left(\sum_{i=1}^k x_i(F)V_i,X\right).
	\]
	We have $x_i(H(q))=1$ for all $i$, so by Proposition \ref{pr:firstderivative} the value of $\lambda$ satisfies
	\[
		f^2(H(q))=\lambda\ext\left(V(q),X\right)=1.
	\]
	Since $q$ has unit area, $\ext\left(V(q),X\right)=1$, so $\lambda=1$. Let $I$ be any foliation such that $H(q)+I$ is well defined, and let $F_t=H(q)+tI$. We have
	\[
		f^2(F_t)=\ext\left(\sum_i V_i + t\sum_i x_i(I) V_i,X\right).
	\]
	Hence, denoting $J=\sum x_i(I) V_i$ and $G_t=V(q)+tJ$ we can apply Proposition \ref{pr:firstderivative} to get
	\begin{multline*}
    \left.\frac{d}{dt^+} \ext\left(G_t,X\right)\right\vert_{t=0}=
    \sum_i  \left.\frac{d x_i}{dt}\right\vert_{t=0}
    \left.\frac{\partial f^2}{\partial x_i}\right\vert_{x_i=1}\\
    =\sum_i \frac{i(V_i,I)}{i(V_i,H(q))}\cdot 2i(V_i,H(q))
    =2i(V(q),I).
	\end{multline*}
	On the other hand, applying Miyachi's Theorem \ref{th:differentialextremallenght} directly we get
	\begin{multline*}
		\left.\frac{d}{dt^+} \ext\left(G_t,X\right)\right\vert_{t=0}=2i(H(q),J)=2\sum_i i(H(q),V_i)x_i(I)\\
		=2\sum_i i(H(q),V_i)\frac{i(V_i,I)}{i(H(q),V_i)}=2i(V(q),I),
	\end{multline*}
	so both expressions coincide, and we have recovered Theorem \ref{th:differentialextremallenght} in this rather restricted setting. We would like to note that Proposition \ref{pr:firstderivative} also gives some information for finding the second derivatives around the point $H(q)$. Namely, the second derivatives cannot diverge to infinity as we approach $H(q)$.
	
	Combining Proposition \ref{pr:firstderivative} with Proposition \ref{pr:busemanclosureshape} we get fairly restrictive necessary conditions for the points in $\Bclose(q)$ for surfaces without boundary. We shall be using these conditions in \cref{se:nondensity} to prove that Busemann points are not dense in the horoboundary. Now we prove a more straightforward consequence. For a topological space $U$, denote $\dim(U)$ its Lebesgue dimension. See the book by Munkres \cite[Chapter 5.80]{Munkres} for some background on basic dimension theory. Given an embedding $U\hookrightarrow V$ we have $\dim(U)\le \dim(V)$, so the conditions for the points on $\Bclose(q)$ gives us the following result. 
		
	\begin{corollary}
		Let $S$ be a surface without boundary. Let $q$ be a quadratic differential such that $V(q)$ has $n$ indecomposable components. Then,
		\[
			\dim(\Bclose(q))\le \frac{n(n-1)}{2}.
		\]
	\end{corollary}
	\begin{proof}
		By Proposition \ref{pr:busemanclosureshape} we have an embedding of $\Bclose(q)$ into the space of homogeneous polynomials of degree $2$. For a given $\xi\in \Bclose(q)$, let $b_{i,j}^\xi$ be the coefficient of $x_i x_j$. Adding the restriction $b_{i,j}=b_{j,i}$ we have a coefficient for each possible pair, so the dimension of homogeneous polynomials of degree $2$ is equal to the number of possible pairs, that is, $\frac{n(n+1)}{2}$. Furthermore, by Proposition \ref{pr:firstderivative} we know the value of the first derivatives at $x_i=c$ for all $i$. For each $i$ this gives us the linear equation $\sum_{j\neq i} b_{i,j}^\xi + 2 b_{i,i}^\xi=2 i(V_i,H(q))$. These $n$ equations are linearly independent, as $b_{i,i}^\xi$ is only contained on the equation related to $x_i$. As such, the dimension of the coefficients is at most $\frac{n(n+1)}{2}-n=\frac{n(n-1)}{2}$.
		
		We note that the sum of the coefficients being $1$ is the equation we get when summing the $n$ equations given by the derivatives, so we cannot use that to restrict further the dimension.
	\end{proof}

	Recall that the number of indecomposable components $n$ is bounded in terms of the topology of the surface. Hence, the previous corollary gives us a uniform upper bound on the dimension of $\Bclose(q)$.
	More interestingly, we can also get a lower bound for the dimension of $\Bclose(q)$. This allows us to get a lower bound on the dimension of $\fibermap^{-1}(q)$. Furthermore, as this is a lower bound, we do not need to restrict ourselves to surfaces without boundary, as the set of Busemann points always contains the set of Busemann points of the form $B(q)$. The bound is obtained by finding a dimensionally big set of different ways to approach a certain $q$ along the boundary and showing that each of these different approaches results in different limits for the associated Busemann points.
	
	\begin{theorem}\label{th:dimensionfiberslowerbound}
		Let $S$ be a surface of genus $g$ with $b_m$ and $b_u$ boundaries with and without marked points respectively and $p$ interior marked points.
		Then there is some unit quadratic differential $q$ such that 
		\[\dim (\Bclose(q))\ge 2\left\lfloor\frac{g+b_m}{2}+\frac{b_u+p}{4}-\sigma(g,b_u+p)\right\rfloor,\]
		where $\sigma$ has value
		\begin{itemize}
			\item 0 if $g\ge 2$,
			\item 1/4 if $g=1$ and $b_u+p\ge 1$,
			\item 1/2 if $g=1$ and $b_u+p=0$ or $g=0$ and $b_u+p\ge 2$,
			\item 3/4 if $g=0$ and $b_u+p=1$ and
			\item 1 if $g=0$ and $b_u+p=0$.
		\end{itemize}
	\end{theorem}

	\begin{proof}
		For simplicity we shall first do the proof in the case where $b_m=b_u=p=0$, and $g\ge 2$. Let $q$ be the quadratic differential such that $V(q)$ is the union of the closed curves $V_1,\ldots, V_{3C}$ shown in \cref{fi:v}, where $C= \lfloor g/2 \rfloor$. Let $U\subset \R^{3C}$ be the space of vectors $(\alpha_1,\alpha_2,\ldots,\alpha_{3C})$ with positive coefficients and such that 
		\begin{equation}\label{eq:dimensionproof1}
			\alpha_{3k+1}+\alpha_{3k+2}+\alpha_{3k+3}=\frac{1}{C}.	
		\end{equation}
		Each independent linear restriction reduces the dimension of the set $U$ by $1$, so $\dim U = 2C$. Hence, to prove the simplest case of the theorem it suffices to build an injective continuous map from $U$ to $\Bclose(q)$.
		
		Choose $\alpha\in U$ and consider the multicurve $\gamma^\alpha=\sum \alpha_i G_i$, where $G_i$ are as in \cref{fi:v}. We will shortly show that by applying Dehn twists about the closed curves $V_i$ to $\gamma^\alpha$ we can get a sequence of multicurves approaching $V(q)$. We can then take the sequences of associated Busemann points, which as we will see converge to distinct points in $\fibermap^{-1}(q)$. We will define the injective continuous map from $U$ to $\fibermap^{-1}(q)$ by setting it as the limit of the associated sequence of Busemann points, giving us the theorem.
		
		Let $\tau_i$ be the Dehn twist around $V_i$, and let $w^\alpha_{i}$ be such that 
		\begin{equation}\label{eq:dimensionproof2}
		w^\alpha_{3k+1}(\alpha_{3k+2}+\alpha_{3k+3})=
		w^\alpha_{3k+2}(\alpha_{3k+3}+\alpha_{3k+1})=
		w_{3k+3}(\alpha_{3k+1}+\alpha_{3k+2})=\frac{1}{3C}.
		\end{equation}
		
		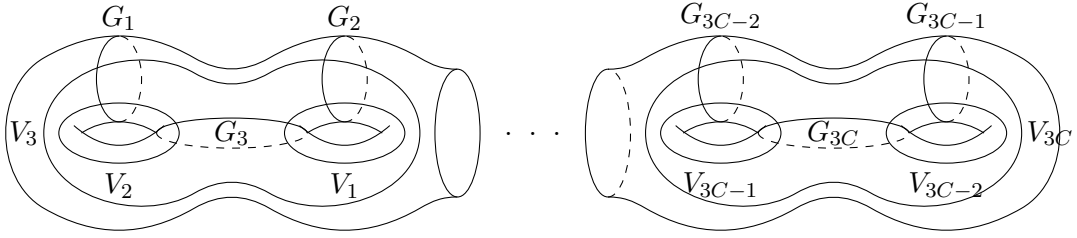
\begin{figure}\centering
			\scalebox{0.8}{
                \begin{tikzpicture}
    				\draw[smooth] (5.5,-0.85) to[out=180,in=30] (5,-1) to[out=210,in=-30] (3,-1) to[out=150,in=30] (2,-1) to[out=210,in=-30] (0,-1) to[out=150,in=-150] (0,1) to[out=30,in=150] (2,1) to[out=-30,in=210] (3,1) to[out=30,in=150] (5,1) to[out=-30,in=180] (5.5,0.85);
    				\draw[smooth] (0.4,0.1) .. controls (0.8,-0.25) and (1.2,-0.25) .. (1.6,0.1);
    				\draw[smooth] (0.5,0) .. controls (0.8,0.2) and (1.2,0.2) .. (1.5,0);
    				\draw[smooth] (3.4,0.1) .. controls (3.8,-0.25) and (4.2,-0.25) .. (4.6,0.1);
    				\draw[smooth] (3.5,0) .. controls (3.8,0.2) and (4.2,0.2) .. (4.5,0);
    				\node at (6.5,0) {$. \; \; . \; \; .$};
    				\draw[smooth] (7.5,0.85) to[out=0,in=210] (8,1) to[out=30,in=150] (10,1) to[out=-30,in=210] (11,1) to[out=30,in=150] (13,1) to[out=-30,in=30] (13,-1) to[out=210,in=-30] (11,-1) to[out=150,in=30] (10,-1) to[out=210,in=-30] (8,-1) to[out=150,in=0] (7.5,-0.85);
    				\draw[smooth] (8.4,0.1) .. controls (8.8,-0.25) and (9.2,-0.25) .. (9.6,0.1);
    				\draw[smooth] (8.5,0) .. controls (8.8,0.2) and (9.2,0.2) .. (9.5,0);
    				\draw[smooth] (11.4,0.1) .. controls (11.8,-0.25) and (12.2,-0.25) .. (12.6,0.1);
    				\draw[smooth] (11.5,0) .. controls (11.8,0.2) and (12.2,0.2) .. (12.5,0);
    				
    				\draw (5.5,-0.85) arc(270:90:0.3 and 0.85);
    				\draw (5.5,-0.85) arc(270:450:0.3 and 0.85);
    				\draw (7.5,-0.85) arc(270:90:0.3 and 0.85);
    				\draw[dashed] (7.5,-0.85) arc(270:450:0.3 and 0.85);

    				\node [label=center:$G_3$] at (2.5,0) {};
    				\draw (1.5,0) arc(180:0:1 and 0.2);
    				\draw[dashed] (1.5,0) arc(180:0:1 and -0.2);
    				
    				\node [label=above:$G_1$] at (1,1.1) {};
    				\draw (1.0,0.15) arc(270:90:0.3 and 1.14/2);
    				\draw[dashed] (1.0,0.15) arc(270:450:0.3 and 1.14/2);
    				
    				\node [label=above:$G_2$] at (4,1.1) {};
    				\draw (4.0,0.15) arc(270:90:0.3 and 1.14/2);
    				\draw[dashed] (4.0,0.15) arc(270:450:0.3 and 1.14/2);
    				
    				\node [label=center:$G_{3C}$] at (10.5,0) {};
    				\draw (9.5,0) arc(180:0:1 and 0.2);
    				\draw[dashed] (9.5,0) arc(180:0:1 and -0.2);
    				
    				\node [label=above:$G_{3C-2}$] at (9,1.1) {};
    				\draw (9.0,0.15) arc(270:90:0.3 and 1.14/2);
    				\draw[dashed] (9.0,0.15) arc(270:450:0.3 and 1.14/2);
    				
    				\node [label=above:$G_{3C-1}$] at (12,1.1) {};
    				\draw (12.0,0.15) arc(270:90:0.3 and 1.14/2);
    				\draw[dashed] (12.0,0.15) arc(270:450:0.3 and 1.14/2);
    				
    				\node [label=center:$V_2$] at (1,-0.7) {};
    				\draw  (1,0) ellipse (0.8 and 0.4);
    				
    				\node [label=center:$V_1$] at (4,-0.7) {};
    				\draw  (4,0) ellipse (0.8 and 0.4);
    				
    				\node [label=left:$V_{3}$] at (0.2,0) {};
    				\draw[smooth] (5,0) to[out=270,in=-30] (3,-0.8) to[out=150,in=30] (2,-0.8) to[out=210,in=270] (0,0)  to[out=90,in=150] (2,0.8) to[out=-30,in=210] (3,0.8) to[out=30,in=90] (5,0) ;
    				
    				\node [label=center:$V_{3C-1}$] at (9,-0.7) {};
    				\draw  (9,0) ellipse (0.8 and 0.4);
    				
    				\node [label=center:$V_{3C-2}$] at (12,-0.7) {};
    				\draw  (12,0) ellipse (0.8 and 0.4);
    				
    				\node [label=right:$V_{3C}$] at (12.8,0) {};
    				\draw[smooth] (13,0) to[out=270,in=-30] (11,-0.8) to[out=150,in=30] (10,-0.8) to[out=210,in=270] (8,0)  to[out=90,in=150] (10,0.8) to[out=-30,in=210] (11,0.8) to[out=30,in=90] (13,0) ;
    			\end{tikzpicture}
            }
			\caption{Labeling of the curves when the surface has no boundaries nor marked points. If $g$ is odd then there is an unused handle.}
			\label{fi:v}
		\end{figure}
	
		Define $\phi^\alpha_n=\tau_1^{\lfloor w^\alpha_1 n \rfloor}\circ\tau_2^{\lfloor w^\alpha_2 n \rfloor} \circ\dots \circ\tau_{3C}^{\lfloor w^\alpha_{3C} n \rfloor}$. For $1\le k \le C$ and $j\in\{1,2,3\}$ Denote $F_{k,j}^\alpha=\sum_{i\in\{1,2,3\}-j}w_{3k+i}^\alpha V_{3k+i}$. By counting the intersections between the curves $V_i$ and $G_i$ we have that there is some sequence $\lambda_n$ such that $\lambda_n\phi^\alpha_n G_{3k+j}$ converges to $F_{k,j}^\alpha$ for all $k,j$ as $n\to\infty$. By the conditions on the weights,  $\lambda_n\phi^\alpha_n\gamma^\alpha$ converges to $V(q)$. Let $q^\alpha_n$ be the quadratic differential associated to $\lambda_n\phi^\alpha_n \gamma^\alpha$. Since $\lambda_n\phi^\alpha_n\gamma^\alpha$ converges to $V(q)$, we have that $q_n$ converges to $q$, so all accumulation points of $(B(q_n))$ are in $\fibermap^{-1}(q)$. We know that $(\Xi^{-1}B(q^\alpha_n))^2=\sum_i  \wals^q(\alpha_i \lambda_n\phi^\alpha_n G_i)$, so by Lemma \ref{le:limith} we have
		\[
		(\xi^\alpha)^2=\lim_{n\to\infty} (\Xi^{-1}B(q^\alpha_n))^2=
		\sum_{k=0}^{C-1}\sum_{j\in\{1,2,3\}}
		\alpha_{3k+j} \wals^q(F_{k,j}^\alpha).
		\]
		Define then the map from $U$ to $\fibermap^{-1}(q)$ sending $\alpha\in U$ to $\Xi\xi^\alpha\in \fibermap^{-1}(q)$.
		As before, we shall denote $x_i:=\frac{i(V_i,\cdot)}{i(V_i,H(q))}=3C i(V_i,\cdot)$. With this notation we have 
		\[
		\wals^q(F_{k,j}^\alpha)=\frac{i(F_{k,j}^\alpha,\cdot)^2}{i(F_{k,j}^\alpha,H(q))}=\frac{\left(\sum_{i\notin\{1,2,3\}-j}w^\alpha_{3k+i}x_{3k+i}\right)^2}{3C\sum_{i\notin\{1,2,3\}-j}w^\alpha_{3k+i}}.
		\]
		That is, given $\alpha$ we know precisely the shape of the polynomial $\xi^\alpha$. Since $\alpha$ has positive coefficients, each of the $w_i^\alpha$ depends continuously on $\alpha$, so $\xi^\alpha$ depends continuously on $\alpha$. 
		
		It remains to show injectivity. Let $\beta\in U$ be such that $\xi^{\alpha}=\xi^{\beta}$. While we have equated two polynomials, we cannot conclude directly that the coefficients are equal, as these cannot be evaluated for arbitrary values. However, we can evaluate at elements of the form $b_1 G_{3k+1}+b_2 G_{3k+2}+b_3 G_{3k+3}$ for $b_1,b_2,b_3\ge 0$, which is enough to prove that $\xi^{\alpha}$ and $\xi^{\beta}$ have the same coefficients.
		
		Equating then the coefficients for $x_{3k+1}x_{3k+2}$, $x_{3k+2}x_{3k+3}$ and $x_{3k+1}x_{3k+3}$ we get
		
		\begin{align*}
			\frac{\alpha_{3k+1} w^\alpha_{3k+2}w^\alpha_{3k+3}}{w^\alpha_{3k+2}+w^\alpha_{3k+3}}=&\frac{\beta_{3k+1} w^\beta_{3k+2}w^\beta_{3k+3}}{w^\beta_{3k+2}+w^\beta_{3k+3}},\\
			\frac{\alpha_{3k+2} w^\alpha_{3k+1}w^\alpha_{3k+3}}{w^\alpha_{3k+1}+w^\alpha_{3k+3}}=&\frac{\beta_{3k+2} w^\beta_{3k+1}w^\beta_{3k+3}}{w^\beta_{3k+1}+w^\beta_{3k+3}} \quad \text{ and }\\
			\frac{\alpha_{3k+3} w^\alpha_{3k+1}w^\alpha_{3k+2}}{w^\alpha_{3k+1}+w^\alpha_{3k+2}}=&\frac{\beta_{3k+3} w^\beta_{3k+1}w^\beta_{3k+2}}{w^\beta_{3k+1}+w^\beta_{3k+2}}.
		\end{align*}
	
		Dividing these equalities and using equations \eqref{eq:dimensionproof1} and \eqref{eq:dimensionproof2} we get
		
		\begin{align*}
			\frac{\alpha_{3k+1}}{\alpha_{3k+2}}\frac{(1/C+\alpha_{3k+2})}{(1/C+\alpha_{3k+1})}=&
			\frac{\beta_{3k+1}}{\beta_{3k+2}}\frac{(1/C+\beta_{3k+2})}{(1/C+\beta_{3k+1})},\\
			\frac{\alpha_{3k+2}}{\alpha_{3k+3}}\frac{(1/C+\alpha_{3k+3})}{(1/C+\alpha_{3k+2})}=&
			\frac{\beta_{3k+2}}{\beta_{3k+3}}\frac{(1/C+\beta_{3k+3})}{(1/C+\beta_{3k+2})}  \quad \text{ and } \\
			\frac{\alpha_{3k+3}}{\alpha_{3k+1}}\frac{(1/C+\alpha_{3k+1})}{(1/C+\alpha_{3k+3})}=&
			\frac{\beta_{3k+3}}{\beta_{3k+1}}\frac{(1/C+\beta_{3k+1})}{(1/C+\beta_{3k+3})}.
		\end{align*}
	
		Rearranging the first equality we have
		\begin{equation}\label{eq:increasingfactors}
			\frac{\alpha_{3k+1}}{\beta_{3k+1}}
			\frac{\beta_{3k+2}}{\alpha_{3k+2}}
			=
			\frac{(1/C+\alpha_{3k+1})}{(1/C+\beta_{3k+1})}
			\frac{(1/C+\beta_{3k+2})}{(1/C+\alpha_{3k+1})}.
		\end{equation}
		If $\frac{\alpha_{3k+1}}{\beta_{3k+1}}<1$ we have $\frac{(1/C+\alpha_{3k+1})}{(1/C+\beta_{3k+1})}>\frac{\alpha_{3k+1}}{\beta_{3k+1}}$, and if $\frac{\alpha_{3k+2}}{\beta_{3k+2}}>1$ we have $\frac{(1/C+\alpha_{3k+1})}{(1/C+\beta_{3k+1})}<\frac{\alpha_{3k+1}}{\beta_{3k+1}}$. Assume then that $\alpha_{3k+1}<\beta_{3k+1}$. One of the factors of the left hand side of the product in \cref{eq:increasingfactors} is replaced in the right hand side by a larger value. Hence, the other factor has to be replaced by a smaller value. That is, the inequality 
		$\alpha_{3k+2}<\beta_{3k+2}$ has to be satisfied. Similarly, if $\alpha_{3k+2}<\beta_{3k+2}$ we have $\alpha_{3k+3}<\beta_{3k+3}$. Equation \eqref{eq:dimensionproof1} leads to \[\frac{1}{C}=\alpha_{3k+1}+\alpha_{3k+2}+\alpha_{3k+3}<\beta_{3k+1}+\beta_{3k+2}+\beta_{3k+3}=\frac{1}{C},\]
		which is a contradiction. Similarly, $\alpha_{3k+1}>\beta_{3k+1}$ leads to another contradiction, so $\alpha_{3k+1}=\beta_{3k+1}$, which leads to $\alpha=\beta$.
		Therefore, 
		$\dim(\Bclose(q))\ge \dim(U)=2\left\lfloor \frac{g}{2}\right\rfloor$.
		
		Assume now that $g\ge 2$ and there are some marked points or boundaries. For each pair of marked points or unmarked boundaries, or for each marked boundary we can repeat the proof with an extra genus, by replacing the curves $G_i$ by the curves shown in \cref{fi:replacement}, and halving the associated weights for $w_i$, as the curves intersect now twice the vertical components instead of once.
		
	\begin{figure} \centering
		\begin{tikzpicture}
			\draw[smooth] (5.5,-0.85) to[out=180,in=30] (5,-1) to[out=210,in=-30] (3,-1) to[out=150,in=30] (2,-1) to[out=210,in=-30] (0,-1)  to[out=150,in=0] (-0.5,-0.85);
			\draw[smooth](-0.5,0.85) to[out=0,in=210] (0,1) to[out=30,in=150] (2,1) to[out=-30,in=210] (3,1) to[out=30,in=150] (5,1) to[out=-30,in=180] (5.5,0.85);
			\draw (-0.5,-0.85) arc(270:90:0.3 and 0.85);
			\draw[dashed] (-0.5,-0.85) arc(270:450:0.3 and 0.85);
			\node[circle,fill,inner sep=1pt] at (3.7,0) {};
			\draw (5.5,-0.85) arc(270:90:0.3 and 0.85);
			\draw (5.5,-0.85) arc(270:450:0.3 and 0.85);
			
			\draw  (4.6,0) ellipse (0.2 and 0.2);
			
			\node [label=center:$G_{3k+3}$] at (2.5,0) {};
			\draw[smooth](1.4,0.2) to[out=20,in=90] (4,0)to[out=270, in=-20 ](1.4,-0.2) ; 
			
			\node [label=above:$G_{3k+2}$] at (4.5,1.1) {};
			\draw[dashed] (4.5,-1.21) arc(270:450:0.3 and 1.21);
			\draw (4.5,-1.21) arc(270:90:0.3 and 1.21);
			
			\node [label=above:$G_{3k+1}$] at (1.25,1.1) {};
			\draw (1.25,-1.27) arc(270:90:0.3 and 1.27);
			\draw  [fill=white](1,0) ellipse (0.45 and 0.45);
			\draw [dashed] (1.25,-1.27) arc(270:450:0.3 and 1.27);
			\node[circle,fill,inner sep=1pt] at (0.55,0) {};
		\end{tikzpicture}
		\begin{tikzpicture}
			\draw[smooth] (5.5,-0.85) to[out=180,in=30] (5,-1) to[out=210,in=-30] (3,-1) to[out=150,in=30] (2,-1) to[out=210,in=-30] (0,-1)  to[out=150,in=0] (-0.5,-0.85);
			\draw[smooth](-0.5,0.85) to[out=0,in=210] (0,1) to[out=30,in=150] (2,1) to[out=-30,in=210] (3,1) to[out=30,in=150] (5,1) to[out=-30,in=180] (5.5,0.85);
			\draw (-0.5,-0.85) arc(270:90:0.3 and 0.85);
			\draw[dashed] (-0.5,-0.85) arc(270:450:0.3 and 0.85);
					
			\draw (5.5,-0.85) arc(270:90:0.3 and 0.85);
			\draw (5.5,-0.85) arc(270:450:0.3 and 0.85);
					
			\node [label=above:$G_{3k+1}$] at (1,1.1) {};
			\draw (1.0,0.15) arc(270:90:0.3 and 1.14/2);
			\draw[dashed] (1.0,0.15) arc(270:450:0.3 and 1.14/2);
			
			\draw[smooth] (0.4,0.1) .. controls (0.8,-0.25) and (1.2,-0.25) .. (1.6,0.1);
			\draw[smooth] (0.5,0) .. controls (0.8,0.2) and (1.2,0.2) .. (1.5,0);
			
			\node [label=center:$G_{3k+3}$] at (2.5,0) {};
			\draw[smooth](4,0) to[out=200,in=-20] (1.3,-0.1) ; 
			\draw[smooth,dashed](1.3,-0.1) to [out= 230, in = 0](-0,-1) ;
			\draw[smooth](4,-0.3) to[out=200,in=-30] (0.7,-0.1) ; 
			\draw[smooth,dashed](0.7,-0.1) to [out= 230, in = -30](0.2,-0.2) to [out= 150,in = 0](-0,1) ;
			\draw (0,1) arc(120:240:0.3 and 1.15);
			
			\node [label=above:$G_{3k+2}$] at (4.5,1.1) {};
			\draw[dashed] (4.5,-1.21) arc(270:450:0.3 and 1.21);
			\draw (4.5,-1.21) arc(270:90:0.3 and 1.21);
			
			\draw  [fill=white](3.7+0.45,0) ellipse (0.45 and 0.45);
			\node[circle,fill,inner sep=1pt] at (3.7+0.9,0) {};
		\end{tikzpicture}
		\caption{Each pair of marked points and boundary components without marked points can replace a genus, as well as each boundary with marked points.}
		\label{fi:replacement}
	\end{figure}
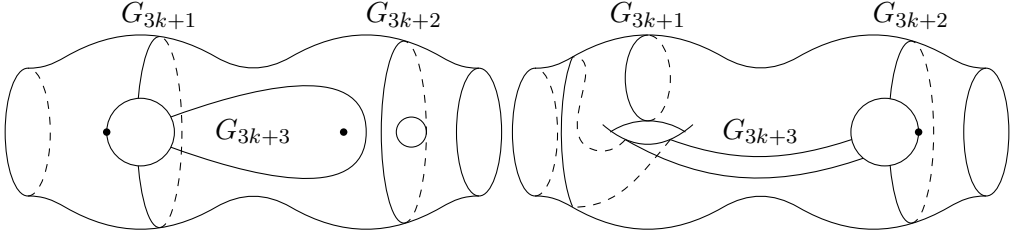
	
		If $g=1$ we need to place at least one feature at one of the ends to prevent the curve $G_1$ from being contractible or parallel to a unmarked boundary, so if we have marked points or boundaries without marked points we place these, as boundaries with marked points are more effective at increasing the dimension. In this way we get that if $b_u+p\ge 1$ then 
		\[
		\dim(\Bclose(q))\ge 2 \left\lfloor \frac{g+b_m}{2}+\frac{b_u+p-1}{4} \right\rfloor
		\]
		and if $b_u+p=0$ then
		\[
		\dim(\Bclose(q))\ge 2 \left\lfloor \frac{g+b_m-1}{2}\right\rfloor.
		\]
				
		Lastly, if $g=0$ we need to place two elements, one at each end. Using the same choice as we took for $g=1$ we get
		\[
		\dim(\Bclose(q))\ge 2 \left\lfloor \frac{b_m}{2}+\frac{b_u+p-2}{4} \right\rfloor \text{ for } b_u+p\ge 2,
		\]
		\[
		\dim(\Bclose(q))\ge 2 \left\lfloor \frac{b_m-1}{2}\right\rfloor \text{ for } b_u+p=1 \text{ and }
		\]
		\[
		\dim(\Bclose(q))\ge 2 \left\lfloor \frac{b_m-2}{2}\right\rfloor \text{ for }b_u+p=0.
		\]
	\end{proof}
	
	We would like to note that this lower bound is does not look optimal to us. Furthermore, the method used is restricted to getting to the dimension of the closure of Busemann points, so 
	the dimension of the whole fiber may be significantly larger than what could be achieved by refining the strategy from the proof.

\section{Non density of the Busemann points}\label{se:nondensity}
\subsection{Busemann points are not dense in the horoboundary}
    By Proposition \ref{pr:busemanclosureshape} we know that points in the closure of Busemann points are smooth in the Gardiner--Masur representation with respect to certain variables. By showing that at least one point in the horoboundary is not smooth with respect to the corresponding variables we will prove that Busemann points are not dense. The points we use for this analysis are once again the ones found by Fortier Bourque in Theorem \ref{th:maxsresult}.
	
	Following Fortier Bourque's reasoning, we shall first prove the non density for the sphere with five marked points, and then lift to general closed surfaces by using the branched coverings given by the following Lemma, found in \cite[Lemma 7.1]{Gekhtman}.
	\begin{lemma}[Gekhtman--Markovic]\label{le:branchingcover}
		Let $S$ be a closed surface of genus $g$ with $p$ marked points, such that $3g+p\ge5$. Then there is a branched cover $\overline{S_{g,p}}\to\overline{S_{0,5}}$ that branches at all preimages of marked points that are not marked and induces an isometric embedding $\T(S_{0,5})\hookrightarrow\T(S_{g,p})$.
	\end{lemma}
	
	The particular conformal structure given to $S_{0,5}$ is obtained as follows. Let $S^1=\R/\mathbb{Z}$ and let $C=S^1\times [-1,1]$. We obtain a sphere $\Sigma$ by sealing the top and bottom of $C$ via the relation $(x,y)\sim(-x,y)$ for all $(x,y)\in S^1\times\{-1,1\}$. Let $P$ be set consisting of the five points $(0,\pm 1)$, $(1/2,\pm 1)$ and $(0,0)$. The pair $S=(\Sigma,P)$, where we view $\Sigma$ as a topological space, is the sphere with five marked points. We get a point $X$ in $\T(S)$ by considering the complex structure on $\Sigma$ obtained by the construction, using the identity map as our marking. 
	
	\begin{figure} \centering
		\begin{tikzpicture}
			\draw (1,2) -- (1,-2) -- (-1,-2) -- (-1,2) --(1,2);
			\node [circle,fill,inner sep=1pt] at (-1,0) {};
			\node [circle,fill,inner sep=1pt] at (1,2) {};
			\node [circle,fill,inner sep=1pt] at (1,-2) {};
			\node [circle,fill,inner sep=1pt] at (-1,-2) {};
			\node [circle,fill,inner sep=1pt] at (-1,2) {};
			
			\draw[dashed] (-1,1) arc(180:0:1 and 0.5/2);
			\draw (-1,1) arc(180:369:1 and 0.5/2);
			\node [label=center:$\alpha$] at (0,0.5) {};
			
			\draw[dashed] (-1,-1) arc(180:0:1 and 0.5/2);
			\draw (-1,-1) arc(180:369:1 and 0.5/2);
			\node [label=center:$\beta$] at (0,-1.5) {};
		\end{tikzpicture}
		\label{fi:pillowcase}
		\caption{Sphere with five marked points, with curves $\alpha$ and $\beta$. We show that the extremal length is not $C^{2}$ along the path $\alpha+t\beta$, $t\in[0,t_0]$.}
	\end{figure}
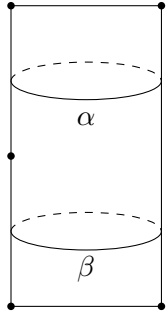

	Let $\alpha(t)=(t,1/2)$ and $\beta(t)=(t,-1/2)$ for $t\in S^1$. Denote $\tau_\alpha$ and $\tau_\beta$ the Dehn twists along $\alpha$ and $\beta$. By Fortier Bourque's theorem, the sequence $(X_n)=((\tau_\alpha\circ\tau_\beta)^n X)$ converges to a multiple of $\ext^{1/2}(i(\alpha,\cdot)\alpha+i(\beta,\cdot)\beta, X))$ in the Gardiner--Masur compactification. Furthermore, the sequence $(X_n)$ converges in the visual compactification based at $X$ to the geodesic spawned by the quadratic differential $q_{\alpha+\beta,X}$. Indeed, as detailed in \cite[Section 4]{Fortier}, the elements $(X_n)$ diverge to infinity along the horocycle defined by the quadratic differential $q_{\alpha+\beta,X}$. Hence, inside embedded hyperbolic plane associated to $q_{\alpha+\beta,X}$, the sequence $(X_n)$ converges in the visual boundary to the geodesic spawned by $q_{\alpha+\beta,X}$, and so the same occurs in the ambient space. That is, $\Xi \ext^{1/2}(i(\alpha,\cdot)\alpha+i(\beta,\cdot)\beta, X) \in \fibermap^{-1}(q_{\alpha+\beta,X})$, so by \cref{pr:busemanclosureshape} if we show that $\ext(i(\alpha,\cdot)\alpha+i(\beta,\cdot)\beta, X)$ is not smooth with respect to the values of $i(\alpha,\cdot)$ and $i(\beta,\cdot)$, then $\Xi \ext^{1/2}(i(\alpha,\cdot)\alpha+i(\beta,\cdot)\beta, X) \notin \Bclose(q_{\alpha+\beta,X})$, and hence it is also not in $\Bclose$.
	
\begin{lemma}\label{le:extremallengthnotsmooth}
	Let $X\in \T(S_{0,5})$ and $G_t$, $t\in[0,t_0]$ be the foliation $\alpha+t\beta$ on $S_{0,5}$. The map $f(t):=\ext(G_t,X)$ is not $C^{2}$.
\end{lemma}
\begin{proof}
	By Miyachi's Theorem \ref{th:differentialextremallenght} we have
	\[\frac{d}{dt}\ext(G_t,X)=2 i(\beta,F_{G_t,X}),\]
	where we remind that $F_{G_t,X}$ is the horizontal foliation of the unique Hubbard--Masur differential associated to $G_t$ on $X$. Hence, the Lemma is equivalent to proving that $g(t)=i(\beta,F_{G_t,X})$ is not $C^{1}$.
	
	For a general surface finding a precise expression of $F_{G,X}$ is a complicated problem, as the relation established by Hubbard and Masur is not explicit. However, in our case the surface is topologically simple, and one can use Schwartz--Christoffel maps to get a map from $G$ to $F_{G,X}$. In particular, it is possible to show that the sphere with 5 marked points is conformally equivalent to the Riemannian surface obtained by doubling an $L$-shaped polygon, marking the inner angles as shown in \cref{fi:lshaped} and setting certain values for $a,b$ and $l$. Furthermore, the quadratic differential obtained by $dz^2$ has $\alpha$ and $\beta$ as vertical foliations, with weights $a$ and $b$. Hence $q_{G_t,X}$ is $dz^2$ on the $L$-shaped pillowcase where $a=1$ and $b=t$, so $i(\beta,F_{G_t,X})=2l$. Markovic estimated in \cite[Section 9]{Markovic} the values of $a,b$ and $l$ around $b=0$ depending on a common parameter $r$. Up to rescaling, these values are given by
	\begin{align*}
		a(r)=&a(0)+D_1r+O(r^2),\\
		b(r)=&D_2r+O(r^2) \text{ and } \\
		l(r)=&l(0)+D_3r\log \frac{1}{r} + o\left(r\log \frac{1}{r}\right),
	\end{align*}
where $A(r)=B(r)+O(f(r))$ means $\frac{|A(r)-B(r)|}{f(r)}$ is bounded around $r=0$, and $A(r)=B(r)+o(f(r))$ means $\frac{|A(r)-B(r)|}{f(r)}$ converges to $0$ as $r$ converges to $0$.
	
	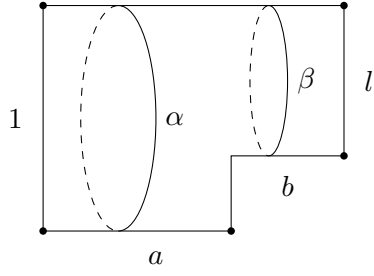
\begin{figure} \centering
		\begin{tikzpicture}
			\draw (-2,-3) -- (-2,0) -- (2,0) -- (2,-2) -- (0.5,-2) -- (0.5,-3) --(-2,-3);
			\node [label=below:$a$] at (-0.5,-3) {};
			\node [label=right:$l$] at (2,-1) {};
			\node [label=below:$b$] at (1.25,-2) {};
			\node [label=left:$1$] at (-2,-1.5) {};
			
			\draw[dashed] (-1,0) arc(90:270:0.5 and 1.5);
			\draw (-1,0) arc(90:-90:0.5 and 1.5);
			\node [label=center:$\alpha$] at (-0.25,-1.5) {};
			
			\draw[dashed] (1,0) arc(90:270:0.25 and 1);
			\draw (1,0) arc(90:-90:0.25 and 1);
			\node [label=center:$\beta$] at (1.5,-1) {};
			
			\node [circle,fill,inner sep=1pt] at (-2,-3) {};			
			\node [circle,fill,inner sep=1pt] at (-2,0) {};			
			\node [circle,fill,inner sep=1pt] at (2,0) {};	
			\node [circle,fill,inner sep=1pt] at (2,-2) {};			
			\node [circle,fill,inner sep=1pt] at (0.5,-3) {};			
		\end{tikzpicture}
		\caption{Doubling of the $L$-shaped polygon together with the curves $\alpha$ and $\beta$.}
		\label{fi:lshaped}
	\end{figure}
	
	Rescaling the pillowcase by $1/a(r)$ we see that the parameter $t$ can be expressed as $t(r)=b(r)/a(r)$, and $g(t(r))=i(\beta,F_{G_t,X})=2l(r)/a(r)$. Observing that $t(0)=0$, we can evaluate the first derivative of $g(t)$ at $0$ by evaluating the limit
	\begin{multline*}
		\lim_{h\to 0}\frac{g(h)-g(0)}{h}=\lim_{r\to 0}\frac{g(t(r))-g(0)}{t(r)}=\lim_{r\to 0}\frac{2l(r)/a(r)-2l(0)/a(0)}{b(r)/a(r)}=\\
  2\lim_{r\to 0}\frac{l(r)-l(0)a(r)/a(0)}{b(r)}
		=2\lim_{r\to 0}\frac{D_3r\log(\frac{1}{r})+o(r\log(\frac{1}{r}))-\frac{l(0)D_1}{a(0)} r}{D_2r+O(r^2)}=\infty.	
	\end{multline*}
	
	And so, $g(t)$ is not differentiable at $t=0$, and hence $f(t)$ is not $C^2$.
\end{proof}

Repeating Fortier Bourque's reasoning we can lift this example to any surface of genus $g$ with $p$ marked points as long as $3g+p\ge5$. Besides  Gekhtman--Markovic's Lemma \ref{le:branchingcover},
the other key ingredient for the lifting is the following result.
\begin{lemma}[Fortier Bourque]\label{le:lifting}
	Let $\pi:S_{g,p}\to S_{0,5}$ be a branched cover of degree $d$ and let $\iota:\T(S_{0,5})\hookrightarrow \T(S_{g,p})$ be the induced isometric embedding. For any measured foliation $F$ on $S_{0,5}$ and any $X\in\T(S_{0,5})$, we have the identity 
	\[\ext(\pi^{-1}(F),\iota(X))=d\ext(F,X).\]
\end{lemma}
\begin{proof}
	Recall that $q_{F,X}$ is the Hubbard--Masur differential associated to $\gamma$. We have that $\pi^*q_{F,X}=q_{\pi^{-1}(F),\iota(X)}$, so
	\[
	\ext(\pi^{-1}(F),\iota(X))=
	\int_{\iota(X)}|q_{\pi^{-1}(F),\iota(X)}|=
	d\int_X|q_{F,X}|=d\ext(F,X).
	\]
\end{proof}

Lifting the foliation $G_t$ from Lemma \ref{le:extremallengthnotsmooth} we get an upper bound for the smoothness of the extremal length. 

\begin{theorem}\label{th:extremallengthnotsmooth}
	Let $S$ be a closed surface of genus $g$ with $p$ marked points, such that $3g+p\ge5$. Then there exist two non intersecting multicurves $\hat{\alpha}$, $\hat{\beta}$ and some $X\in \T(S)$ such that the map $f(t):=\ext(\hat\alpha+t\hat\beta,X)$, $t\in[0,t_0]$ is not $C^{2}$.
\end{theorem}
\begin{proof}
	Since $3g+p\ge5$ we have a map $\pi:S_{g,p}\to S_{0,5}$, with an induced isometric embedding $\iota:\T(S_{0,5})\hookrightarrow\T(S_{g,p})$. By Lemma \ref{le:extremallengthnotsmooth} we have two curves $\alpha,\beta\in S_{0,5}$ such that, for any $X\in \T(S_{0,5})$ the map $t\to \ext(\alpha+t\beta,X)$ is not $C^2$. Let $\hat\alpha=\pi^{-1}(\alpha)$ and $\hat\beta=\pi^{-1}(\beta)$. We have $\hat\alpha+t\hat\beta=\pi^{-1}(\alpha+t\beta)$, so applying Lemma \ref{le:lifting} we get $\ext(\hat\alpha+t\hat\beta,i(X))=d\ext(\alpha+t\beta,X)$. By Lemma \ref{le:extremallengthnotsmooth} the function $\ext(\alpha+t\beta,X)$ is not $C^{2}$, so we get the theorem.
\end{proof}

Theorem \ref{th:extremallengthnotsmoothintro} is essentially a rephrasing of the previous theorem. Finally, we are able to prove that Busemann points are not dense.

\begin{proof}[Proof of Theorem \ref{th:busemannnotdense}]
  Let $\alpha$ and $\beta$ be as in \cref{le:extremallengthnotsmooth}. Furthermore, let $\pi:S_{g,p}\to S_{0,5}$ and $\iota: \T(S_{0,5})\hookrightarrow\T(S_{g,p})$ be as in \cref{le:branchingcover}. For the $X\in \T(S_{0,5})$ described before \cref{le:extremallengthnotsmooth} the sequence $(X_n)=(\tau_{\beta}\circ \tau_{\alpha})^n X$ is contained in the horocycle generated by $q_{\alpha+\beta,X}$ and the distance $d(X_n,X)$ goes to infinity. Therefore $(X_n)$  converges in $\vbc{\T(S_{0,5})}_X$ to the geodesic spawned by $q_{\alpha+\beta,X}$.  Following Fortier Bourque's reasoning in the proof of \cite[Theorem 1.1]{Fortier}, using half translation structures, applying the Dehn twist $\tau_{\alpha}\circ \tau_{\beta}$ to $X$ is equivalent to applying the shearing transformation 
 \[h_m=\begin{pmatrix}
1 & m \\
0 & 1
\end{pmatrix}
\]
to the half translation structure defined by $q_{\alpha+\beta,X}$. This action commutes with the pull-back coming from the branched cover, so the elements $(X_n)$ are associated with the half translation structure defined by $h_n \pi^*(q_{\alpha+\beta,X})$. These points diverge to infinity along the horocycle defined by $\pi^*(q_{\alpha+\beta,X})$, and so converge in $\vbc{\T(S_{g,p})}_{\iota(X)}$ to the geodesic spawned by $q_{\pi^{-1}(\alpha)+\pi^{-1}(\beta),\iota(X)}$.

Let $c_i$, $1\le i \le k$, be the components of the half translation structure associated to $\pi^{-1}(\alpha+\beta,X)$. Each $c_i$ covers either $\alpha$ or $\beta$ with some degree $d_i\in \mathbb{N}$. Hence, each component $c_i$ corresponds to a curve and is a cylindrical with height $1$ and circumference $d_i$. Therefore, if $m$ is the common multiple between all $d_i$, and $\gamma_i$ is the curve associated to the component $c_i$, shifting the flat metric via the matrix $h_m$ is equivalent to performing $m/d_i$ Dehn twists around each curve $\gamma_i$. Letting $\phi$ be the composition of such Dehn twists, we have $\iota(X_{mn})=\phi^n \iota(X)$. Hence, by Fortier Bourque's \cref{th:maxsresult}, in the Gardiner--Masur compactification the sequence $(\iota(X_{mn}))_n$ converges, as $n\to \infty$, to
\[
\xi=\left[ \ext^{1/2}\left(\sum_{i=1}^k\frac{1}{d_i}i(F,\gamma_i) \gamma_i, \iota(X)\right)\right]_{F\in \MF(S_{g,n})}.
\]
Therefore, $\Xi \xi \in \fibermap^{-1} (q_{\pi^{-1}(\alpha)+\pi^{-1}(\beta),\iota(X)})$. To see that $\Xi \xi$ is not in $\Bclose$ it remains to see that it is not in $\Bclose(q_{\pi^{-1}(\alpha)+\pi^{-1}(\beta),\iota(X)}).$ We have, $i(c_i,H(q_{\pi^{-1}(\alpha)+\pi^{-1}(\beta),\iota(X)})=d_i$, so by \cref{pr:busemanclosureshape} it remains to prove that there is some path of foliations $G_t$ such that the functions $x_i=\frac{i(\gamma_i,G_t)}{d_i}$ vary smootly, while the function $f(x_1,\ldots, x_k)=\ext\left(\sum_{i=1}^k\frac{1}{d_i}x_i\gamma_i, \iota(X)\right)$ does not. Reorder the curves so there is some $p\ge 1$ such that $\pi^{-1}\alpha=\gamma_1+\ldots +\gamma_p$ and $\pi^{-1} \beta=\gamma_{p+1}+\ldots + \gamma_k$. It follows from Dehn-Thurston's coordinates that for any natural numbers $n_j$, $1\le j \le k$ there is a multicurve $G_{(n_j)}$ such that $i(G_{(n_j)},\gamma_i)=n_i$. See, for example, the book by Penner--Harer \cite[Theorem 1.2.1]{PennHar}. Allowing renormalizations of the multicurves we get that $n_j$ can be any non-negative rationals. Finally, doing a limit argument in the space of projective measured foliations we can take $n_j$ to be any non-negative real numbers. That is, for any $t\ge 0$ there exists a measured foliation $G_t$ such that $i(G_t,\gamma_i)=d_i$ for $i\le p$, and $i(G_t,\gamma_i)=t d_i$ otherwise. Hence, along such foliations we have $x_i=1$ for $i\le p$ and $x_i= t$ otherwise. Therefore, along this path,
\[
f(1,\ldots,1,t,\ldots,t)=\ext\left(\pi^{-1}(\alpha)+t\pi^{-1}(\beta), \iota(X)\right),\]
which by \cref{th:extremallengthnotsmooth} is not smooth, as $\pi^{-1}(\alpha)$ and $\pi^{-1}(\beta)$ are the curves used in the proof of the Theorem.
\end{proof}	

\subsection{Busemann points with one indecomposable component are nowhere dense}

The Thurston compactification can be build in a similar way as the Gardiner--Masur compactification, by using the hyperbolic length of the curves instead of the extremal length. Let $\phi$ be the map between $\T(S)$ and $P\R_+^\Curv$ defined by sending $X\in \T(S)$ to the projective vector $[\ell(\alpha,X)]_{\alpha\in \Curv}$. The pair $(\phi,\overline{\phi(\T(S))})$ defines a compactification, and the boundary is given by the space of projective measured foliations, denoted $\PMF$.

As explained by Miyachi \cite{Miyachi}, neither the Thurston nor the horofunction compactification is finer than the other one. However, it is possible to get some relation. Let $\PMF^{UE}\subset \PMF$ be the set of uniquely ergodic foliations. Following the work of Masur \cite{Masur2}, $\PMF^{UE}$ has full Lebesgue measure within $\PMF$. Miyachi \cite[Corollary 1]{Miyachi} shows that the mapping $\phi$ on $\T(S)$ can be extended to an homeomorphism $f$ between $\phi(\T(S))\cup \PMF^{UE}$ and $h(\T(S))\cup B_{UE}$ such that for $x\in \T(S)$ we have $f(\phi(x))=h$, where $B_{UE}$ are the Busemann points associated to quadratic differentials whose vertical foliation is uniquely ergodic. One might understand this result as stating that the two compactifications are the same almost everywhere with respect to the Lebesgue measure on $\PMF$. As we shall see, the same does not follow with respect to any strictly positive measure on the horoboundary.

The homeomorphism $f$ described by Miyachi is obtained by first defining a map between the boundaries. For a given $x\in \T(S)$, the map on the boundary is denoted $\mathcal{G}_x$, and by its definition we have $\mathcal{G}_x(F)=B(q_{F,x})$, where we recall that $q_{F,x}$ is the quadratic differential on $x$ with $V(q_{F,x})=F$. Denote $\B_1$ the set of Busemann points associated to foliations with one indecomposable component. We have $\mathcal{G}_x(\PMF^{UE})=\B_{UE} \subset \B_1$. However, the following is also satisfied.

\begin{theorem}\label{th:nowheredense}
	Let $S$ be a closed surface of genus $g$ with $p$ marked points, such that $3g+p\ge5$. Then the set $\B_1$ is nowhere dense in the horoboundary.
\end{theorem}
\begin{proof}
	The action of $\MCG(S)$ on $\T(S)$ is extended to the projectivized version of the Gardiner--Masur compactification by $\psi[f(\alpha)]_{\alpha\in \Curv}=[f(\psi\alpha)]_{\alpha\in \Curv}$. For any $q$ such that $V(q)$ is an indecomposable measured foliation, $\E_q=\Xi^{-1}B(q)=[i(V(q),\alpha)]_{\alpha\in\Curv}$, so $\psi \E_q=[i(V(q),\psi(\alpha))]_{\alpha\in\Curv}=[i(\psi^{-1}(V(q)),\alpha)]_{\alpha\in\Curv}.$ Hence, $\psi \E_q$ is equal to the representation of the Busemann point in the Gardiner--Masur compactification associated to the quadratic differential with vertical foliation $\psi^{-1}V(q)$, which also is an indecomposable measured foliation. Therefore, $\B_1$ is invariant under the action of $\MCG(S)$, and since $\MCG(S)$ acts by homeomorphisms, the complement of the closure is also invariant.
	
	Let $q_0$ be a quadratic differential such that there is some $f\in \Xi^{-1}\fibermap^{-1}(q_0)$ not in $\Xi^{-1}\Bclose$. Such a quadratic differential exists, by \cref{th:busemannnotdense}. By the proof of the theorem, we can assume that $V(q_0)$ is a multicurve. Furthermore, let $q$ be a quadratic differential such that $V(q)$ and $H(q)$ are the stable and unstable foliations respectively of some pseudo-Anosov element $\phi\in MCG(S)$. It is well known \cite[Expose 12]{FLP} that for any closed curve $\alpha$ we have that $\lambda^{-n}\phi^n(\alpha)$ converges to $\frac{i(\alpha,V(q))}{i(H(q),V(q))}H(q)$, where $\lambda$ is the stretch factor of $\phi$. For any foliation $F$ we have that $\Xi^{-1}M(q_0)(F)=0$ if and only if $i(V(q_0),F)=0$, where $M(q_0)$ is the minimal point defined in \cref{se:shapeoffibers}. Hence, since $H(q)$ is the unstable foliation of a pseudo-Anosov element and $V(q_0)$ is a multicurve, we have $i(V(q_0),H(q))\neq 0$, and so $f(H(q))\ge\Xi^{-1}M(q_0)(H(q))>0$. We have $\phi^{n}[f(\alpha)]_{\alpha\in\Curv}=[f(\phi^n(\alpha))]_{\alpha\in\Curv}$.	
	Taking limits and using that the functions in the Gardiner--Masur compactification are homogeneous of degree 1, we get that 
	\[
		\lim_{n\to\infty}[\phi^nf(\alpha)]_{\alpha\in\Curv}=\left[i(\alpha,V(q))f\left(\frac{H(q)}{i(V(q),H(q))}\right)\right]_{\alpha\in\Curv}=[i(\alpha,V(q))]_{\alpha\in\Curv},
	\]
	Hence, in the normalized version, $\phi^n f$ converges to $i(\cdot ,V(q))=\Xi^{-1} B(q)$, as $V(q)$ is uniquely ergodic and therefore indecomposable. That is, $B(q)$ can be approached through a sequence of elements contained in the complement of the closure of $\B_1$.
	
	Let $B(q')$ be any element in $\B_1$, where $q'$ is any quadratic differential such that $V(q')$ has one indecomposable component. The set of pseudo-Anosov foliations is dense in $\MF(S)$, so we have a sequence of quadratic differentials $(q_n)$ converging to $q'$ with $V(q_n)$ being a pseudo-Anosov foliation. Since $q'$ has one indecomposable component, the convergence is strong, and so $B(q_n)$ converges to $B(q')$. Each $B(q_n)$ can be approached through a sequence of elements contained in the complement of the closure of $\B_1$, so taking a diagonal sequence the same can be said for $B(q')$. 
\end{proof}

\begin{corollary}\label{co:nofullmeasure}
	Let $S$ be a closed surface of genus $g$ with $p$ marked points, such that $3g+p\ge5$. Then, for any finite strictly positive measure $\nu$ on the horoboundary, the set $\Bclose_1$ does not have full $\nu$-measure.
\end{corollary}
\begin{proof}
	By Theorem \ref{th:nowheredense}, the complement of $\Bclose_1$ is open and nonempty, so it must have positive $\nu$-measure.
\end{proof}

This last result tells us that the image of Miyachi's homeomorphism does not have full \(\nu\)-measure within the horoboundary for any strictly positive measure \(\nu\). However, as announced in the introduction, any attempt to extend the identity from the Thurston compactification to the horoboundary compactification to a set of full measure within the Thurston compactification results in the same problem. We restate here the result as we shall use the notation for the proof.

\noequivalentmeasuresintro*

\begin{proof}
	Assume such a $U$ exists. Choose then a basepoint $x\in \T(S)$ and let $U'=U\cap \PMF^{UE}$. For each element of $F \in U'$ the associated Hubbard--Masur quadratic differential $q_{F,x}$ satisfies $\tray{q_{F,x}}{t}\to F$ as $t\to\infty$. Hence, since $\phi$ is continuous at $F$ we have $\phi(F)=B(q_{F,x})$. That is, $\phi(U')\subset \B_1$.
	
	Let $G \in U$. The set $\PMF^{UE}$ has full $\mu$ measure, so $U'=\PMF^{UE}\cap U$ also has full measure. Hence, since the Lebesgue measure is strictly positive, $U'$ is dense within $\PMF$. Therefore $G$ can be accessed through a sequence $(F_n)\subset U'$. Hence, since $\phi$ is continuous in $G$ we have $\phi(G)=\lim \phi(F_n)$, so $\phi(U)\subset \Bclose_1$ and $\phi(U)$ can not have full $\nu$-measure.
\end{proof}

Another natural family of measures on the boundary is obtained by considering harmonic measures. Given a non-elementary measure $\mu$ on $\MCG(S)$ it is possible to define a random walk $(w_n)$ as the sequence of random variables defined by  
\[
w_n=g_0g_1g_2\ldots g_n,
\] 
where $g_i$ are independent, identically distributed random variables on $\MCG(S)$ sampled according to the distribution $\mu$. As proven by Kaimanovich and Masur in \cite[Theorem 2.2.4]{Kaimanovich}, random walks generated by a non-elementary probability measure converges almost surely in Thurston's compactification, so we can define the hitting measure $\nu$ in $\PMF$. Furthermore, the walk converges almost surely to uniquely ergodic projective foliations, so we can translate this result to the horofunction compactification in the following way.
\begin{corollary}
	Let $\mu$ be a non-elementary measure on $\MCG(S)$. Then the associated harmonic measure on the horoboundary is supported in a nowhere dense set.
\end{corollary}
\begin{proof}
	For any $x\in\T(S)$ the sequence $(w_nx)$ converges almost surely in Thurston compactification to some  $F\in \PMF^{UE}$. Hence, by \cite[Corollary 1]{Miyachi}, the sequence $(w_nx)$ converges almost surely to the Busemann point generated by a quadratic differential $q$ with $V(q)$ being a multiple of $F$. Hence, the support of the harmonic measure is contained in $\B_1$, which is nowhere dense by Theorem \ref{th:nowheredense}.
\end{proof}
 
\section{Topology of the Horoboundary}\label{se:globaltopology}
	
	In this section we make some progress towards determining the global topology of the horoboundary. We begin by showing that the minimal point $M(q)$ introduced in Proposition \ref{pr:lowerboundGM} serves as a section for the map $\fibermap$ whenever $S$ does not have a boundary. Our main goal for this section is proving the following Theorem.
	
	\begin{theorem}\label{th:globalsection}
	Let $S$ be a surface of genus $g$ with $b_m$ and $b_u$ boundaries with and without marked points respectively and $p$ interior marked points. Then, the map $\fibermap$ restricted to the boundary has a global continuous section $\vbd{\T} \to  \hbd{\T}$ if and only if at least one of the two following conditions is satisfied:
	\begin{itemize}
	 \item $b_m=b_u=0$ or
	 \item $2g+2b_m+b_u+p-\max(1-b_u,0)\le 4$.
	\end{itemize}	
	The section is given by sending the ray in the direction of $q$ to the point $M(q)$ defined before Proposition $\ref{pr:lowerboundGM}$.

	Furthermore, if the map does not admit a global section, then it does not admit any local section around some points.
	\end{theorem}

	We begin by proving the theorem for surfaces without boundary, as it is significantly easier to prove.
	
	\begin{proposition}\label{pr:continuoussection}
		Let $S$ be a surface without boundary. Then the projection map $\fibermap$ restricted to the boundary admits a global section, given by the map $M:\vbd{\T}\to \hbd{\T}$.
	\end{proposition}
	\begin{proof}
		By Proposition \ref{pr:lowerboundGM} every preimage $\fibermap^{-1}(q)$ contains $M(q)$. We have $M(q)=\Xi (i(V(q),\cdot))$, which is continuous, as the map $\Xi$ is continuous.
	\end{proof}
	
	The rest of the cases of Theorem \ref{th:globalsection} require a more careful analysis.

	\begin{proposition}\label{pr:continuoussection2}
		Let $S$ be either		
		\begin{itemize}
			\item a torus with up at most two unmarked boundaries or interior marked points,
			\item a torus with one marked boundary and one interior marked point,
			\item a sphere with one marked boundary and up to three interior marked points or
			\item a sphere with two marked boundaries and interior marked point.
		\end{itemize}
		Then the projection map $\fibermap$ restricted to the boundary admits a global section, given by the map $M:\vbd{\T}\to \hbd{\T}$.
	\end{proposition}
	\begin{proof}
		We shall build the section in the same way we built it in Proposition \ref{pr:continuoussection}, that is, sending $q$ to $M(q)$. 
		
		Our first step in the proof is seeing that if $V(q)$ contains a separating proper arc then only one of the two parts separated by the proper arc admit interior components. We shall do this by inspecting each possible case. Assume then that $V(q)$ has a separating proper arc. 
		
		If $S$ is a torus with up to two unmarked boundaries or marked points or a torus with one marked boundary and one marked point, then the separating proper arc splits the surface into a torus with a marked boundary and a sphere with a marked boundary and a marked point or unmarked boundary. The latter does not admit an interior component.
		
		If $S$ is a sphere with one marked boundary and up to three boundaries then the separating proper arc splits the surface into two spheres, both with one marked boundary, one of them with two marked points and the other one with one marked point. Again, the latter does not admit an interior component.
		
		Finally, if $S$ is a sphere with two marked boundaries and one marked point or unmarked boundary, the proper arc splits the surface into one sphere with two marked boundaries and a sphere with one marked boundary and one marked point, which again does not admit an interior component.
		
		Take then a sequence of unit quadratic differentials $(q_n)$ converging to $q$. Let $P_i$, $i\in \{1,\ldots, c\}$ be the boundary components of $V(q)$. Furthermore, denote $G$ the union of the interior components. By the first part of the proof, all the interior components are contained in the same interior part. We thus have
		\[\Xi^{-1}M(q)=\left(\sum_i \wals^q(P_i) + \wals^q( G)\right)^{1/2}.\]
		By Proposition \ref{pr:notsplittingboundary} all boundary components of $V(q)$ are contained in $V(q_n)$ for $n$ big enough, and all other boundary components of $V(q_n)$, denoted $P^n$, vanish in the limit. Denote $G^n$ the union of the interior components of $V(q_n)$. As before, each indecomposable component of $G^n$ is contained in the same interior part, so we have
		\[
			\Xi^{-1}M(q_n)=\left(\sum_i \wals^{q_n}(\alpha_i^nP_i) +\wals^{q_n}(P^n)+ \wals^{q_n}(G^n)\right)^{1/2},
		\]
		which converges to $\Xi^{-1}M(q)$.
	\end{proof}
			
	\begin{proposition}\label{pr:continuoussection3}
		Let $S$ be either
		\begin{itemize}
			\item a surface of genus at least two and at least one boundary;
			\item a torus with at least one boundary and two more boundaries or interior marked points;
			\item a torus with at least two boundaries, one being marked, and possibly interior marked points;
			\item a sphere with at least one boundary, and four more boundaries or interior marked points;
			\item a sphere with at least two boundaries, one being marked, and two interior marked points or
			\item a sphere with at least three boundaries, two being marked, and possibly interior marked points.
		\end{itemize}
		Then the projection map $\fibermap$ restricted to the boundary does not admit a local section around some points. 

	\end{proposition}
	\begin{proof}
		We shall prove this by finding a quadratic differential $q$ and sequences $(q_n^1)$ and $(q_n^2)$ converging to $q$ such that their preimages by $\fibermap$ are singletons, but such that $\fibermap^{-1}(q_n^1)$ and $\fibermap^{-1}(q_n^2)$ converge to different points in $\fibermap^{-1}(q)$. If we had a section around $q$, then its value at $q_n^1$ and $q_n^2$ would be $\fibermap^{-1}(q_n^1)$ and $\fibermap^{-1}(q_n^2)$ respectively, giving us a contradiction.
		
		In all cases the construction will be similar. For $q_n^1$ we build a foliation with a separating proper arc $P$ such that each of the parts has precisely one interior component consisting of a closed curve, which we denote $G_1$ and $G_2$. Letting the weight of the proper arc diminish to $0$ we can get a sequence of quadratic differentials $(q_n^1)$ converging to a quadratic differential $q$ such that $V(q)=G_1+G_2$. Let $F_n^1=P+nG_1+nG_2$, $A_n^1$ and $A$ the area of the Hubbard--Masur differentials $q_{F_n^1,X}$ and $q_{G_1+G_2,X}$ respectively. Denote $\frac{1}{\sqrt{A_n^1}}q_{F_n^1,X}$ as $q_n^1$. These quadratic differentials have unit area, and converge to $\frac{1}{\sqrt{A}}q_{G_1+G_2,X}$, which we denote $q$. By construction, $V(q_n^1)$ is internally indecomposable, so $\fibermap^{-1}(q_n^1)$ is a singleton, and $\Xi^{-1}\fibermap^{-1}(q_n^1)=\left\{\left(\frac{\wals^{q^1_n}(P)+n\wals^{q^1_n}(G_1)+n\wals^{q^1_n}(G_2)}{\sqrt{A_n^1}}\right)^{1/2}\right\}$.
		The sequences $\frac{P}{\sqrt{A_n^1}}$, $\frac{nG_1}{\sqrt{A_n^1}}$ and $\frac{nG_2}{\sqrt{A_n^1}}$ converge respectively to $0$, $\frac{G_1}{\sqrt{A}}$ and $\frac{G_2}{\sqrt{A}}$. Hence, by Lemma \ref{le:limith} the sequence $\fibermap^{-1}(q_n^1)$ converges to $\left\{\left(\frac{\wals^q(G_1)+\wals^q(G_2)}{\sqrt{A}}\right)^{1/2}\right\}$.
		
		For building $q_n^2$ we take a curve $\gamma$ intersecting $G_1$ and $G_2$ at $b_1$ and $b_2$ times, where $b_1,b_2\in\{1,2\}$. Denote $\tau_1$ and $\tau_2$ the Dehn twists around $G_1$ and $G_2$. Let $F_n^2=\tau_1^{2n/b_1}\tau_2^{2n/b_2}\gamma$ and $A_n^2$ the area of the Hubbard--Masur differential $q_{F_n^2,X}$. As before. Denote $\frac{1}{\sqrt{A_n^2}}q_n^2$ the quadratic differentials $\frac{1}{\sqrt{A_n^2}}q_{F_n^2,X}$. These quadratic differentials have unit area, and converge to $q$. Furthermore, each $V(q_n^2)$ is a singleton and $\Xi^{-1}\fibermap^{-1}(q_n^2)=\left\{\left(\frac{\wals^{q^2_n}((\tau_1\tau_2)^n\gamma)}{\sqrt{A_n^2}}\right)^{1/2}\right\}$. The sequence $\frac{(\tau_1\tau_2)^n\gamma}{\sqrt{A_n^2}}$ converges to $\frac{G_1+G_2}{A}$, so by \cref{le:limith} the sequence $\Xi^{-1}\fibermap^{-1}(q_n^2)$ converges to $\left\{\left(\frac{\wals^q(G_1+G_2)}{\sqrt{A}}\right)^{1/2}\right\}$, which is different than the limit of $\Xi^{-1}\fibermap^{-1}(q_n^1)$.
		
		It remains then to find such a multicurve. For genus at least two we take $P$ to be a separating proper arc such that each of the parts is of genus at least one, and $G_1$ and $G_2$ to be non contractible curves, not parallel to unmarked boundaries on each part, as shown in \cref{fi:noncontractible1}. 
		
		For the torus we take $P$ to be a separating proper arc with both endpoints in the unmarked boundary, or a marked boundary if there are no unmarked boundaries. Further, we choose the proper arc such that, after cutting along the arc, one part is a torus with one boundary. That is, every other feature of the surface lies in the other part. Then we let $G_1$ and $G_2$ be non contractible curves on each part, as shown in \cref{fi:noncontractible2}.
		
		Finally, for the sphere we let $P$ be a separating proper arc with both endpoints on an unmarked boundary, or a marked boundary if there are no boundaries without marked points. Further, we choose the arc such that each interior part has at least either a combination of two marked points or boundaries without marked points, or a boundary with marked points. Hence, each interior part supports an interior component formed by a curve, as shown in \cref{fi:noncontractible3}.
		
		 \begin{figure} \centering
			\begin{subfigure}[b]{0.30\textwidth}
				\centering
				\resizebox{\linewidth}{!}{
					\begin{tikzpicture}
						\draw[smooth]  (1.75,-1.5) to[out=120,in=-30] (0,-1) to[out=150,in=-150] (0,1)
						to[out=30,in=150] (2,1) to[out=-30,in=210] (3,1) to[out=30,in=150] (5,1)
						to (5,1) to[out=-30,in=30] (5,-1) to[out=210,in=70] (3.25,-1.5);
						\draw[smooth] (0.4,0.1) .. controls (0.8,-0.25) and (1.2,-0.25) .. (1.6,0.1);
						\draw[smooth] (0.5,0) .. controls (0.8,0.2) and (1.2,0.2) .. (1.5,0);
						\draw[smooth] (3.4,0.1) .. controls (3.8,-0.25) and (4.2,-0.25) .. (4.6,0.1);
						\draw[smooth] (3.5,0) .. controls (3.8,0.2) and (4.2,0.2) .. (4.5,0);
						\draw (1.75,-1.5) arc(180:360:0.75 and 0.2);
						\draw [dashed](1.75,-1.5) arc(180:0:0.75 and 0.2);
						
						\node [label=above:$P$] at (2.5,0.7) {};
						\draw (2.5,0.855) arc(90:190:0.25 and 2.174);
						\draw [dashed](2.5,0.855) arc(90:0:0.25 and 2.174);

						\node [label=above:$G_1$] at (1,1.1) {};
						\draw (1.0,0.15) arc(270:90:0.3 and 1.14/2);
						\draw[dashed] (1.0,0.15) arc(270:450:0.3 and 1.14/2);
						
						\node [label=above:$G_2$] at (4,1.1) {};
						\draw (4.0,0.15) arc(270:90:0.3 and 1.14/2);
						\draw[dashed] (4.0,0.15) arc(270:450:0.3 and 1.14/2);
						
						\node [label=center:$\gamma$] at (2.5,-0.1) {};
						\draw[smooth] (5,0) to[out=270,in=-30] (3,-0.5) to[out=150,in=30] (2,-0.5) to[out=210,in=270] (0,0)  to[out=90,in=150] (2,0.5) to[out=-30,in=210] (3,0.5) to[out=30,in=90] (5,0) ;
				\end{tikzpicture}}  
				\caption{}
				\label{fi:noncontractible1}
			\end{subfigure}
			\begin{subfigure}[b]{0.30\textwidth}
				\centering
				\resizebox{\linewidth}{!}{
					\begin{tikzpicture}
					\draw[smooth]  (1.75,-1.5) to[out=120,in=-30] (0,-1) to[out=150,in=-150] (0,1)
					to[out=30,in=150] (2,1) to[out=-30,in=190] (3,0.9) to[out=10,in=200] (4.5,1.3) to[out=210,in=170] (5,0) ;
					
					\draw[smooth] (5,-0.5) to[out=190,in=70] (3.25,-1.5);
					\draw[smooth] (0.4,0.1) .. controls (0.8,-0.25) and (1.2,-0.25) .. (1.6,0.1);
					\draw[smooth] (0.5,0) .. controls (0.8,0.2) and (1.2,0.2) .. (1.5,0);
					
					\draw (1.75,-1.5) arc(180:360:0.75 and 0.2);
					\draw [dashed](1.75,-1.5) arc(180:0:0.75 and 0.2);
					
					\draw (5,0) arc(90:270:0.08 and 0.25);
					\draw (5,0) arc(90:-90:0.08 and 0.25);
					
					\node[circle,fill,inner sep=1pt] at (4.5,1.3) {};
					
					\node [label=above:$P$] at (2.5,0.7) {};
					\draw (2.5,0.855) arc(90:190:0.25 and 2.174);
					\draw [dashed](2.5,0.855) arc(90:0:0.25 and 2.174);
					
					\node [label=above:$G_1$] at (1,1.1) {};
					\draw (1.0,0.15) arc(270:90:0.3 and 1.14/2);
					\draw[dashed] (1.0,0.15) arc(270:450:0.3 and 1.14/2);
					
					\node [label=above:$G_2$] at (3.8,0.9) {};
					\draw (3.8,1.07) arc(90:270:0.25 and 0.955);
					\draw[dashed] (3.8,1.07) arc(90:-90:0.25 and 0.955);
					\node[white] at (5.5,0) {};
					
					\node [label=center:$\gamma$] at (2.5,-0.3) {};
					\draw[smooth] (4.41,0.5) to [out=250, in=0] (1,-0.6) to[out=180,in=270] (0,0)  to[out=90,in=150] (2,0.5) to[out=-30,in=210] (3,0.5) to[out=30,in=200] (3.5,1);
					\draw[dashed][smooth] (3.5,1) [out=350, in=140]to (4.41,0.5);
				\end{tikzpicture}}  
				\caption{}
				\label{fi:noncontractible2}
			\end{subfigure}
			\begin{subfigure}[b]{0.30\textwidth}
				\centering
				\resizebox{\linewidth}{!}{
					\begin{tikzpicture}
					\draw[smooth] (1.75,-1.5) to[out=120,in=10] (0,-1);
					
					\draw[smooth] (0,1)	to[out=-10,in=190] (3,0.9) to[out=10,in=200] (4.5,1.3) to[out=210,in=170] (5,0) ;
					
					\draw[smooth] (5,-0.5) to[out=190,in=70] (3.25,-1.5);
					
					\draw (1.75,-1.5) arc(180:360:0.75 and 0.2);
					\draw [dashed](1.75,-1.5) arc(180:0:0.75 and 0.2);
					
					\draw (5,0) arc(90:270:0.08 and 0.25);
					\draw (5,0) arc(90:-90:0.08 and 0.25);
					
					\node[circle,fill,inner sep=1pt] at (4.5,1.3) {};
					
					\draw (0,1) arc(90:270:0.3 and 1);
					\draw [dashed](0,1) arc(90:-90:0.3 and 1);
					\node[circle,fill,inner sep=1pt] at (-0.3,0) {};
					
					\node [label=above:$P$] at (2.5,0.7) {};
					\draw (2.5,0.83) arc(90:190:0.25 and 2.15);
					\draw [dashed](2.5,0.83) arc(90:0:0.25 and 2.15);
					
					\node [label=above:$G_1$] at (1,0.7) {};
					\draw (1.0,0.84) arc(90:270:0.3 and 0.905);
					\draw [dashed](1.0,0.84) arc(90:-90:0.3 and 0.905);
					
					\node [label=above:$G_2$] at (3.8,0.9) {};
					\draw (3.8,1.07) arc(90:270:0.25 and 0.955);
					\draw[dashed] (3.8,1.07) arc(90:-90:0.25 and 0.955);

					\node [label=center:$\gamma$] at (2.5,-0.3) {};
					\draw[smooth] (4.41,0.5) to [out=250, in=0] (-0.26,-0.5);
					\draw[smooth] (-0.26,0.5) to[out=0,in=210] (3,0.7) to[out=30,in=200] (3.5,1);
					\draw[dashed][smooth] (3.5,1) [out=350, in=140]to (4.41,0.5);
					
					\node[white] at (5.5,1) {};

				\end{tikzpicture}}  
				\caption{}
				\label{fi:noncontractible3}
			\end{subfigure}
		\caption{Curves chosen in the proof of Proposition \ref{pr:continuoussection3}}
		\end{figure}
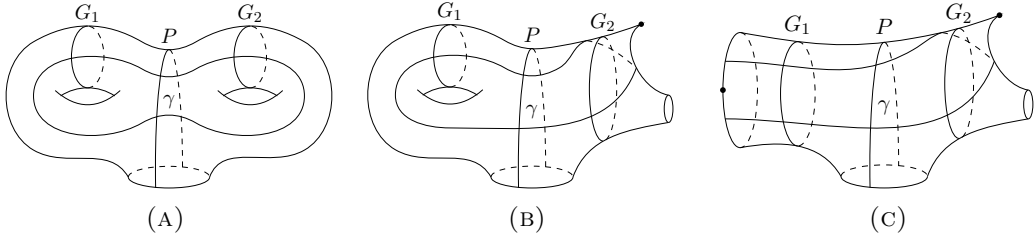

	\end{proof}

	\begin{proof}[Proof of Theorem \ref{th:globalsection}]
		This is a combination of the results from Propositions \ref{pr:continuoussection}, \ref{pr:continuoussection2} and \ref{pr:continuoussection3}.
	\end{proof}
		
	By Proposition \ref{pr:finslerconnected} we know that the horoboundary is connected whenever the real dimension of Teichmüller space is at least 2. In the following result we go a bit further, by showing that it is actually path connected.
	
	\begin{proof}[Proof of Theorem \ref{th:teichconnected}]
		
		Let $x,y\in \hbd{\T(S)}$. If $S$ does not have boundary then $\fibermap$ has a global section, so we can lift any path between $\fibermap(x)$ and $\fibermap(y)$ to a path between $M(\fibermap(x))$ and $M(\fibermap(y))$. Then, since $\fibermap^{-1}(x)$ and $\fibermap^{-1}(y)$ are path connected, we can connect $x$ to $M(\fibermap(x))$ and $y$ to $M(\fibermap(y))$ via paths.

		If $S$ has boundary we might have to be a bit more careful, as we might not have a global section. However, as we shall see, we can take a path $q_t$ between $\Pi(x)$ and $\Pi(y)$ such that $B(q_t)$ has finitely many discontinuities. Then, since each of the preimages is path connected these discontinuities can be fixed by using paths in the fibers, so we will have a path between $x$ and $y$.
		
		Choose a boundary component of $S$, denote $b$ a curve parallel to that boundary and let $F_x=V(\Pi(x))$. If $F_x$ contains $b$ then all the expressions of the form $(1-t)F_x+tb$ with $t\in[0,1]$ correspond to foliations on $S$, which we denote $F_t$. Denote $q_t$ the unit area quadratic differential such that $V(q_t)$ is a multiple of $F_t$. This defines a continuous path joining $\fibermap(x)$ and the unit area quadratic differential associated to a multiple of $b$. Let $V_i$ be the vertical components of $F_x$ that are not $b$, and let $w_0$ be the weight of $b$ in $F_x$. Then, $B(q_t)^2=\frac{1}{\sqrt{\operatorname{Area}(q_{F_t,X})}}\left((1-t)\sum \wals^{q_t} (V_i)+(t+(1-t)w_0)\wals^{q_t}(b)\right)$, which gives a continuous path from $B(q_0)\in\fibermap^{-1}\fibermap(x)$ to $B(q_1)\in \fibermap^{-1}(q_1)$. If $F_x$ does not contain $b$, but $b$ can be added to the foliation then we proceed just as before. Hence, if both $x$ and $y$ result in foliations where $b$ can be added, we create a path by concatenating the paths between $x$, the Busemann point in $\fibermap^{-1}\fibermap(x)$, the Busemann point associated to $b$, the Busemann point in $\fibermap^{-1}\fibermap(y)$ and $y$.
		
		If $b$ cannot be added to the foliation $F_x$ then there must be some set $P$ of proper arcs in $F_x$ incident to the boundary component associated to $b$. Let $F'_x$ be the foliation $F_x$ without the proper arcs $P$ and assume $F'_x$ is nonempty. Denote $F_t$ the foliations $(1-t)P+(1+t)F'_x$, $t\in [0,1]$, and $q_t$ the unit area quadratic differentials such that $V(q_t)$ is a multiple of $F_t$. Denoting $V_i$ the vertical components of $F'_x$, and $P_j$ the proper arcs incident to the boundary component associated to $b$, we have $B(q_t)^2=\frac{1}{\sqrt{\operatorname{Area}(q_{F_t,X})}}\left((1-t)\sum_j\wals^{q_t}(P_j)+(1+t)\sum\wals^{q_t}(V_i)\right)$ for $t<1$, which is continuous. Furthermore, $\lim_{t\to 1} B(q_t)\in \fibermap^{-1}(q_1)$. Hence, we can concatenate a paths between $x$, the Busemann point in $\fibermap^{-1}\fibermap(x)$, the limit $\lim_{t\to 1} B(q_t)$, the Busemann point $B(q_1)$ and Busemann point associated to $b$.
		
		If $F'_x$ is empty we want to add some other components to $F_x$. If it admits some other component $k$ then we repeat the previous reasoning with $F_t=(1-\frac{t}{2})F_x+\frac{t}{2}k$, which does not result in any discontinuity. If $F_x$ does not admit any other component then there must be at least 2 proper arcs incident to the boundary component associated to $b$, so we choose one of them, denoted $p$, and repeat the previous reasoning with $F_t=(1-t)F_x+tp$, which does not result in any discontinuity. Finally, we concatenate this last path with the previous paths.
	\end{proof}

\section{Formulas for limits of extremal lengths}\label{se:formulas}

	We finish by reframing the bounds we got for the elements of $\Xi^{-1}\fibermap^{-1}(q)$ as results regarding limits of extremal lengths, getting in this way some extensions of \cite[Theorem 1]{Walsh}.
	
	\begin{proposition}\label{pr:boundslimits}
		Let $F$ be a measured foliation, $(q_n)$ be a sequence of unit area quadratic differentials converging to a quadratic differential $q$ and $(t_n)$ be a sequence of real numbers converging to infinity. Then,
		\[\left(\Xi^{-1}M(q)\right)^2\le \liminf_{n\to\infty} e^{-2t_n} \ext_\tray{q_n}{t_n}(F)\le\limsup_{n\to\infty} e^{-2t_n} \ext_\tray{q_n}{t_n}(F)\le
		\left(\Xi^{-1}B(q)\right)^2
		\]
	\end{proposition}
	\begin{proof}
		Take a subsequence such that $e^{-2t_n} \ext_\tray{q_n}{t_n}(F)$ converges to the liminf. Furthermore, take a subsequence such that $\tray{q_n}{t_n}$ converge to a point $\xi\in \fibermap^{-1}(q)$. By Proposition \ref{pr:lowerboundGM} we have $(\Xi^{-1}M(q))^2\le \xi^2$. Since $e^{-2t_n} \ext_\tray{q_n}{t_n}(F)$ converges to $\xi^2(F)$ we have the lower bound. For the upper bound we repeat the process taking the limsup and using Proposition \ref{pr:upperboundGM}.
	\end{proof}
	
	By noting that $\Xi^{-1}M(q)(F)$ and $\Xi^{-1}B(q)(F)$ evaluate to $0$ if and only if $i(V(q),F)=0$, we get the following corollary, which has also been proven for surfaces without boundary by Liu--Shi in \cite[Corollary 3.11]{LiuShi}.

	\begin{corollary}\label{pr:relationqdtoGM}
		Let $(q_n)$ be a sequence of unit area quadratic differentials converging to a quadratic differential $q$, and $(t_n)$ be a sequence of real numbers converging to infinity. Then, 
		\[\liminf_{n\to\infty} e^{-2t_n} \ext_\tray{q_n}{t_n}(F)=0 \iff i(V(q),F)=0.\]
	\end{corollary}

	Proposition \ref{pr:boundslimits} can be strengthened slightly in the following manner.
	
	\begin{proposition}\label{pr:lowerboundrefinedlimit}
		Let $(q_n)$ be a sequence of unit area quadratic differentials converging to a quadratic differential $q$. Furthermore, denote $V_i^n$ the indecomposable components of $q_n$. If the vertical components can be reordered so that for each $i$ we have that $V_i^n$ converges to a foliation $V_i$, then
		\[\liminf_{n\to\infty} e^{-2t_n} \ext_\tray{q_n}{t_n}(F)\ge
		\sum_i \wals^q(V_i).\]
	\end{proposition}
	\begin{proof}
		Take a sequence such that the limit is equal to the liminf, and such that we have convergence in the Gardiner--Masur compactification. Let $\xi$ be the limit in the horofunction compactification. By
		Lemma \ref{le:walshlowerbound} we have $e^{-2t_n} \ext_\tray{q_n}{t_n}(F)\ge \left(\Xi^{-1}B(q_n)\right)^2$, and by \cref{co:walshbusemanshape} we have $\left(\Xi^{-1}B(q_n)\right)^2=\sum_i \wals^{q_n}(V_i^n)$. Hence, by Lemma \ref{le:limith}, taking limits on both sides we get the proposition.
	\end{proof}	

	If we have strong convergence the upper bound from Proposition \ref{pr:boundslimits} and the lower bound from Proposition \ref{pr:lowerboundrefinedlimit} coincide, so adding Walsh's formula for the Busemann points \cite[Theorem 1]{Walsh} we have a proof of Theorem \ref{th:limits}.
	
	Finally, the path connectedness of the fibers can be translated to the following result.
	
	\begin{proposition}\label{pr:pathGM}
		Let $(q_n)$ be a sequence of unit quadratic differentials converging to $q$, and $(t_n)$ be a sequence of times converging to infinity. Further, for any $F\in \MF$ denote $L(F):=\liminf_{n\to\infty}\ext_\tray{q_n}{t_n}(F)$. Then, for any $s\in [L(F), \E^2_q(F)]$ there is a subsequence of $q_{n_k^s}$ and a sequence $(t_k^s)$ of times such that, for any $G\in \MF$ the limit \[\lim_{k\to\infty} e^{-2t_{k}^s} \ext_\tray{q_{n_k^s}}{t_k^s}(G)\] is defined, and if $G=F$ it has value $s$.
	\end{proposition}
	\begin{proof}
		We can take a subsequence such that $\lim_{n\to\infty}\ext_\tray{q_n}{t_n}(F)$ converges to the liminf, and a further subsequence such that we have convergence in the Gardiner--Masur compactification to a point $\Xi^{-1}\xi\in\Xi^{-1}\fibermap^{-1}(q)$. By Theorem \ref{pr:pathconnected} we have a path between $\xi$ and $B(q)$ contained in $\fibermap^{-1}(q)$, and hence a path $\gamma$ between $\Xi^{-1}\xi$ and $\Xi^{-1}B(q)$ contained in $\Xi^{-1}\fibermap^{-1}(q)$. By continuity there is a point in that path such that $\gamma_t(F)=\sqrt{s}$, and by the way we constructed $\gamma_t$, it is reached by taking a subsequence of $(q_{n_k^s})$ and a sequence $(t_k^s)$ of times converging to infinity. Finally, since $\gamma_t$ is a point in the Gardiner--Masur compactification approached by $\tray{q_{n_k^s}^s}{t_{k}^s}$, the value of $\gamma_t(G)^2$ is equal to the limit from the proposition.
	\end{proof}
		
\bibliographystyle{gtart}
\bibliography{horoboundary}{}	
\end{document}